\documentclass[12pt]{article}%

\usepackage{amsmath}
\usepackage{amssymb,amsfonts,amsthm}
\usepackage{fancyhdr}
\usepackage{tabularx}
\usepackage{dsfont}
\usepackage{mathtools}
\usepackage{authblk}
\usepackage{xr}
\usepackage[colorlinks = True]{hyperref}
\hypersetup{
    colorlinks = true,
    linkcolor = blue,
    anchorcolor = blue,
    citecolor = blue,
    filecolor = blue,
    urlcolor = blue
    }
\usepackage[a4paper, top=2.5cm, bottom=2.5cm, left=1.5cm, right=1.5cm]%
{geometry}
\usepackage{times}
\usepackage[table,xcdraw]{xcolor}
\usepackage[capitalize, nameinlink]{cleveref}
\usepackage{changepage}
\usepackage{enumitem}
\usepackage{mathrsfs}
\usepackage{tcolorbox}
\usepackage{stackrel,amssymb}
\usepackage{graphicx}

\usepackage{appendix}

\usepackage[backend=biber,
style=alphabetic,
minalphanames=3,
maxalphanames=3,
maxnames=50
]{biblatex}
\def\hmath$#1${\texorpdfstring{{\rmfamily\textit{#1}}}{#1}}

\newtheorem{theorem}{Theorem}
\newtheorem{corollary}[theorem]{Corollary}
\newtheorem{observation}[theorem]{Observation}
\newtheorem{conjecture}[theorem]{Conjecture}
\newtheorem{definition}[theorem]{Definition}
\newtheorem{lemma}[theorem]{Lemma}
\newtheorem{remark}[theorem]{Remark}
\newtheorem{example}[theorem]{Example}
\newtheorem{proposition}[theorem]{Proposition}

\newtheorem{problem}[theorem]{Problem}
\newtheorem{claim}[theorem]{Claim}
\counterwithin{theorem}{section}

\newcommand{\indicator}{\mathds{1}}
\newcommand{\hypercubegraph}{\mathsf{RAG}(n, \{\pm 1\}^d, p,\sigma)}
\newcommand{\orthogonalgroup}{\mathcal{O}}

\newcommand{\dsphere}{\mathbb{S}^{d-1}}
\newcommand{\wishart}{\mathcal{W}}
\newcommand{\sm}{\mathsf{sm}}

\newcommand{\PLSG}{\mathsf{PLSG}(n, \Omega, \mathcal{D}, \sigma, p)}
\newcommand{\plsg}{\mathsf{PLSG}}
\newcommand{\RAG}{\mathsf{RAG}}
\newcommand{\RGG}{\mathsf{RGG}}
\newcommand{\RRAG}{\mathsf{RRAG}}
\newcommand{\LRAG}{\mathsf{LRAG}}

\newcommand{\rag}{\RAG(n, \Group, p,\sigma)}

\newcommand{\ergraph}{\mathsf{G}(n,p)}
\newcommand{\sbm}{\mathsf{SBM}}
\newcommand{\hypercube}{\{\pm 1\}^d}
\newcommand{\ergraphhalf}{\mathsf{G}(n, 1/2)}
\newcommand{\unif}{\mathsf{Unif}}
\newcommand{\Bernoulli}{\mathsf{Bernoulli}}
\newcommand{\postuniform}{\mathsf{PostU}}
\newcommand{\anteuniform}{\mathsf{AnteU}}
\newcommand{\polylog}{\mathsf{polylog}}

\newcommand{\bfx}{\mathbf{x}}
\newcommand{\bfv}{\mathbf{v}}
\newcommand{\bfy}{\mathbf{y}}
\newcommand{\bfz}{\mathbf{z}}
\newcommand{\bfg}{\mathbf{g}}
\newcommand{\bfh}{\mathbf{h}}

\newcommand{\bfa}{\mathbf{a}}
\newcommand{\bfb}{\mathbf{b}}
\newcommand{\bfA}{\mathbf{A}}
\newcommand{\bfB}{\mathbf{B}}
\newcommand{\bfX}{\mathbf{X}}

\newcommand{\expect}{{\mathbf{E}}}
\newcommand{\distribution}{{\mathcal{D}}}
\newcommand{\fluctuation}{{\mathsf{Fl}}}
\newcommand{\threshold}{{\mathbf{T}}}
\newcommand{\doublethreshold}{\mathbf{D}}
\newcommand{\maj}{{\mathbf{Maj}}}
\newcommand{\stab}{{\mathbf{Stab}}}
\newcommand{\hnmaj}{{\mathbf{HNMaj}}}
\newcommand{\weight}[2]{\mathcal{W}^{#1}({#2})}
\newcommand{\Var}{{\mathbf{Var}}}
\newcommand{\Cov}{{\mathbf{Cov}}}

\newcommand{\TV}{{\mathsf{TV}}}
\newcommand{\KL}{{\mathsf{KL}}}
\newcommand{\prob}{{\mathbf{P}}}
\newcommand\norm[1]{\left\lVert#1\right\rVert}

\newcommand{\ER}{Erd{\H o}s-R\'enyi }

\newcommand{\B}{\Big}

\newcommand{\Group}{\mathcal{G}}
\newcommand{\cycle}{\mathcal{C}}
\newcommand{\Orbit}{\mathcal{O}}

\newcommand{\quadand}{\quad \text{and} \quad}

\newcommand{\Bgn}{\Big\|}
\newcommand{\bggn}{\bigg\|}

\usepackage[titles]{tocloft}
\title{Random Algebraic Graphs and \\ Their Convergence to \ER}

\author{Kiril Bangachev\thanks{Dept. of EECS, MIT. \texttt{kirilb@mit.edu} } \quad Guy Bresler\thanks{Dept. of EECS, MIT. \texttt{guy@mit.edu}. Supported by NSF Career Award CCF-1940205.}}

\begin{document}

\pagenumbering{roman}

\maketitle

\begin{abstract}
    We introduce a new  probabilistic model of latent space graphs, which we call \textit{random algebraic graphs}. A (right) random algebraic graph is defined by a group $\mathcal{\Group}$ with a well-defined ``uniform'' distribution $\mathcal{D}$ over it and a measurable function $\sigma:\Group \longrightarrow [0,1]$ with expectation $p$ over $\distribution,$ satisfying $\sigma(\bfg) = \sigma(\bfg^{-1}),$ which we call a ``connection''.  The random graph $\rag$ with vertex set $[n]$ is formed as follows. First, $n$ independent latent vectors $\bfx_1, \bfx_2, \ldots, \bfx_n$ are sampled according to $\distribution.$ Then, every two vertices $i,j$ are connected with probability $\sigma(\bfx_i\bfx_j^{-1}).$ The random algebraic model
    captures random geometric graphs with latent space the unit sphere and the hypercube, certain regimes of the stochastic block model, and random subgraphs of Cayley graphs.
    
    The main question of interest to the current paper is: for what parameters is a random algebraic graph $\rag$ statistically and/or computationally distinguishable from an \ER  random graph $\ergraph$? Our results fall into two main categories.
    \begin{enumerate}
        \item \textit{Geometric:} We mostly focus on the hypercube case $\Group =\hypercube,$ where we use Fourier-analytic tools. 
    We expect that some of the insights will also transfer to the sphere. For hard thresholds, we match  \cite{Liu2022STOC} for $p = \omega(1/n)$ and for connections that are $\frac{1}{r\sqrt{d}}$-Lipschitz we extend the results of \cite{Liu2021APV} when $d = \Omega(n\log n)$ to the non-monotone setting. We also study other connections such as indicators of interval unions and low-degree polynomials. One novel phenomenon is that when $\sigma$ is even, that is $\sigma(\bfx) = \sigma(-\bfx),$ indistinguishability from $\ergraph$ occurs at $d  = \Theta (n^{3/2})$ instead of $d = \Theta(n^3),$ the latter condition (up to problem-dependent hidden parameters) appearing in most previous results.
        \item \textit{Algebraic:} We provide evidence for the following exponential statistical-computational gap. Consider any finite group $\Group$ and let $A\subseteq \Group$ be a 
        set of elements formed by including each set of the form $\{\bfg, \bfg^{-1}\}$ independently with probability $p = 1/2.$ Let $\Gamma_n(\Group,A)$ be the distribution of random graphs formed by taking a uniformly random induced subgraph of size $n$ of the Cayley graph $\Gamma(\Group,A).$ Then,  $\Gamma_n(\Group, A)$ and $\ergraphhalf$ are statistically indistinguishable with high probability over $A$ if and only if $\log |\Group| \gtrsim n.$ However, low-degree polynomial tests fail to distinguish $\Gamma_n(\Group, A)$ and $\ergraphhalf$ with high probability over $A$ when $\log |\Group| = \log^{\Omega(1)}n.$
    \end{enumerate}
    
    En-route, we also obtain two novel, to the best of our knowledge, probabilistic results that might be of independent interest. First, we give a nearly sharp bound on the moments of elementary symmetric polynomials of independent Rademacher variables in certain regimes (\cref{thm:introelemntarysymmetric}).
    Second, we show that the difference of two independent $n\times n$ Wishart matrices with parameter $d$ converges to an appropriately scaled $n\times n$ GOE matrix in KL when $d = \omega(n^2).$ This is polynomially smaller than the $d = \omega(n^3)$ required for the convergence of a single $n\times n$ Wishart matrix with parameter $d$ to a GOE matrix, occurring at $d = \omega(n^3)$ \cite{Bubeck14RGG,Jiang2013ApproximationOR}.
    
\end{abstract}

\setcounter{tocdepth}{2}
{
  \hypersetup{linkcolor=black}
  \tableofcontents
}
\newpage
\pagenumbering{arabic}
\section{Introduction}
Latent space random graphs are a family of random graph models in which edge-formation depends on hidden (latent) vectors associated to nodes. One of many examples of latent space graphs---appearing in \cite{Smith19}, for example---is that of social networks, which form based on unobserved features (such as age, occupation, geographic location, and others). One reason these models are studied in the literature is that, unlike the more common \ER random graphs, latent space graphs capture and explain real-world phenomena such as ``the friend of my friend is also (likely) a friend of mine'' \cite{Smith19}.  

A mathematical model that captures this structure is given by a probability distribution $\distribution$ over some latent space $\Omega,$ an integer $n$, and a \textit{connection} function $\sigma:\Omega\times \Omega\longrightarrow [0,1]$ such that $\sigma(\bfx,\bfy) = \sigma(\bfy,\bfx)$ a.s.\  with respect to $\distribution.$ Denote
$p = \expect_{\bfx, \bfy\sim_{iid}\distribution}[\sigma(\bfx,\bfy)] \in [0,1].$
The probabilistic latent space graph $\PLSG$ is defined as follows \cite{Liu2021APV}. 
For an adjacency matrix $\bfA = (\bfA_{i,j})_{i,j\in [n]},$
\begin{align}
\prob(G = \bfA) = 
\expect_{\bfx_1, \bfx_2, \ldots, \bfx_n \sim_{iid} \distribution}\Bigg[
\prod_{1\le i < j\le n} \sigma(\bfx_i,\bfx_j)^{\bfA_{i,j}}
(1-\sigma(\bfx_i,\bfx_j))^{1-\bfA_{i,j}}\Bigg].
\end{align}
In words, for each node $i\in [n] = \{1,2,\ldots, n\},$ an independent latent vector $\bfx_i$ is drawn from $\Omega$ according to $\distribution$ and then for each pair of vertices $i$ and $j$ an edge is drawn independently with probability $\sigma(\bfx_i,\bfx_j).$  

Note that if $\sigma$ is equal to $p$ a.s., this is just the \ER distribution $\ergraph.$ More interesting examples occur when there is richer structure in $(\Omega, \mathcal{D},\sigma).$

The focus of this paper is the case when $\Omega$ is a group and $\sigma$ is ``compatible'' with the group structure as in \cref{def:rag}. Before delving into this novel algebraic setting, we provide some context and motivation for our work by describing the setting with a geometric structure, which is widely studied in the literature. 

Suppose that $\Omega$ is a metric space (most commonly $\Omega \subseteq \mathbb{R}^d$ with the induced $\| \cdot\|_2$ metric) and $\sigma(\bfx, \bfy)$ depends only on the 
distance between $\bfx$ and $\bfy.$ Such graphs are called \textit{random geometric graphs} and we write $\RGG$ instead of $\mathsf{PLSG}.$ In practice, random geometric graphs have found applications in 
wireless networks \cite{haenggi_2012}, consensus dynamics \cite{ESTRADA201620}, and protein-protein interactions \cite{higham08} among others (see \cite{penrose03,duchemin22} for more applications). On the more theoretical side, random geometric graphs are, by construction, random graph models with correlated edges which provide an interesting and fruitful parallel theory to the better understood \ER model. Khot, Tulsiani, and Worah use random geometric graphs on the sphere to provide approximation lower-bounds for constraint satisfaction problems \cite{Khot14}, and Liu et al. recently showed that random geometric graphs are high-dimensional expanders in certain regimes \cite{Liu22Expander}. Finally, the study of random geometric graphs has catalyzed the development of other areas in probability such as the convergence of Wishart matrices to GOE matrices \cite{Bubeck15ntropicCLT,Bubeck14RGG, Brennan21DeFinetti}, which in turn has applications such as average-case reductions between statistical problems \cite{Brennan19Reductions}.

Starting with \cite{Devroye11,Bubeck14RGG}, the high-dimensional setting (in which $\Omega\subseteq \mathbb{R}^d$ and $d$ grows with $n$) has gained considerable attention in recent years \cite{Brennan21DeFinetti,Liu2021APV,Liu2022STOC,Liu22Expander,Mikulincer20,Pieters22CommunityDetection,RaczRichet2019SmoothTransition,Liu21PhaseTransition,Brennan19PhaseTransition, Brennan22AnisotropicRGG}. The overarching direction of study in most of these papers and also our current work is based on the following observation, first made in \cite{Devroye11}. When $d\longrightarrow +\infty$ for a fixed $n,$ random geometric graphs typically lose their geometric structure and become indistinguishable from \ER graphs of the same expected density. This motivates the following question. How large does the dimension $d$ need to be so that edges become independent?

More concretely, we view each of $\Omega, \distribution, p,\sigma$ as an implicit sequence indexed by $n$ and take $n\longrightarrow+\infty.$ This gives rise to the following hypothesis testing problem for latent space graphs:
\begin{equation}
    \label{eq:hypoithesistesting}
    H_0: G\sim \ergraph\quad  \text{versus}\quad  H_1:G\sim \PLSG. 
\end{equation}
Associated to these hypotheses are (at least) two different questions:
\begin{enumerate}
    \item \textit{Statistical:} 
    First, when is there a consistent test?
    To this end, we aim to characterize the parameter regimes in which the total variation between the two distributions tends to zero or instead to one. 
    \item \textit{Computational:} Second, we can ask for a computationally efficient test. In particular, when does there exist a polynomial-time test solving \cref{eq:hypoithesistesting} with high probability?
\end{enumerate}

The literature has been, thus far, predominantly interested in geometric models for which there has been no evidence that the answers to the statistical and computational question are different. In particular, an efficient test (counting signed subgraphs, which we will see in \cref{sec:detection}) always matches or nearly matches the statistical threshold. In \cref{sec:polydetection}, we give an example of a model with evidence of an \emph{exponential} gap between the statistical and computational detection thresholds. Evidence for the gap is provided in the increasingly popular framework of low-degree polynomial tests (see \cite{Schramm_2022}, for example).

To evince this gap (as well as numerous other new phenomena depending on the choice of $\sigma,$ see \cref{sec:introhypercube}), we first focus on the case of the Boolean hypercube\footnote{Throughout the entire paper, when we say \textit{hypercube}, we mean the Boolean hypercube $\hypercube$ rather than $[0,1]^d.$ This distinction is very important as $[0,1]^d$ with its $\norm{\cdot}_2$ metric is not a homogeneous metric space, which makes its behaviour substantially different from the behaviour of the unit sphere and Boolean hypercube. Specifically, recent work \cite{Erba20} suggests that random geometric graphs over $[0,1]^d$ do not necessarily converge to \ER.}, $\Omega = \{\pm 1\}^d\subseteq \mathbb{R}^d.$ Most previous papers in the high-dimensional setting study the unit sphere or Gauss space. The advantage of the hypercube model is that it possesses a very simple algebraic structure --- it is a product of $d$ groups of order 2 --- in addition to its geometric structure. The underlying group structure facilitates the use of Fourier-analytic tools which lead to our results for much more general connections $\sigma$ than previously considered.

Our results for the hypercube motivate us to analyze more general latent space graphs for which $\Omega$ is a group. We call such graphs, defined in \cref{sec:grouptheorymodel}, \textit{random algebraic graphs}. It turns out that a natural condition in this setting is that $\sigma(\bfx,\bfy)$ depend only on $\bfx\bfy^{-1}$ (or, alternatively, on $\bfx^{-1}\bfy$). As we will see in \cref{sec:grouptheorymodel}, random algebraic graphs turn out to be extremely expressive. 

\begin{figure}[!htb]
    \centering
    \includegraphics[width = 0.6\linewidth]{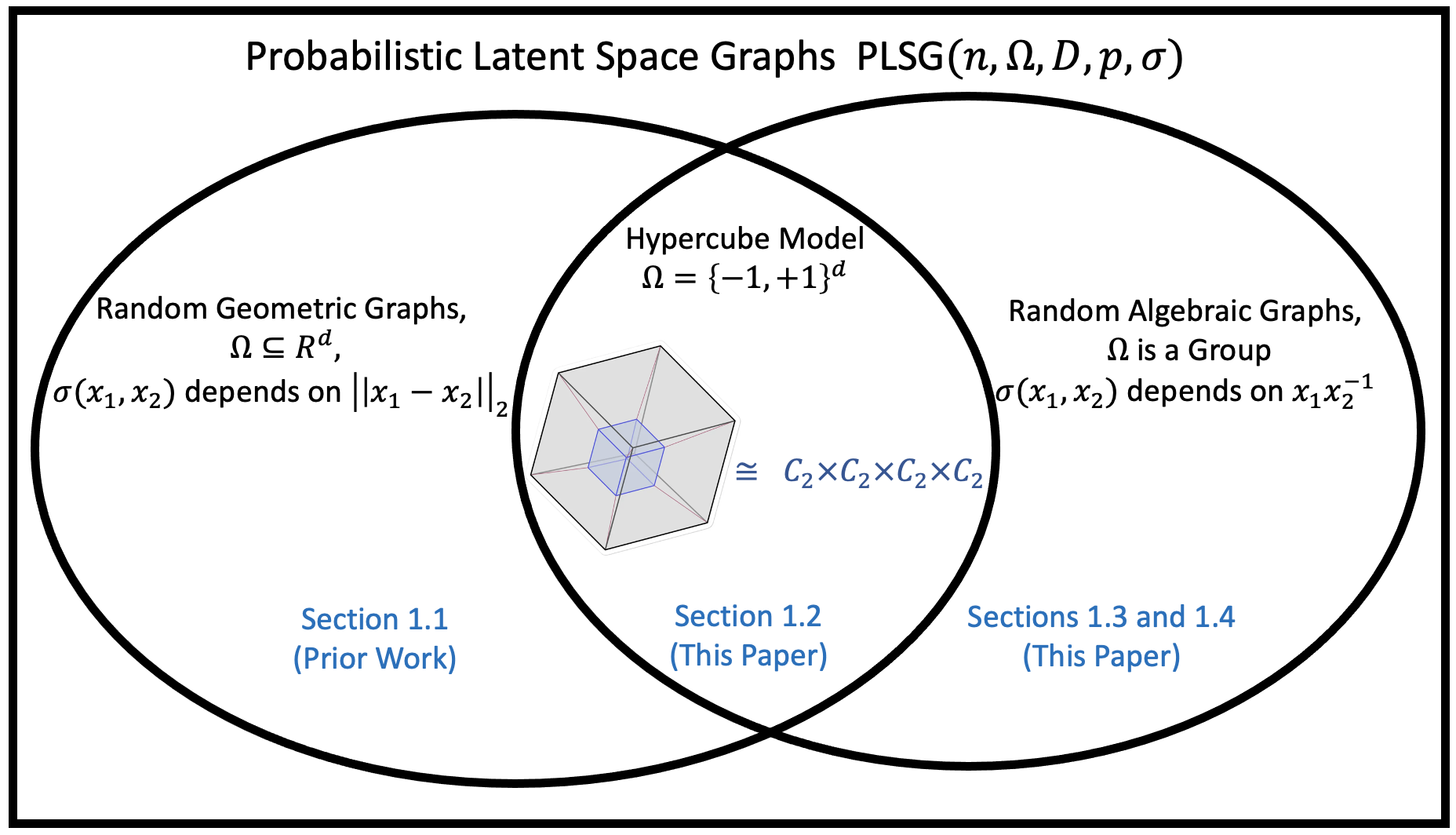}
    \caption{Different probabilistic latent space graph models depending on the underlying structure of $\Omega.$ Denoted are the sections in the introduction discussing the respective models.}
    \label{fig:plsg_map}
\end{figure}

The hypercube model can be viewed as a bridge between random algebraic graphs and random geometric graphs. Of course, one needs to be careful when interpreting results for the hypercube model since they can have their origin both in the underlying geometry or underlying algebra. In \cref{sec:discussion}, we discuss the dependence of our results on geometry and algebra.

\subsection{Prior Work}
\label{sec:previouswork}
\paragraph{1. Hard Thresholds on The Unit Sphere and Gauss Space.}A majority of the papers on high-dimensional random geometric graphs deal with hard thresholds on the unit sphere: the latent space is $\Omega = \dsphere$ and $\distribution$ is the uniform distribution $\unif$ over it. The connection $\sigma$ is given by $\threshold_p(\bfx, \bfy) := \indicator[\langle \bfx, \bfy\rangle\ge \tau_{p,d}],$ where $\tau_{p,d}$ is chosen so that $\expect [\threshold_{p,d}(\bfx, \bfy)] = p.$ For example, $\tau_{1/2,d} = 0.$ 
In \cite{Devroye11}, the authors initiated this line of work by showing that when $d = \exp(\Omega(n^2)),$ the two graph models are indistinguishable. Bubeck et.\ al.\ make an exponential improvement to $d = {\omega}(n^3)$ in \cite{Bubeck14RGG}.  Implicitly, the authors of \cite{Bubeck14RGG} also showed that the signed-triangle count statistic (which can be computed in polynomial time) distinguishes $\ergraph$ and $\RGG(n, \dsphere, \unif,p, \threshold_p)$ when $d = \Tilde{O}(n^3p^3)$ (explicitly, the calculation appears in \cite{Liu2022STOC}, for example). It is believed that the the signed triangle statistic gives the tight bound.

\begin{conjecture}[Implicit in \cite{Bubeck14RGG,Brennan19PhaseTransition,Liu2022STOC} and others]
\label{conj:sphericalhardthresholds}
$$\TV\Big(\ergraph,\RGG(n, \dsphere, \unif, p, \threshold_p)\Big) = o_n(1)$$ when $d = \tilde{\Omega}(n^3p^3)$ for $p = \Omega(\frac{1}{n}).$ 
\end{conjecture}

The current progress towards this conjecture is the following. In the sparse regime, when $p = \Theta(\frac{1}{n}),$ Liu et. al. resolve the conjecture in \cite{Liu2022STOC} by showing indistinguishability when $d = \Omega(\log^{36}n).$  When $p \le \frac{1}{2}$ and $p = \omega(\frac{1}{n}),$ the same paper shows that $\TV\Big(\ergraph,\RGG(n, \dsphere, \unif, p, \threshold_p)\Big) = o(1)$ holds when $d = \tilde{\Omega}(n^3p^2),$ missing a factor $p$ relative to \cref{conj:sphericalhardthresholds}.

Several papers, including \cite{Bubeck14RGG}, also study the related model when $(\dsphere, \unif)$ is replaced by\linebreak  $(\mathbb{R}^d,\mathcal{N}(0,I_d)^{\otimes n}).$ The two models are nearly equivalent since if $\bfx'\sim \mathcal{N}(0,I_d),$ then \linebreak $\bfx'\times \norm{\bfx'}^{-1}_2\sim \unif(\dsphere)$ and $\norm{\bfx'}^2_2$ is extremely well concentrated around $d.$
We will say more about the significance of the Gauss space model when discussing Wishart matrices momentarily.

\paragraph{2. Wishart Matrices.} For $X\sim \mathcal{N}(0,I_d)^{\otimes n}\in \mathbb{R}^{n\times d},$ denote by $\wishart(n,d),$ the law of the Wishart matrix $XX^T.$ In \cite{Bubeck14RGG}, the authors prove and use the following result. A GOE matrix $M(n)$ is a symmetric $n\times n$ matrix with random Gaussian entries, which are independent on the main diagonal and above. The variance of each entry on the main diagonal is $2$ and above the diagonal it is 1.  

\begin{theorem}
[\cite{Bubeck14RGG,Jiang2013ApproximationOR}]
\label{thm:wisharttogoe}
When $d = \omega(n^3),$
$\displaystyle
\TV\Big(\wishart(n,d), \sqrt{d}M(n) + I_n\Big) = o(1).
$
\end{theorem}

This result shows that when $d$ is sufficiently large, $\wishart(n,d)$ is, in total variation, a Gaussian matrix with i.i.d. entries. In the context of random geometric graphs, this implies that inner products of latent vectors are independent (in total variation) and, thus, the corresponding random geometric graph has independent edges. The result \cref{thm:wisharttogoe} has been generalized in several ways since. In \cite{Bubeck15ntropicCLT}, the authors obtain similar results when the Gaussian density used in the construction of the Wishart ensemble is replaced with a log-concave density. In \cite{RaczRichet2019SmoothTransition}, the authors show that the transition between a Wishart ensemble and GOE is ``smooth'' by calculating $
\TV\Big(\wishart(n,d), \sqrt{d}M(n) + I_n\Big) $ when $\lim_{n\longrightarrow +\infty}d/n^3 = c$ for a fixed constant $c.$ Finally, in \cite{Brennan21DeFinetti}, the authors consider the convergence of a Wishart matrix with hidden entries to a GOE with hidden entries. They consider both structured and random ``masks'' hiding the entries. A special case of one of their results for bipartite masks, which we will use later on, is the following. Let $\mathcal{M}$ be an $n\times n$ matrix split into four $\frac{n}{2}\times \frac{n}{2}$ blocks, where the upper-left and lower-right ones contain only zeros and lower-left and upper-right contain only ones. Let $\odot$ be the usual Schur (entrywise) product. 

\begin{theorem}[\cite{Brennan21DeFinetti}]
\label{thm:bipartitemaskwishart}
When $d = {\omega}(n^2),$
$\displaystyle
\TV\Big(\mathcal{M}\odot \wishart(n,d), \mathcal{M}\odot(\sqrt{d}M(n) + I_n)\Big) = o_n(1).
$
\end{theorem}

We also mention that Ch'etelat and Wells establish a sequence of phase transitions of the density of $\wishart$ when $n^{K + 3} = \omega(d^{K+1})$ for all $K \in \mathbb{N}$ \cite{Chetelat2017TheMA}.

\paragraph{3. Soft Thresholds.}A line of work by Liu and Racz \cite{Liu21PhaseTransition,Liu2021APV} studies a soft threshold model. In it, edges are random rather than deterministic functions of $\langle \bfx_i,\bfx_j\rangle,$ which corresponds to 
$\sigma$ taking other values than just $0$ and $1.$ In \cite{Liu21PhaseTransition}, the connection of interest is given as a convex combination between pure noise 
$\sigma \equiv p,$ and $\threshold_p,$ namely
$\phi_q(\bfx, \bfy) = (1-q)p + q \threshold_p(\bfx, \bfy)$ for some $q\in [0,1].$ In \cite{Liu2021APV}, on the other hand, the authors consider smooth monotone connections formed as follows. Take $(\Omega, \distribution) = (\mathbb{R}^d, \mathcal{N}(0, I_d)).$ Suppose that $f$ is a $C_2$ $\alpha$-Lipschitz density corresponding to a zero-mean distribution over $\mathbb{R}.$  Let $F$ be the corresponding CDF. Let $\sigma_f(\bfx, \bfy):=F\big(\frac{\langle \bfx,\bfy\rangle - \mu_{p,d,r}}{r\sqrt{d}}\big),$ where $r\ge 1$ and $\mu_{p,d,r}$ is chosen so that $\expect[\sigma_f(\bfx,\bfy)] = p.$ The fact that $f$ is $\alpha$-Lipschitz implies that $\sigma_f$ is $O_{\alpha}(\frac{1}{r\sqrt{d}})$-Lipschitz when viewed as a function of $\langle\bfx,\bfy\rangle$ (this follows directly from \cite[Lemma 2.6.]{Liu2021APV}).
The main result is the following.

\begin{theorem}[\cite{Liu2021APV}]
\label{thm:liuraczlipschitzconnections}
Let $ \RGG(n, \mathbb{R}^d, \mathcal{N}(0, I_d), p, \sigma_f)$ be defined as above.\\ 
1) If 
$d = \omega_{p,\alpha}(\frac{n^3}{r^4}),$ then 
$$
\TV\Big(\ergraph, \RGG(n, \mathbb{R}^d, \mathcal{N}(0, I_d), p, \sigma_f)\Big) = o(1).
$$
2) If $d = o_{p,\alpha}(\frac{n^3}{r^6})$ and, additionally, $d/\log^2d = \omega(r^6)$ or $r/\log^2(r) = \omega(d^{1/6}),$
then 
$$
\TV\Big(\ergraph, \RGG(n, \mathbb{R}^d, \mathcal{N}(0, I_d), p, \sigma_f)\Big) = 1 - o(1).
$$  
\end{theorem}

The authors conjecture that the lower bound 2), which is derived via the signed triangle statistic, is tight \cite{Liu2021APV}. They derive their upper bounds in 1) using the following general claim on latent space random graphs. This claim is the starting point for the indistinguishability results in the current paper.
\begin{claim}[\cite{Liu2021APV}]
\label{claim:RaczLiuIndistinguishability} Let $p,q\in [0,1]$ be two not-necessarily equal probabilities.
Consider the graph $\plsg(n, \Omega, \distribution, \sigma, q)$ for some $n,\Omega, \distribution, \sigma, q.$ Define 
\begin{equation}
    \label{eq:gammadefinition}
\gamma_p(\bfx,\bfy) = \expect_{\bfz\sim \distribution}\bigg[(\sigma(\bfx,\bfz)-p)(\sigma(\bfz,\bfy)-p)\bigg].
\end{equation}
Then,
$$
\KL\Big(\plsg(n, \Omega, \distribution, \sigma, q) \|\ergraph\Big)\le  
\sum_{k = 0}^{n-1}
    \log \expect_{(\bfx,\bfy)\sim \mathcal{D}\times\mathcal{D}}\bigg[\Big(1 + \frac{\gamma_p(\bfx,\bfy)}{p(1-p)}\Big)^k\bigg].
$$
\end{claim}

Typically, we take $p = q.$ We state the result in this more general form for the purely technical reason that in \cref{sec:ragindistinguishability} we will need it for $q = p(1+\epsilon),$ where $\epsilon$ is exponentially smaller than $p.$

In light of Pinsker's inequlity (see \cref{thm:pinsker}), showing that the KL-divergence is of order $o(1)$ implies that the total variation is also of order $o(1).$
Expanding the expression above, we obtain
\begin{equation}
\label{eq:firsteqonindist}
    \expect_{(\bfx,\bfy)\sim \mathcal{D}\times\mathcal{D}}\Big[\Big(1 + \frac{\gamma_p(\bfx,\bfy)}{p(1-p)}\Big)^k\Big] = 
    \sum_{t = 0}^k \binom{k}{t}\frac{\expect[\gamma_p^t]}{p^t(1-p)^t}.
\end{equation}
Thus, \cref{claim:RaczLiuIndistinguishability} reduces the task of proving indistinguishability bounds to bounding moments of $\gamma_p.$ In \cite{Liu2021APV}, the authors achieve this via a wide range of tools in Gaussian analysis. Our approach is quite different and it relies on 
interpreting $\gamma_p$ as an autocorrelation function when $\Omega$ has a group-theoretic structure.

Unfortunately, as noted by \cite{Liu2021APV}, the claim is likely not tight. In \cref{sec:improvingautocorrelation}, we demonstrate new instances in which (and reasons why) it is not tight, which hopefully will shed light on how to improve \cref{claim:RaczLiuIndistinguishability}. Improving the claim has the potential to resolve several open problems such as \cref{conj:sphericalhardthresholds} and the gap in \cref{thm:liuraczlipschitzconnections}.

\paragraph{4. Other Random Geometric Graphs.}Several other random geometric graph models have been studied recently as well. For example, \cite{Brennan22AnisotropicRGG,Mikulincer20} consider the hard thresholds model for anisotropic Gaussian vectors. Together, the two papers prove a tight condition on the vector of eigenvalues of the covariance matrix under which the total variation is of order $o(1).$

\subsection{Our Results on The Hypercube}
\label{sec:introhypercube}
There is a large gap in the literature, with all prior works focusing on connections $\sigma$ with the following two properties:
\begin{enumerate}
\item \textit{Monotonicity:}  Monotonicity of $\langle \bfx, \bfy\rangle\longrightarrow \sigma(\langle\bfx,\bfy\rangle)$ is a natural assumption which can be interpreted as closer/more aligned vectors are more likely to form a connection. 
Nevertheless, many interesting choices of $\sigma$ are not monotone. For example, connections could be formed when vectors are instead weakly correlated, corresponding to the non-monotone $\sigma(\bfx,\bfy) = \indicator[|\langle \bfx,\bfy\rangle|\le \delta]$ for some $\delta.$  As we will see in \cref{sec:symmetricconnectionsapplications}, the underlying symmetry around $0,$ i.e., $\sigma(\bfx,\bfy) =\sigma(\bfx, -\bfy)$ leads to a different indistinguishability rate.

\item \textit{Symmetry:} Connections $\sigma(\bfx,\bfy)$ depending only on $\langle \bfx,\bfy\rangle$ are symmetric with respect to permutations of coordinates. However, in some interesting examples different coordinates can influence the connection in qualitatively different ways. For example, suppose that the latent vectors correspond to characteristics of people and edges in the graph to friendships among the set of people. Similarity in some characteristics---such as geographical location---can make people more likely to form a friendship, but others---such as competition for scarce resources---can make them less likely to become friends due to competitiveness. 
In \cref{sec:transformationsofsymmetric} we use a simple mathematical model of such connections, which again leads to different indistinguishability rates.
\end{enumerate}

\cref{thm:maintheoremindistingishability} in the current paper makes a step towards filling this gap. 
It addresses a large family of probabilistic latent space graphs over the hypercube $\Omega = \{\pm 1 \}^d$ with the uniform measure $\distribution = \unif$ for which neither monotonicity nor symmetry holds. We remark that because our proofs are based on Fourier-analytic tools, which have generalizations to products of abelian groups and tori, our results likely hold in greater generality.  
In the interest of obtaining simple and interpretable results, we do not pursue this here.

The main insight is that over the hypercube, Fourier-analytic tools facilitate the analysis of connections depending on the coordinate-wise (group-theoretic) product $(\bfx,\bfy)\longrightarrow \bfx\bfy\in \hypercube.$ This setup is strictly more general than the setting of inner products since $\sum_{i = 1}^d (\bfx\bfy)_i = \langle \bfx,\bfy\rangle.$ When $\sigma$ only depends on $\bfx\bfy,$ we use the notation $\RAG$ (standing for random algebraic graph, see \cref{def:rag}) instead of $\mathsf{PLSG}.$ \cref{thm:maintheoremindistingishability} yields indistinguishability rates depending on the largest magnitude of Fourier coefficients on each level (see \cref{sec:booleanfourier} for preliminaries on Boolean Fourier analysis). Informally, we have:

\begin{theorem}[Informal \cref{thm:maintheoremindistingishability}]
\label{thm:mainintro}
Suppose that $d = \Omega(n),$ $\sigma:\{\pm 1 \}^d\longrightarrow [0,1]$ has expectation $p,$ and  
$m\in \mathbb{N}$ is a constant. Let $\sigma = \sum_{S\subseteq [d]}\widehat{\sigma}(S)\omega_S$.
For $1\le i \le d,$ let $ B_i = \max \big\{ |\widehat{\sigma}(S)|\binom{d}{i}^{1/2} : 
|S| = i\big\}
$. Let
\begin{equation*}
C_m = \sum_{i = m+1}^{ d/2en} B_i^2 + \sum_{i = d-\frac{d}{2en}}^{d-m-1} B_i^2\quadand
D  = \sum_{\frac{d}{2en}\le 
j
\le d - \frac{d}{2en}
} B_i^2.
\end{equation*}
If $C_m, B_u^2, B_{d-u}^2\le p(1-p)$ for $0\le u \le m,$ then
\begin{align*}
        & \KL\Big( \hypercubegraph\| \ergraph\Big)
        \\ & \qquad \qquad = 
        O\Bigg(
        \frac{n^3}{p^2(1-p)^2}\times 
\left(
\sum_{i = 1}^m
\frac{B_i^4}{d^i} + 
\sum_{i = d-m}^d
\frac{B_i^4}{d^i} 
+ \frac{C^2_m}{{d}^{m+1}} + 
{D^2}\times \exp\left( - \frac{d}{2en}\right)
\right)\Bigg).
\end{align*}
\end{theorem}
Qualitatively, the $\KL$ divergence is small whenever all Fourier coefficients of $\sigma$ are small. In \cref{rmk:qualitativeconverseofmain}, we show the converse -- a single large Fourier coefficient makes the $\KL$ divergence large.

In the beginning of \cref{sec:indistinguishability}, we summarize the four main steps in our proof of \cref{thm:mainintro} before we fill in all the technical details. Here, we highlight the following inequality we derive which is crucial to the proof. Its utility comes into play as we 
derive an upper bound of the desired $\KL$ divergence via the moments of a symmetric function. Symmetric functions over the hypercube are weighted sums of elementary symmetric polynomials.
\begin{theorem}[\cref{thm:symmetricpolybounds}]
\label{thm:introelemntarysymmetric}
Let $d\in \mathbb{N}.$ Consider the elementary symmetric polynomial $e_s :\mathbb{R}^d\longrightarrow \mathbb{R}^d$ given by $e_s(\bfx) = \sum_{S\subseteq [d]\; : \; |S| =s}\prod_{i \in S}x_i.$ If $t\in \mathbb{N}$ is such that $\frac{d}{2et}<s<d - \frac{d}{2et},$ then
$$
\binom{d}{s} 
\exp\Big( - \ln 2\times \frac{d}{t}\Big)\le 
\expect_{\bfx\sim \unif(\hypercube)}\big[|e_s(\bfx)|^t\big]^{1/t} \le \binom{d}{s}\exp\Big( - \frac{d}{2et}\Big).
$$
\end{theorem}

We have not tried to optimize the constants in the exponent. The upper bound also holds for other distributions such as uniform on the sphere and Gaussian when $s<{d}/{2}.$

We now describe some applications of \cref{thm:mainintro} 
to various choices of $\sigma$. 
 While the claim of the theorem is rather long, it turns out to be relatively simple to apply. 
  We organize our results based on the degree of symmetry in the respective connections. Note that in the symmetric case over $\hypercube,$ $\sigma(\bfx, \bfy) = \sigma(\bfx\bfy)$ only depends on 
$\sum_{i = 1}^d x_i y_i = \langle \bfx, \bfy\rangle$ and this is the usual inner-product model.

\subsubsection{Symmetric Connections}
\label{sec:introsymmetric}
 Most previous work on indistinguishability focuses on monotone connections and proves a necessary condition of the form $d = \Tilde{\Omega}(n^3)$ (with hidden constants, depending on $p,$ Lipschitzness, etc).  We show that even some of the simplest instances of non-monotone, but still symmetric, connections violate this trend.

\begin{theorem}[Informal, \cref{cor:doublethresholdconnectionsindist,cor:evenlipschitzindist,cor:fluctuationindist}] If $\sigma$ is symmetric and even (i.e., $\sigma(\bfg) = \sigma(-\bfg)$), then $\TV\big(\hypercubegraph, \ergraph\big) = o(1),$ whenever $d = \tilde{\Omega}(n^{3/2})$.
\end{theorem}

The reason why this holds is that odd Fourier levels of $\sigma$ vanish for even connections (so, in particular $B_1 = 0$ in \cref{thm:mainintro}). We generalize this result with the following conjecture, which, unfortunately, does not follow from our \cref{thm:maintheoremindistingishability} for $m>3$ as the latter only applies to the regime $d = \Omega(n).$

\begin{conjecture}[\cref{conj:nolowandhighdegreeterms,}]
\label{conj:intronolowandhigh}
Suppose that $\sigma$ is a symmetric connection such that $\widehat{\sigma}(S) = 0$ whenever $1\le |S|\le m-1$ or $d-m-1\le |S|\le d$ for a fixed constant $m.$ If $d = \tilde{\Omega}(n^{3/m})$ holds, 
$\TV\big(\ergraph, \hypercubegraph\big) = o(1).$
\end{conjecture}

In \cref{prop:nolowandhighdegreeterms} we show that, if correct, this conjecture is tight up to logarithmic factors. Specifically, there exist connections $\sigma$ for which the conditions of  \cref{conj:nolowandhighdegreeterms} are satisfied and the signed-triangle statistic distinguishes the two models whenever $d = \tilde{o}(n^{3/m}).$

Some further nearly immediate applications of  \cref{thm:mainintro} in the symmetric setting are the following.
\begin{enumerate}
\item \textbf{Hard Thresholds Connections.} When $\sigma = \threshold_p$ (defined over the hypercube analogously to the definition of the unit sphere), we prove in \cref{cor:classicrggindistinguishability} that if $d = \max\big(\tilde{\omega}(n^3p^2), \Omega(n\log n)\big)$ holds, then $\TV\big(\RAG(n,\hypercube,p, \threshold_p), \ergraph\big) = o(1).$ This matches the state of the art result on the sphere in the regime $p = \omega(n^{-1})$ in \cite{Liu2022STOC}. For the even analogue of 
$\threshold_p$ given by $\doublethreshold_p(\bfx,\bfy) := \indicator\left[|\langle \bfx,\bfy\rangle|\ge \tau_{p/2,d}|
\right]$ with expectation $p,$ $\ergraph$ and $\RAG(n,\hypercube, p,\doublethreshold_p)$ are indistinguishable when $d = \max\big((\tilde{\omega}(n^{3/2}p),\Omega(n\log n)\big)$ and distinguishable when $d = o(n^{3/2}p^{3/2})$ (see \cref{cor:doublethresholdconnectionsindist,cor:doublethresholddetection}).
\item \textbf{Lipschitz Connections.} In \cref{cor:generallipschitzindist} we show that if $d = \max\big(\Omega(n^3/r^4), \Omega(n\log n)\big)$ for an $\frac{1}{r\sqrt{d}}$-Lipshcitz connection $\sigma,$ $\TV\big(\RAG(n,\hypercube,p, \sigma), \ergraph\big) = o(1)$, extending the work of \cite{Liu2021APV} to the non-monotone case. Again, we improve the dependence to \linebreak $d =\max\big(\Omega( n^{3/2}/r^2), \Omega(n\log n)\big)$ in the even case in \cref{cor:evenlipschitzindist}.
\item \textbf{Interval Unions and Fluctuations.} In \cref{sec:fluctuations} we extend the classical threshold model in which $\sigma$ is an indicator of a single interval, i.e. $[\tau_{p,d},+\infty)$ to the case when $\sigma$ is an indicator of a union of $s$ disjoint intervals. Namely, we have indistinguishability when $d =\Omega( n^3 + n^{3/2}s^2)$ in the general case and $d = \Omega(n^{3/2}s^2)$ in the even case.
We give lower bounds in \cref{sec:nonmonotonedetection} when $d = o(ns^2),$
 showing that the dependence on $s$ is tight.
\item \textbf{Low-Degree Connections.} In \cref{sec:lowdegreepolyindist} we consider the case when $\sigma$ is a symmetric polynomial of constant degree. In that case, we prove that the two graph models are indistinguishable when $d = \Omega(n)$ and $d = \tilde{\omega}(\min(n^3p^2, np^{-2/3})).$ In \cref{sec:lowdegreedetection}, we give a detection lower bound for $d = o((np)^{3/4}),$ which leaves a polynomial gap in this setting.
\end{enumerate}

\subsubsection{Modifications of Symmetric Connections}
As \cref{thm:maintheoremindistingishability} only depends on the size of Fourier coefficients on each level, it also applies to transformations of symmetric connections which do not increase the absolute values of Fourier coefficients. We analyse two such transofrmations, which we call ``coefficients contractions'' and ``repulsion-attraction'' twists. Both have natural interpretations, see \cref{sec:coefficientcontractions}.

A special case of a coefficient contraction is the following. For a symmetric connection $\sigma: \hypercube\longrightarrow [0,1]$ and fixed $\bfg\in \hypercube,$ the coefficient contraction $\sigma_\bfg$ is given by $\sigma_\bfg(\bfx):=\sigma(\bfx\bfg)$ (which simply negates the coordinates $x_i$ for which $g_i = -1$). All of our indistinguishability results from \cref{sec:introsymmetric} continue to hold with the exact same quantitative bounds on $d$ if we replace the connection $\sigma$ with an arbitrary coefficient contraction $\sigma_\bfg$ of it. While it is a virtue of our proof techniques for indistinguishability results that they capture these more general non-symmetric connections, it is also a drawback. It turns out that certain coefficient contractions become statistically indistinguishable from $\ergraph$  for polynomially smaller dimensions, but this phenomenon is not captured by \cref{thm:mainintro}.

A particularly interesting case which  illustrates this is the following. In \cref{sec:doom}, we consider coefficient contractions with $\bfh = \underbrace{1,1,\ldots, 1}_{d_1}\underbrace{-1,-1,\ldots, -1}_{d_1}$ for $d = 2d_1.$ This choice of $\bfh$ negates exactly half the variables. In that case,  $\sigma_\bfh(\bfx^{-1}\bfy) = \sigma_\bfh(\bfx\bfy)$ depends on $\sum_{i \le d_1}x_iy_i - \sum_{j> d_1}x_jy_j.$ Thus, edges depend on a difference of inner products rather than inner products (equivalently, inner products with signature 
$(\underbrace{+,+,\ldots, +}_{d_1}\underbrace{-,-,\ldots, -}_{d_1})$
). The corresponding analogue of Wishart matrices is of the form $XX^T - YY^T,$ where $X, Y \in \mathbb{R}^{ n\times d_1}.$ In the Gaussian case, we prove the following statement. 

\begin{theorem}[\cref{thm:differenceofwisharts}]
\label{thm:informaldifferenceofwisharts}
Define the law of the difference of two Wishart matrices as follows.\linebreak $\wishart(n,d_1,-d_1)$ is the law of $XX^T - YY^T,$ where $X,Y\in \mathbb{R}^{n\times d_1}$ are iid matrices with independent standard Gaussian entries. If 
$d_1 = {\omega}(n^2),$ then
$\displaystyle
\TV\big(\wishart(n,d_1,-d_1), \sqrt{d_1}M(n)\big) = o_n(1).
$
\end{theorem}

Compared to \cref{thm:wisharttogoe}, this statement differs only in dimension. A Wishart matrix converges to GOE when $d = \omega(n^3),$ but a difference of independent Wishart matrices converges to GOE when $d = \omega(n^2).$ 

This proves that in the Gaussian case, for any connection $\sigma(\bfx, \bfy)$ over $\mathbb{R}^{d}\times \mathbb{R}^{d}$
that only depends on $\sum_{i \le d_1}x_iy_i - \sum_{j> d_1}x_jy_j,$ it is the case that 
$\TV\big(\RGG(n,\mathbb{R}^{2d_1},\mathcal{N}(0,I_{2d_1}), p,  \sigma), \ergraph\big) = o(1)$ whenever $d = {\omega}(n^2).$ We expect that the same thing holds for the hypercube model, but our \cref{thm:maintheoremindistingishability} yields dependence $d = {\omega}(n^3).$ 

In \cref{sec:improvingautocorrelation}, we use \cref{thm:informaldifferenceofwisharts} to rigorously show that there is an inefficiency in deriving \cref{claim:RaczLiuIndistinguishability} arising from the use of KL-convexity, as predicted in \cite{Liu2021APV}.

\subsubsection{Typical Indicator Connections Induce a Statistical-Computational Gap}
Finally, we study the behaviour of $\RAG(n, \hypercube, p_A, \sigma_A)$ where $A\subseteq \hypercube$ is a ``typical'' subset of the hypercube, $\sigma_A$ is its indicator, and $p_A = |A|/2^d$ is the expectation. We obtain strong evidence for a nearly-exponential statistical-computational gap. 

\begin{theorem}[Informal, \cref{thm:typicaldense,thm:entropyargument,thm:lowdegreepolyagainstrandomsubsets,thm:birthdaypaardox}]
\label{thm:informalhypercubegap}
In the setup above,
\begin{enumerate}[label=(\alph*)]
    \item If $d = \Omega(n\log n),$ then $\TV\Big(\RAG(n, \hypercube, p_A, \sigma_A)\|\ergraphhalf\Big) = o(1)$ for a $1 - 2^{-\Omega(d)}$ fraction of the subsets $A$ of $\hypercube.$
    \item If $d  = \frac{n}{2} - \omega(1),$ then 
    $\TV\Big(\RAG(n, \hypercube, p_A, \sigma_A)\|\ergraphhalf\Big) = 1- o(1)$ for all subsets $A$ of $\hypercube.$
    \item If $d = \log^{\Omega(1)}n,$ no low-degree polynomial test can distinguish $\RAG(n, \hypercube, p_A, \sigma_A)$ and $\ergraphhalf$ for a $1 - 2^{-\Omega(d)}$ fraction of the subsets $A$ of $\hypercube.$ 
\end{enumerate}
\end{theorem}

Our result is more general and holds for other values of $p$ beyond $1/2.$ We describe this in more detail in the next section, when we extend \cref{thm:informalhypercubegap} to arbitrary groups of appropriate size. The advantage of the Fourier-analytic proof in the hypercube setting over our proof in the setting of general groups is that it gives an explicit construction of connections for which the statistical limit occurs at  $d = \tilde{\Theta}(n).$ In particular, this is the case for all $A$ such that $\sigma_A$ is ${{d^C}}/{2^{d/2}}$-regular for a fixed constant $C$ (see \cref{claim:fouriercoefftypicalindicators,rmk:constrcutionwithgeularfunctions}). 

\subsection{Random Algebraic Graphs}
\label{sec:grouptheorymodel}
The main insight about the Boolean hypercube is that we can analyze connections $\sigma(\bfx, \bfy)$ depending only the group product, which is much more general than the inner product induced by $\mathbb{R}^d.$ It turns out that this construction can be naturally extended to a wide class of groups.

\begin{definition}[Random Algebraic Graphs]
\label{def:rag}
Suppose that $\Omega$ is a unimodular locally compact Hausdorff topological group $\Group$ with the unique left- and right-invariant probability Haar measure $\mu$ corresponding to distribution $\mathcal{D}.$\footnote{
If $\Group$ is finite, this measure is just the uniform measure over the group elements. If $\Group$ is, say, the orthogonal group, this is the Haar measure over it. Importantly, since the group is unimodular, the left and right Haar measures coincide.
} Suppose further that the connection $\zeta :\Group \times \Group \longrightarrow [0,1]$ given by $\zeta(\bfx, \bfy)$ only depends on $\bfx^{-1}\bfy,$ that is $\zeta(\bfx, \bfy) = \sigma(\bfx^{-1}\bfy)$ for some $\sigma:\Group\longrightarrow[0,1].$ Then, we call $\mathsf{PLSG}(n, \Group, \distribution, p, \zeta)$ a left random algebraic graph. We use the notation
$
\LRAG(n,\Group,p,\sigma).
$ Similarly, define a right random algebraic graph and denote by $\RRAG$ for connections depending only on $\bfx\bfy^{-1}.$
\end{definition}

Clearly, for abelian groups such as $\hypercube$, right- and left- random algebraic graphs coincide, for which reason we will use the simpler notation $\RAG.$ We now demonstrate the expressivity of random algebraic graphs by relating them  to (approximate versions) of other well-studied graph models.

\paragraph{1. Random Geometric Graphs over the Sphere and Hypercube.} Random geometric graphs over the hypercube are random algebraic graphs. More surprisingly, any random geometric graph defined by $\Omega = \dsphere, \distribution = \unif,$ and a connection $\zeta$ depending only on inner products
can be represented as a random algebraic graph even though $\dsphere$ does not possess a group structure when $d>3.$
Namely, consider a different latent space - the orthogonal group $\orthogonalgroup(d)  =\Omega$ with its Haar measure. The latent vectors are iid 
``uniformly distributed'' orthogonal matrices
$U_1, U_2, \ldots, U_n.$ To generate uniform iid latent vectors on $\dsphere,$ take
$U_1v, U_2v, \ldots, U_nv$ for an arbitrary $v\in \dsphere,$ say $v = e_1.$ Define $\sigma(U):=\zeta(\langle vv^T, U\rangle)$ and observe that 
$$
\zeta(U_iv, U_jv) = 
\zeta(\langle U_iv,U_jv\rangle) = 
\zeta(\langle vv^T, U_i^TU_j\rangle) = 
\sigma(U_i^{-1}U_j).
$$
Clearly, $\RGG(n, \dsphere, \unif, p, \zeta)$ and $\LRAG(n, \orthogonalgroup(d), p,\sigma)$ have the same distribution.

\paragraph{2. Cayley Graphs: Blow-ups and Random Subgraphs.} Suppose that $\Group$ is a finite group and $\sigma$ is an indicator of some $S\subseteq\Group$ which satisfies $a\in S$ if and only if $a^{-1}\in S.$ The Cayley graph  $\Gamma(\Group,S)$ has vertex set $\Group$ and edges between any two distinct vertices $g,h \in V(\Gamma(\Group,S)) = \Group$ for which $gh^{-1}\in S.$ We consider two regimes.

First, when $n = o(\sqrt{|\Group|}).$ With high probability, the latent vectors in $\RRAG(n, \Group, p, \sigma)$ are pairwise distinct by the birthday paradox. This immediately leads to the following observation.

\begin{observation}
    \label{obs:CayleyAndRAG}
    Suppose that $\Group$ is a group of size at least $n$ and $S\subseteq \Group$ is closed under taking inverses.
    Denote by $\Gamma_n(\Group,S)$ the uniform distribution over $n$-vertex induced subgraphs of $\Gamma(\Group,S).$ 
    If $|\Group| = \omega(n^2),$
$$\TV\Big(\RRAG(n, \Group, p, \sigma), \Gamma_n(\Group,S)\Big) = o(1).$$
\end{observation}

Second, when $|\Group| = o(\frac{\sqrt{n}}{\log n}).$ With high probability, each $g\in \Group$ appears $\frac{n}{|\Group|}(1 + o(1))$ times as a latent vector in $\RRAG(n, \Group, p, \sigma).$ It follows that $\RRAG(n, \Group, p, \sigma)$ is approximately an $\frac{n}{|\Group|}$-blow-up of $\Gamma(\Group,S).$

It is interesting to consider whether this relationship to Cayley graphs can be used to understand random geometric graphs. For example, there are several results on the expansion properties of Cayley graphs \cite{Alon94RandomCayleyGraphs,Conlon17HypergraphExpanders}, which could be related to the expansion properties of random geometric graphs \cite{Liu22Expander}.

\paragraph{3. Stochastic Block Model.}
The stochastic block-model is given by an $n$ vertex graph, in which $k$ communities and edges between them are formed as follows (\cite{Bubeck17NetworkSurvey}). First, each vertex is independently assigned a uniform label in $\{1,2,\ldots,k\}.$ Then, between every two vertices, an edge is drawn with probability $p$ if they have the same label and with probability $q<p$ otherwise. Denote this graph distribution by $\sbm(n,k,p,q).$

$\sbm(n,k,p,q)$ can be modeled as a random algebraic graph as follows. Take an arbitrary abelian group $\Group$ which has a subgroup $\mathcal{H}$ of order $|\Group|/k.$ Define $\sigma:\Group\longrightarrow [0,1]$ by 
$\sigma(\bfg):=q + (p-q)\times \indicator[\bfg \in \mathcal{H}].$ Then,
$\RAG(n , \Group, \frac{p + (k-1)q}{k}, \sigma)$ has the same distribution as $\sbm(n,k,p,q).$ This follows directly from the observation that communities in the stochastic block model correspond to subsets of vertices with latent vectors in the same coset with respect to $\mathcal{H}$ in the random algebraic graph.

\subsection{Our Results on Typical Cayley Graphs}
We utilize the connection between random algebraic graphs and Cayley graphs to study the high-probability behaviour of Cayley graphs when the generating set is chosen randomly. The distribution over random generating sets that we consider is the following.

\begin{definition}[The ante-inverse-closed uniform measure]
\label{def:anteuniform}
Let $\Group$ be a group and $0\le p \le 1$ a real number, possibly depending on $\Group.$ Define the density $\anteuniform(\Group, p)$ as follows. Consider the action of $C_2 = \{1,\rho\}$ on $\Group$ defined by $\rho(\bfg) = \bfg^{-1}$ and let $\Orbit_1, \Orbit_2, \ldots, \Orbit_t$ be the associated orbits. Let $X_1, X_2, \ldots, X_t$ be $t$ iid $\Bernoulli(p)$ random variables. Then,  $\anteuniform(\Group, p)$ is the law of $S = \bigcup_{j \; : \; X_j = 1}\Orbit_j.$
\end{definition}

Our main result, extending \cref{thm:informalhypercubegap}, is the following.

\begin{theorem}[Informal, \cref{thm:generalragindistinguishability,thm:entropyargument,thm:lowdegreepolyagainstrandomsubsets}]
\label{thm:informalraggap}
Consider a finite group $\Group.$
\begin{enumerate}[label=(\alph*)]
    \item If $|\Group| = \exp\big(\Omega(n\log \frac{1}{p})\big),$ then $\TV\Big(\Gamma_n(\Group, A),\ergraph\Big) = o(1)$ with probability at least $1 - |\Group|^{-1/7}$ over $A\sim \anteuniform(\Group,p).$
    \item If $|\Group| = \exp(O(nH(p))),$ where $H$ is the binary entropy function, then $\TV\Big(\Gamma_n(\Group, A),\ergraph\Big) = 1 - o(1)$ for any generating set $A.$
    \item If $|\Group| = \exp(\Omega(k\log n)),$ no polynomial test of degree at most $k$ with input the graph edges can distinguish $\ergraph$ and $\Gamma_n(\Group, A)$ with high probability over $A\sim \anteuniform(\Group,p).$
\end{enumerate}
\end{theorem}

We note that these results can be equivalently phrased if we replace $\Gamma_n(\Group, A)$ by $\RRAG(n, \Group, p_A, \sigma_A),$ where $\sigma_A$ is the indicator of $A$ and $p_A$ is its expectation.

In particular, this suggests that indistinguishability between $\Gamma_n(\Group, A)$ and $\ergraph$ depends only on the size of $\Group,$ but not on its group structure. This continues the long-standing tradition of certifying that important properties of random Cayley graphs only depend on group size. Prior literature on the topic includes results on random walks \cite{aldous86,dou86,Hermon2021CutoffFA} (see \cite{Hermon2021CutoffFA} for further references), related graph-expansion properties \cite{Alon94RandomCayleyGraphs}, and chromatic numbers \cite{ALON20131232} of typical Cayley graphs.

One major difference between previous results and our current result is the measure over generating sets used. Previous works used is what we call the \textit{post-inverse-closed uniform measure.}

\begin{definition}[The post-inverse-closed uniform measure]
\label{def:postuniform}
Let $\Group$ be a finite group and $k \le |\Group|$ be an integer. Define the density $\postuniform(\Group,k)$ as follows. Draw $k$ uniform elements $Z_1, Z_2, \ldots, Z_k$ from $\Group$ and set $Z = \{Z_1, Z_2, \ldots, Z_k\}.$ Then, 
$\postuniform(\Group,k)$ is the law of $S = Z\cup Z^{-1}.$
\end{definition}

There are several differences between the two measures on generating sets, but the crucial one turns out to be the following. In the \textit{ante-inverse-closed uniform measure}, each element appears in $S$ with probability $p.$ In the \textit{post-inverse-closed uniform measure}, an element $\bfg$ of order two appears with probability $q := 1 - (1 - \frac{1}{|\Group|})^k$ since $\bfg \in Z\cup Z^{-1}$ if and only if $\bfg\in Z.$ On the other hand, an element $\bfh$ of order greater than 2 with probability  $1 - (1 - \frac{2}{|\Group|})^k\approx 2q.$ This suggests that we cannot use the $\postuniform$ measure for groups which simultaneously have a large number of elements of order two and elements of order more than two to prove indistinguishability results such as the ones given in \cref{thm:informalraggap}. Consider the following example.

\begin{example}
\normalfont
Consider the group $\Group = \cycle_3\times \underbrace{\cycle_2\times \cycle_2\times\cdots\times\cycle_2}_d.$ Denote the three different cosets with respect to the $\cycle_3$ subgroup by $H_i = \{i\}\times \underbrace{\cycle_2\times \cycle_2\times\cdots\times\cycle_2}_d$ for $i \in \{0,1,2\}.$ Elements in $H_0$ have order at most $2,$ while elements in $H_1, H_2$ have order $3.$ Now, consider $\postuniform(\Group,k),$ where $k$ is chosen so that $\expect [|S|\times |\Group|^{-1}] \approx \frac{1}{2}.$ Each element from $H_0$ appears in $S$ with probability $q$ and each element in $H_1, H_2$ with probability $\sim 2q,$ such that $q\times |H_0| + 2q \times (|H_1| + |H_2|) = \frac{|\Group|}{2}.$ It follows that $q\approx \frac{3}{10}.$ 

Let $S$ be drawn from $\postuniform(\Group,k)$ and $G \sim \Gamma_n(\Group, S).$ With high probability, the vertices of $G$ can be split into three groups - $G_0, G_1, G_2,$ where $G_i$ contains the vertices coming from coset $H_i.$ The average density within each of $G_0, G_1, G_2$ in $G$ will be approximately $\frac{3}{10},$ while the average density of edges of $G$ which have at most one endpoint in each of $G_0, G_1, G_2$ will be approximately $\frac{6}{10}.$ Clearly, such a partitioning of the vertices $G_0, G_1, G_2$ typically does not exist for \ER graphs. Thus, no statistical indistinguishability results can be obtained using the $\postuniform$ measure over $\Group = \cycle_3\times \underbrace{\cycle_2\times \cycle_2\times\cdots\times\cycle_2}_d.$
\end{example}

Of course, there are other differences between $\anteuniform$ and $\postuniform,$ such as the fact that in the $\postuniform$ measure elements from different orbits do not appear independently (for example, as the maximal number of elements of $S$ is $2k$). These turn out to be less consequential. In \cref{appendix:beyondanteuniform}, we characterize further distributions beyond $\anteuniform$ for which our statistical and computational indistinguishability results hold. Informally, the desired property of a distribution $\mathsf{D}(\Group, p)$ is the following. Let $A\sim \mathsf{D}(\Group, p).$ For nearly all elements $\bfx\in \Group,$ if we condition on any event $E$ that $\polylog(|G|)$ elements besides $\bfx, \bfx^{-1}$ are or are not in $A,$ we still have $\prob[\bfx\in A|E] = p + O(|\Group|^{-c}),$ where $c\in (0,1)$ is constant. 

In particular, this means that our results over $\Group = \hypercube$ hold when we use $\postuniform(\Group, k)$ for an appropriately chosen $k$ so that $(1 - \frac{1}{|\Group|})^k = p + O\big(\frac{1}{|\Group|}\big)$ since all elements of $\hypercube$ have order $2.$ Similarly, we can use the $\postuniform$ density with an appropriate $k$ for all groups without elements of order two, which include the groups of odd cardinality.

\section{Notation and Preliminaries}
\label{sec:preliminaries}
\subsection{Information Theory}
The main subject of the current paper is to understand when two graph distributions - $\ergraph$ and\linebreak $\PLSG$  - are statistically indistinguishable or not. For that reason, we need to define similarity measures between distributions. 

\begin{definition}[For example, \cite{Polianskiy22+}] Let $\mathbf{P}, \mathbf{Q}$ be two probability measures over the measurable space $(\Omega, \mathcal{F}).$ Suppose that $\mathbf{P}$ is absolutely continuous with respect to $\mathbf{Q}$ and has a Radon-Nikodym derivative $\frac{d\prob}{d\mathbf{Q}}.$ For a convex function $f:\mathbb{R}_{\ge 0}\longrightarrow \mathbb{R}$ satisfying $f(1) = 0,$ the $f$-divergence between $\mathbf{P}, \mathbf{Q}$ is given by 
$$
\mathsf{D}_f(\prob\|\mathbf{Q}) = 
\int_\Omega f\bigg(\frac{d\prob}{d\mathbf{Q}}\bigg)d\mathbf{Q}. 
$$    
\end{definition}

The two $f$-divergences of primary interest to the current paper are the following:
\begin{enumerate}
    \item \textit{Total Variation:} $\TV(\mathbf{P}, \mathbf{Q}):=\mathsf{D}_f(\prob\|\mathbf{Q})$ for $f(x) = \frac{1}{2}|x-1|.$ It easy to show that $\TV$ can be equivalently defined by 
    $\TV(\mathbf{P}, \mathbf{Q}) = \sup_{A\in \mathcal{F}}|\mathbf{P}(A) - \mathbf{Q}(A)|.$
    \item \textit{KL-Divergence:} $\KL(\mathbf{P}\| \mathbf{Q}):=\mathsf{D}_f(\prob\|\mathbf{Q})$ for $f(x) = x\log x$ (where $f(0) = 0$).
\end{enumerate}

It is well-known that total variation characterizes the error probability of an optimal statistical test. 
However, it is often easier to prove bounds on KL and translate to total variation via Pinsker's inequality.

\begin{theorem}[For example, \cite{Polianskiy22+}]
    \label{thm:pinsker} Let $\mathbf{P}, \mathbf{Q}$ be two probability measures over the measurable space $(\Omega, \mathcal{F})$ such that $\mathbf{P}$ is absolutely continuous with respect to $\mathbf{Q}.$ Then, 
    $$
    \TV(\mathbf{P}, \mathbf{Q})\le \sqrt{\frac{1}{2}\KL(\mathbf{P}\| \mathbf{Q})}.
    $$
\end{theorem}


Another classic inequality which we will need is the data-processing inequality. We only state it in the very weak form in which we will need it.

\begin{theorem}[For example, \cite{Polianskiy22+}]
\label{thm:dataprocessing}
For any two distributions $\mathbf{P}, \mathbf{Q},$ any measurable deterministic function $h,$ and any  $f$-divergence $\mathsf{D}_f,$ the inequality $\mathsf{D}_f(\prob\|\mathbf{Q})\ge \mathsf{D}_f(h_*(\prob)\|h_*(\mathbf{Q}))$ holds. Here, $h_*(\prob), h_*(\mathbf{Q})$ are the respective push-forward measures.
\end{theorem}

\subsection{Boolean Fourier Analysis}
\label{sec:booleanfourier}
One of the main contributions of the current paper is that it introduces tools from Boolean Fourier analysis to the study of random geometric graphs. Here, we make a very brief introduction to the topic. An excellent book on the subject is \cite{ODonellBoolean}, on which our current exposition is largely based.

We note that $\{\pm 1\}^d$ is an abelian group (isomorphic to $\mathbb{Z}_2^d$) and we denote the product of elements $\bfx, \bfy$ in it simply by $\bfx\bfy.$ Note that $\bfx = \bfx^{-1}$ over $\{\pm 1\}^d.$

Boolean Fourier analysis is used to study the behaviour of functions $f:\{\pm 1\}^d\longrightarrow\mathbb{R}.$ An important fact is the following theorem stating that all functions on the hypercube are polynomials.

\begin{theorem} Every function $f:\{\pm 1\}^d\longrightarrow\mathbb{R}$ can be uniquely represented as 
$$
f(\bfx) = \sum_{S\subseteq [d]} \widehat{f}(S)\omega_S(\bfx),
$$
where $\omega_S(\bfx) = \prod_{i \in S}x_i$ and 
$\widehat{f}(S)$ is a real number, which we call the Fourier coefficient on $S.$
\end{theorem}

We call the monomials $\omega_S(\bfx)$ \textit{characters} or \textit{Walsh polynomials}. One can easily check that $\omega_\emptyset\equiv 1$ and $\omega_s\omega_T = \omega_{S\Delta T}.$
The Walsh polynomials form an orthogonal basis on the space of real-valued functions over $\{\pm 1\}^d.$ Namely, for two such functions $f,g,$ define 
$
\langle f, g\rangle = 
\expect_{\bfx\sim \unif(\{\pm 1\}^d)}[f(\bfx)g(\bfx)].
$
This is a well-defined inner-product, to which we can associate an $L_2$-norm $\norm{f}^2_2:= \langle f, f\rangle$.
\begin{theorem}[For example, \cite{ODonellBoolean}]
The following identities hold.
\begin{enumerate}
    \item $\langle \omega_s, \omega_T\rangle = \indicator[T = S].$
    \item $\langle f, g\rangle = \sum_{S\subseteq[d]}\widehat{f}(S)\widehat{g}(S).$ In particular,
    \begin{enumerate}
        \item $\widehat{f}(S) = \expect[f(\bfx)\omega_S(\bfx)].$
        \item $\norm{f}_2^2 = \sum_{S\subseteq[d]}\widehat{f}(S)^2.$
        \item $\expect[f] = \widehat{f}(\emptyset), \Var[f] = \sum_{ \emptyset\subsetneq S\subseteq [d]}\widehat{f}(S)^2.$
    \end{enumerate}
    \item The Fourier convolution $(f*g):\{\pm1\}^d\longrightarrow \mathbb{R}$ defined by 
    $(f*g)[\bfx] = \expect[f(\bfx \bfz^{-1})g(\bfz)]$ can be expressed as 
    $$
    (f*g)[\bfx] = \sum_{S\subseteq [d]}\widehat{f}(S)\widehat{g}(S)\omega_S(\bfx).
    $$
\end{enumerate}
\end{theorem}

Oftentimes, it is useful to study other $L_p$-norms, defined for $p\ge 1$ by 
$\displaystyle
\norm{f}_p: = \expect_{\bfx\sim \unif(\{\pm 1\}^d)}[|f|^p]^{1/p}.
$
By Jensen's inequality, $\norm{f}_p \le \norm{f}_q$ holds whenever $p <q.$ It turns out, however, that a certain reverse inequality also holds. To state it, we need the noise operator. 

\begin{definition} For $\bfx\in \hypercube,$ let $\mathcal{N}_\rho(\bfx)$ be the $\rho$-correlated distribution of $\bfx$ for $\rho \in [0,1].$ That is, $\bfy \sim \mathcal{N}_\rho(\bfx)$ is defined as follows.
For each $i \in [d],$ independently $y_i = x_i$ with probability $\frac{1 + \rho}{2}$ and $y_i = -x_i$ with probability $\frac{1-\rho}{2}.$ Define:
\begin{enumerate}
    \item Noise Operator: The linear operator $T_\rho$ on functions over the hypercube is defined by 
    $T_\rho f[\bfx]: = \expect_{\bfy \sim \mathcal{N}_\rho(\bfx)}[f(\bfy)].$ The Fourier expansion of $T_\rho f$ is given by 
    $$T_\rho f = \sum_{S\subseteq [d]}\rho^{|S|}\widehat{f}(S)\omega_S.$$
    In light of this equality, $T_\rho$ can also be defined for $\rho>1.$
    \item Stability: The Stability of a function $f:\hypercube\longrightarrow \mathbb{R}$ is given by 
    $$
    \stab_\rho[f]:=\langle T_\rho f, f\rangle = 
    \sum_{S\subseteq [d]}\rho^{|S|}\widehat{f}(S)^2 = 
    \norm{T_{\sqrt{\rho}} f}_2^2
    $$
\end{enumerate}
\end{definition}

We are ready to state the main hypercontractivity result which we will use in \cref{sec:indistinguishability}.

\begin{theorem}[\cite{Bonami1970}]
\label{lem:hypercontractivity}
For any $1\le p \le q$ and any $f :\{\pm 1 \}^d\longrightarrow\mathbb{R},$ we have the inequality
$$
\norm{f}_q\le 
\Bgn{T_{\sqrt{\frac{q-1}{p-1}}}f}\Bgn_p.
$$
In particular, when $q\ge p = 2,$ one has
$\displaystyle 
\norm{f}^2_q \le 
\sum_{S \subseteq\emptyset}
(q-1)^{|S|}\widehat{f}(S)^2.
$
\end{theorem}

The way in which the noise operator changes the Fourier coefficient of $\omega_S$ depends solely on the size of $S.$ For that, and other, reasons, it turns out that it is useful to introduce the Fourier weights. Namely, for a function $f:\hypercube\longrightarrow\mathbb{R},$ we define its weight on level $i \in \{0,1,2,\ldots, d\}$ by 
$$
\weight{i}{f} = 
\sum_{S\subseteq [d]\; : \; |S| = i}
\widehat{f}(S)^2.
$$
We similarly denote $\weight{\le i}{f} = \sum_{j \le i }\weight{ j}{f}$ and $\weight{\ge i}{f} = \sum_{j \ge i }\weight{j}{f}.$
Trivially, for each $i,$ one has $\weight{i}{f}\le \sum_{j = 0}^d \weight{j}{f} = \norm{f}_2^2.$ In particular, this means that if $|f|\le 1$ holds a.s. (which is the case for connections $\sigma$ in random geometric graphs by definition), $\weight{i}{f}\le 1$ also holds for each $i.$ It turns out that in certain cases, one can derive sharper inequalities. 

\begin{theorem}[\cite{Kahn1988TheIO}]
\label{thm:levelkinequalities}
For any function $f: \{\pm1\}^d\longrightarrow [-1,1]$ such that $\norm{f}_1 = \alpha,$ the inequalities $$\weight{\le k}{f}\le \Big(\frac{2e}{k}\ln(1/\alpha)\Big)^k\alpha^2\quadand
\weight{\ge d-k }{f}\le \Big(\frac{2e}{k}\ln(1/\alpha)\Big)^{k}\alpha^2
$$
hold whenever $k \le 2\ln(1/\alpha).$
\end{theorem}


We end with a discussion on low-degree polynomials. We say that $f:\hypercube\longrightarrow \mathbb{R}$ has degree $k$ if it is a degree $k$ polynomial, that is $\widehat{f}(S) = 0$ whenever $|S|>k.$ It turns out that such polynomials have surprisingly small Fourier coefficients. Namely, Eskenazis and Ivanisvili use a variant of the Bohnenblust-Hille to prove the following result.

\begin{theorem}[\cite{Eskenazis_21}]
\label{thm:lowdegreecoefficients}
For any fixed number $k\in \mathbb{N},$ there exists some constant $C_k$ with the following property. If 
$f:\{\pm 1\}^d\longrightarrow[-1,1]$ is a degree $k$ function, then for any $\ell >0$
$$
\sum_{S\; : \; |S| = \ell}
|\widehat{f}(S)|\le 
C_k 
d^{\frac{\ell-1}{2}}.
$$
\end{theorem}

In particular, if $f$ is symmetric ($f(x_1, x_2, \ldots, x_n) = f(x_{\pi(1)}, x_{\pi(2)}, \ldots, x_{\pi(n)})$ holds for each $\bfx$ and permutation $\pi$), the following bound holds.

\begin{corollary}
\label{cor:lowdegreepolyweights}
If
$f:\{\pm 1\}^d\longrightarrow[-1,1]$ is a symmetric constant degree function, then 
$
\weight{\ell}{f} = O(1/d)
$ for any $\ell>0.$
\end{corollary}
\begin{proof}
We simply note that \cref{thm:lowdegreecoefficients} implies that 
\[
\weight{\ell}{f}  = 
\binom{d}{\ell}
\bigg(
\frac{\sum_{S \; : \: |S|=\ell}|\widehat{f}(S)|}{\binom{d}{\ell}}
\bigg)^2 = 
O\B(d^\ell \times \frac{d^{\ell-1}}{d^{2\ell}}\B).
\qedhere\]
\end{proof}


\subsection{Detection and Indistinguishability via Fourier Analysis}
\label{sec:ragviafourier}
Before we jump into the main arguments of the current paper, we make two simple illustrations of Fourier-analytic tools in the study of random algebraic graphs, which will be useful later on. These observations cast new light on previously studied objects in the literature.
\subsubsection{Indistinguishability}

First, we will interpret the function $\gamma$ from \cref{claim:RaczLiuIndistinguishability} as an autocorrelation function. 

\begin{proposition}
\label{prop:fourierviewindistinguishability}
Let $\Group = \hypercube$ and suppose that
$\sigma$ only depends on $\bfx\bfy = \bfx\bfy^{-1}.$ Then, 
$$\gamma(\bfx,\bfy) = (\sigma - p)*(\sigma - p)[\bfx\bfy^{-1}].$$ 
\end{proposition}
\begin{proof}
Let $\sigma(\bfg) = \sum_{S\subseteq [d]}\widehat{\sigma}(S)\omega_S(\bfg)$ for $\bfg\in \hypercube.$ In particular, $\widehat{\sigma}(\emptyset) = p.$ Now, we have 
\begin{align*}
    \gamma(\bfx,\bfy) & = \expect_{\bfz\sim \unif(\hypercube)}
    [(\sigma(\bfx,\bfz)-p)(\sigma(\bfy,\bfz)-p)] = \expect_\bfz\sum_{\emptyset\subsetneq S,T\subseteq [d]}\left[\widehat{\sigma}(S)\widehat{\sigma}(T)\omega_S(\bfx\bfz^{-1})\omega_T(\bfz\bfy^{-1})\right]
    \\
    & =  \sum_{\emptyset\subsetneq S,T\subseteq [d]}\widehat{\sigma}(S)\widehat{\sigma}(T)\omega_S(\bfx)\omega_T(\bfy^{-1})\expect[\omega_{T\delta S}(\bfz)] = 
    \sum_{S\neq \emptyset }\widehat{\sigma}(S)^2\omega_S(\bfx\bfy^{-1})\\
    & = (\sigma - p)*(\sigma - p)[\bfx\bfy^{-1}].\qedhere
\end{align*}    
\end{proof}

The same argument applies verbatim when we replace $\hypercube$ with an arbitrary finite abelian group (and some other groups such as a product of a finite abelian group with a torus, provided $\sigma$ is $L_2$-integrable with respect to $\distribution$). 
Hence our use of $\bfx\bfy^{-1}$.
In light of \cref{eq:firsteqonindist}, we simply need to prove bounds on the moments of the autocorrelation of $\sigma-p$ (i.e., the centered moments of the autocorrelation of $\sigma$).

\subsubsection{Detection Using Signed Cycles in Temrs of Fourier Coefficients}

Let the indicator of edge $i,j$ be $G_{i,j}.$ Essentially all detection algorithms appearing in the literature on random geometric graphs are derived by counting ``signed'' $k$-cycles in the respective graph for $k \in \{3,4\}$ 
\cite{Liu2021APV,Bubeck14RGG,Brennan21DeFinetti}.
Formally, these works provide the following detection algorithm, which is efficiently computable in time $O(n^k)$ when $k$ is a constant. One sums over the set $\mathcal{C}$ of all simple cycles $i_1, i_2, \ldots, i_k$ the quantity $ \prod_{t = 1}^k (G_{i_t,i_{t+1}} - p),$
where $i_{k+1} := i_1.$ The resulting statistic is
\begin{equation}
\label{eq:kcyclesstats}
\tau_{k,p}(G) := \sum_{(i_1, i_2, \ldots, i_k)\in \mathcal{C}}
    \prod_{t = 1}^k (G_{i_t,i_{t+1}} - p) \,.
\end{equation}
Since each edge $G_{i,j}$ appears independently with probability $p$ in an \ER graph, clearly \linebreak 
$\displaystyle
    \expect_{G\sim \ergraph}[\tau_k(G)] = 0.
$
Thus, to distinguish with high probability between $\hypercubegraph$ and $\ergraph,$ it is enough to show that 
\begin{equation*}
    |\expect_{H\sim \hypercubegraph}\tau_k(G)| = 
    \omega\b(\Var_{G\sim \ergraph}[\tau_k(G)] + 
    \Var_{H\sim \hypercubegraph}[\tau_k(H)]
    \b)
\end{equation*}
by Chebyshev's inequality. We now demonstrate how to compute the relevant quantities in the language of Boolean Fourier analysis.
\begin{observation}
\label{obs:fouriercycles}
For any set of $k$ indices $i_1, i_2, \ldots, i_k,$ it is the case that 
$$\expect_{G\sim \hypercubegraph}\left[\prod_{t = 1}^k (G_{i_t,i_{t+1}} - p)\right] = \sum_{S\neq \emptyset} \widehat{\sigma}(S)^k.$$
\end{observation}
\begin{proof} We simply compute
\begin{align*}
        &\expect\left[\prod (G_{i_t,i_{t+1}} - p)\right] =  \expect\left[ \expect\left[\prod(G_{i_t,i_{t+1}} - p)| \bfx_{i_1}, \bfx_{i_2}, \ldots, \bfx_{i_t}\right]\right]\\ 
        & = \expect \left[\expect\left[\prod(\sigma(\bfx_{i_t}^{-1}\bfx_{i_{t+1}}) - p)| \bfx_{i_1}, \bfx_{i_2}, \ldots, \bfx_{i_t}\right]\right].
\end{align*}
Now, observe that $\bfx_{i_1}, \bfx_{i_2}, \ldots, \bfx_{i_k}$ are independently uniform over the hypercube. Thus, when we substitute $\bfg_\ell = \bfx_{i_{\ell-1}}^{-1}\bfx_{i_\ell},$ the vectors $\bfg_1, \bfg_2, \ldots, \bfg_{k-1}$ are independent and 
$\bfg_k = \bfg_1^{-1}\bfg_2^{-1}\cdots \bfg_{k-1}^{-1}.$
We conclude that 
\begin{align*}
& \expect \left[\expect\left[\prod(\sigma(\bfx_{i_t}^{-1}\bfx_{i_{t+1}}) - p)| \bfx_{i_1}, \bfx_{i_2}, \ldots, \bfx_{i_t}\right]\right]\\
& = \expect_{(\bfg_1, \bfg_2, \ldots, \bfg_{k-1})\sim_{iid} \unif(\hypercube)}
(\sigma(\bfg_1) - p)
(\sigma(\bfg_2) - p)\cdots
(\sigma(\bfg_{t-1})-p)
(\sigma(\bfg_1^{-1}\bfg_2^{-1}\cdots \bfg_{k-1}^{-1}) - p).
\end{align*}
Expanding the above product in the Fourier basis yields
\begin{align*}
       & \expect \left[\sum_{S_1, S_2, \ldots, S_{t}}
\prod_{i = 1}^t\widehat{\sigma}(S_i)\times 
\prod_{i = 1}^{t-1} \omega_{S_i}(\bfg_i)\times 
\omega_{S_t}(\bfg_1^{-1}\bfg_2^{-1}\cdots \bfg_{t-1}^{-1})\right]\\
& =  \sum_{S_1, S_2, \ldots, S_{t}}
\prod_{i = 1}^t\widehat{\sigma}(S_i)\times 
\prod_{i = 1}^{t-1}\expect\left[ \omega_{S_i}
(\bfg_i)\omega_{S_t}(\bfg_i^{-1})\right].
\end{align*}
Since $\langle \omega_{S_i}, \omega_{S_t}\rangle = \indicator[S_i = S_t],$ the conclusion follows.
\end{proof}

It follows that when $k = O(1),$
\begin{align*}
    \expect_{G \sim \hypercubegraph}[\tau_k(G)] = 
    \Theta\B(
    n^k \sum_{S\neq \emptyset} \widehat{\sigma}(S)^k
    \B).
\end{align*}

In \cref{sec:appendixvariancecomputation}, using very similar reasoning, we show that the variance can also be bounded in terms of the Fourier coefficients of $\sigma.$
\begin{proposition}
\label{prop:variancecomputation}
    $\displaystyle
    \Var_{G\sim \hypercubegraph}[\tau_3(G)] = O\B(
n^3p^3
 + n^3\sum_{S\neq \emptyset}\widehat{\sigma}(S)^3 + n^4\sum_{S\neq \emptyset}\widehat{\sigma}(S)^4
\B).
    $
\end{proposition}

Combining \cref{obs:fouriercycles,prop:variancecomputation}, we make the following conclusion.\\

\begin{corollary}
\label{cor:generaltriangledetection}
If 
$$
n^6\B(
\sum_{S\neq \emptyset}\widehat{\sigma}(S)^3
\B)^2 = 
\omega\B(
n^3p^3 + 
n^3 \sum_{S\neq \emptyset}\widehat{\sigma}(S)^3
+ 
n^4 \sum_{S\neq \emptyset}\widehat{\sigma}(S)^4
\B),
$$
then $\TV\Big(\hypercubegraph,\ergraph\Big) = 1 - o(1).$ 
Furthermore, the polynomial-time algorithm of counting signed triangles distinguishes $\ergraph$ and
$\hypercubegraph$ with high probability. 
\end{corollary}

Thus, we have shown that the entirety of the detection via signed triangles argument has a simple interpretation in the language of Boolean Fourier analysis. One can argue similarly for counting signed $k$-cycles when $k>3,$ but the variance computation becomes harder. We illustrate a very simple instance of this computation for $k= 4$ in \cref{appendix:signedfourcyclesrag}.

\section{The Main Theorem On Indistinguishability Over The Hypercube}
\label{sec:indistinguishability}
In this section, we prove our main technical result on indistinguishability between $\ergraph$ and\linebreak $\hypercubegraph.$ The statement itself does not look appealing due to the many conditions one needs to verify. In subsequent sections, we remove these conditions by exploring specific connections $\sigma.$
\begin{theorem}
\label{thm:maintheoremindistingishability}
Consider a dimension $d\in \mathbb{N},$ connection $\sigma:\{\pm 1 \}^d\longrightarrow [0,1]$ with expectation $p,$ and constant 
$m\in \mathbb{N}.$\footnote{We will usually take $m \le 3$ when applying the theorem, so one can really think of it as a small constant.} There exists a constant $K_m$ depending only on $m,$ but not on $\sigma, d, n,p,$ with the following property.
Suppose that $n \in \mathbb{N}$ is such that $nK_m < d.$
For $1\le i \le d,$ let\linebreak $ B_i = \max \B\{ |\widehat{\sigma}(S)|\binom{d}{i}^{1/2} : 
|S| = i\B\}
.$ Denote also
\begin{equation*}
C_m  = \sum_{i = m+1}^{ \frac{d}{2en}} B_i^2 + \sum_{i = d-\frac{d}{2en}}^{d-m-1} B_i^2\quadand
D   = \sum_{\frac{d}{2en}\le 
j
\le d - \frac{d}{2en}
} B_i^2.
\end{equation*}
If the following conditions additionally hold
\begin{itemize}
    \item $d \ge K_m\times n\times  \B(\frac{C_m}{p(1-p)}\B)^{\frac{2}{m+1}},$
    \item $d\ge K_m \times n \times \B(\frac{B^2_u}{p(1-p)}\B)^{\frac{2}{u}}$ for all $2\le u \le m,$
    \item $d\ge K_m \times n \times \B(\frac{B^2_{d-u}}{p(1-p)}\B)^{\frac{2}{u}}$ for all $2\le u \le m,$
\end{itemize}
then
\begin{align*}
        & \KL( \hypercubegraph\| \ergraph)\\ & \le  K_m\times \frac{n^3}{p^2(1-p)^2}\times 
\left(
\sum_{i = 1}^m
\frac{B_i^4}{d^i} + 
\sum_{i = d-m}^d
\frac{B_i^4}{d^i} 
+ \frac{C^2_m}{{d}^{m+1}} + 
{D^2}\times \exp\left( - \frac{d}{2en}\right)
\right).
\end{align*}
\end{theorem}

Before we proceed with the proof, we make several remarks. First, when one applies the above theorem, typically the four conditions are reduced to just the much simpler $d\ge K_m\times n.$ This is the case, for example, when $\sigma$ is symmetric. Indeed, note that $\sum_{i = 1}^dB_i^2 = \Var[\sigma]\le p(1-p)$ in that case, so the three inequalities are actually implied by $d\ge K_m\times n.$ This also holds for ``typical'' $\{0,1\}$-valued connections as we  
will see in \cref{sec:typicaldense}.

Next, we explain the expression bounding the $\KL$ divergence. In it, we have separated the Fourier levels 
into essentially $2m+3$ intervals:
\begin{itemize}
    \item $I_u = \{u\}$ for $u \in [m],$
    \item $I_{m+u} = \{d-u\}$ for $u \in [m],$
    \item $I_{2m+1} = [m+1, \frac{d}{2en}]\cup 
    [d - \frac{d}{2en}, d-m-1]
    ,$
    \item $I_{2m+2} = (\frac{d}{2en}, d-\frac{d}{2en}),$
    \item $I_{2m+3} = \{d\}.$
\end{itemize}
There are two reasons for this. First, as we will see in \cref{sec:symmetricconnectionsapplications,sec:nolowdegreeandhighdegreeterms}, levels very close to 0 and $d,$ i.e indexed by $O(1)$ or $d - O(1)$ play a fundamentally different role than the rest of the levels. This explains why $I_1, I_2, \ldots I_{2m}, I_{2m+3}$ are handled separately and why $I_{2m+1}$ and $I_{2m+2}$ are further separated (note that if $m$ is a constant, levels $m, m+1,\ldots, 2m$ in $I_{2m+1}$ also have indices of order $O(1)$). There is a further reason why $I_{2m+2}$ is held separately from the rest. On the respective levels, significantly stronger hypercontractivity bounds hold (see \cref{thm:symmetricpolybounds}).

While the above theorem is widely applicable, as we will see in \cref{sec:applications}, it has an unfortunate strong limitation - it requires that $d = \Omega_m(n)$ even when the other three inequalities are satisfied for much smaller values of $d$ (for an example, see \cref{sec:nolowdegreeandhighdegreeterms}). In \cref{sec:improvingautocorrelation}, we show that this is not a consequence of the complex techniques for bounding moments of $\gamma$ that we use. In fact, any argument based on \cref{claim:RaczLiuIndistinguishability} applied to  $\{0,1\}$-valued connections over the hypercube requires $d = \Omega(n).$
In particular, this means that one needs other tools to approach \cref{conj:sphericalhardthresholds} in its analogue version on the hypercube when $p  = o(n^{-2/3})$ and our conjecture \cref{conj:nolowandhighdegreeterms}. In \cref{thm:entropyargument}, we illustrate another, independent of proof techniques, reason why $d = \Omega(n)$ is necessary for $\{0,1\}$-valued connections with expectation of order $\Theta(1).$

Finally, we end with a brief overview of the proof, which is split into several sections.
\begin{enumerate}
    \item First, in \cref{sec:symmetrizing}, we ``symmetrize'' $\sigma$ by increasing the absolute values of its Fourier coefficients, so that all coefficients on the same level are equal. This might lead to a function which does not take values in $[0,1].$ As we will see, this does not affect our argument.
    \item Then, in \cref{sec:convexitysplitting},  using a convexity argument, we bound the moments of the symmetrized function by the moments on several (scaled) elementary symmetric polynomials which depend on the intervals $I_1, I_2, \ldots, I_{2m+3}$ described above.
    \item Next, in \cref{sec:elemntarysymmetric}, we bound separately the moments of elementary symmetric polynomials using hypercontractivity tools. We prove new, to the best of our knowledge, bounds on the moments of elementary symmetric polynomials in certain regimes.
    \item Finally, we put all of this together to prove
    \cref{thm:maintheoremindistingishability} in \cref{sec:puttingitalltogether}.
\end{enumerate}
\subsection{Symmetrizing The Autocorrelation Function}
\label{sec:symmetrizing}
In light of \cref{eq:firsteqonindist} and \cref{claim:RaczLiuIndistinguishability}, we need to bound the moments of $\gamma.$ Our first observation is that we can, as far as bounding $k$-th moments is concerned, replace $\gamma$ by a symmetrized version. 

\begin{proposition} 
\label{prop:symmetrizing}
Suppose that $\gamma, B_1, B_2, \ldots, B_d$ are as in \cref{thm:maintheoremindistingishability} and consider the function 
$$
f(\bfg):=\sum_{\emptyset \subseteq S\subsetneq [d]} \frac{B_{|S|}^2}{\binom{d}{|S|}}\omega_S(\bfg).
$$
Then, for any $k \in \mathbb{N},$ it holds that 
$\displaystyle
0 \le \expect_\bfg[\gamma(\bfg) ^k]\le \expect_\bfg[f(\bfg)^k].$
\end{proposition}
\begin{proof} Using \cref{prop:fourierviewindistinguishability}, we know that 
\begin{align*}
\expect_\bfg[\gamma(\bfg) ^k] & = 
 \expect \left(
 \sum_{\emptyset\subsetneq S\subset [d]}\widehat{\sigma}^2(S)\omega_S(\bfg)
\right)^k\\
& = \sum_{S_1, S_2, \ldots, S_k}
\widehat{\sigma}^2(S_1)
\widehat{\sigma}^2(S_2)
\cdots
\widehat{\sigma}^2(S_k)
\expect[\omega_{S_1}(\bfg)
\omega_{S_2}(\bfg)
\cdots 
\omega_{S_k}(\bfg)].
\end{align*}
Now, observe that $\expect[\omega_{S_1}(\bfg)
\omega_{S_2}(\bfg)
\cdots 
\omega_{S_k}(\bfg)]$ takes only values $0$ and 1, which immediately implies non-negativity of $\expect[\gamma^k].$ Let $\mathcal{ECV}$ be the set of $k$-tuples $(S_1, S_2, \ldots, S_k)$ on which it is 1 (this is the set of even covers). Using that $|\widehat{\sigma}^2(S)|\le B_{|S|}^2 \binom{d}{|S|}^{-1}$ holds for all $S\neq \emptyset$ by the definition of $B_i,$ we conclude that
\begin{equation}
\label{eq:endofsymmetrizing}
    \begin{split}
\expect_\bfg[\gamma(\bfg) ^k] 
& = \sum_{(S_1, S_2, \ldots, S_k)\in \mathcal{ECV}}
\widehat{\sigma}^2(S_1)
\widehat{\sigma}^2(S_2)
\cdots
\widehat{\sigma}^2(S_k)\\
& \le 
\sum_{(S_1, S_2, \ldots, S_k)\in \mathcal{ECV}}
\frac{B_{|S_1|}^2}{\binom{d}{|S_1|}}
\frac{B_{|S_2|}^2}{\binom{d}{|S_2|}}
\cdots
\frac{B_{|S_k|}^2}{\binom{d}{|S_k|}} = 
\expect_\bfg[f(\bfg) ^k]. 
    \end{split}
\end{equation}
The last equality holds for the same reason as the first one in \cref{eq:endofsymmetrizing}.
\end{proof}

Since $f$ is symmetric, we can express it in a more convenient way as a weighted sum of the elementary symmetric polynomials $e_i.$ That is,  
\begin{equation}
\label{eq:asasumofelemnraysymmetric}
    f(\bfg) = 
    \sum_{i = 1}^d \frac{B_i^2}{\binom{d}{i}}e_i(\bfg).
\end{equation}

Note that if $\sigma$ is symmetric, we have not incurred any loss at this step since $\gamma = f.$ Moving forward, we will actually bound the $L_k$-norms of $f,$ i.e. use the fact that $\expect[f^k]\le \expect[|f|^k] = \norm{f}_k^k.$ There is no loss in doing so when $k$ is even. One can also show that there is no loss (beyond a constant factor in the KL bound in \cref{thm:maintheoremindistingishability}) for $k$ odd as well due to the positivity of $\expect[f^k].$ Considering the expression in \cref{eq:firsteqonindist}, we have 
\begin{equation}
\label{eq:symmetrizednormbbound}
    \begin{split}
     \expect_{(x,y)\sim \mathcal{D}\times\mathcal{D}}\left[\left(1 + \frac{\gamma(\bfx,\bfy)}{p(1-p)}\right)^k\right] = 
    \sum_{t = 0}^k \binom{k}{t}\frac{\expect_{\bfg\sim \unif(\hypercube)}[\gamma(\bfg)^k]}{p^t(1-p)^t} & = 1 + \sum_{t \ge 2}\binom{k}{t}\frac{\expect[\gamma^k]}{p^t(1-p)^t}\\
    & \le  1 + \sum_{t \ge 2}\binom{k}{t}\frac{\expect[f^k]}{p^t(1-p)^t}\\
    & \le 1 + \sum_{t \ge 2}\binom{k}{t}\frac{\norm{f}_k^k}{p^t(1-p)^t},
    \end{split}
\end{equation}
where we used that $\expect[\gamma] = \widehat{\gamma}(\emptyset) = 0$ by \cref{prop:fourierviewindistinguishability}.

\subsection{Splitting into Elementary Symmetric Polynomials}
\label{sec:convexitysplitting}
Our next step is to use a convexity argument to show that when bounding \cref{eq:symmetrizednormbbound}, we simply need to consider the case when all the mass is distributed  on constantly many levels.
The key idea is to view \cref{eq:symmetrizednormbbound} as a function of the coefficients on different levels. To do so, define 
\[
    h(W^1, W^2, \ldots, W^d) := 
    1 + \sum_{t \ge 2}\binom{k}{t}p^{-t}(1-p)^{-t}
    \norm{\sum_{i = 1}^d \frac{W^i}{\binom{d}{i}} e_i(\bfg)}^k_k.
\]
Clearly, $(B^2_1, B^2_2, \ldots, B^2_d) = (W^1, W^2, \ldots, W^d)$ corresponds to $f$ and $h(B_1^2, B_2^2, \ldots, B_d^2)$ is the expression in the last line of \cref{eq:symmetrizednormbbound}. The high-level idea is that $h$ is convex and, thus, its maximum on a convex polytope is attained at a vertex of a polytope. We will construct polytopes in such a way that vertices correspond to specific vectors of constant sparsity.

\begin{proposition}
\label{prop:reductionbyconvexity}
The function $h$ is convex. 
\end{proposition}
\begin{proof}
We simply show that $h$ is a composition of convex functions. First, for any $t,$
$$\displaystyle({W}^i)_{i \in [d]}\longrightarrow
\bigg\|
\sum_{i =1}^d
\frac{{W}^{i}}{\binom{d}{i}}e_i(g)\bigg\|_t
$$ is convex. Indeed, for any $\alpha\in [0,1]$ and vectors $({W}_1^i)_{i \in [d]}$ and 
$({W}_2^i)_{i \in [d]},$
we have 
\begin{align*}
     & \bggn{
\sum_{i \in [d]}
\frac{\alpha{W}_1^{i} + (1-\alpha){W}_2^{i}}{\binom{d}{i}}e_i
}\bggn_t\\
&\le   \bggn{
\alpha\sum_{i \in [d]}
\frac{{W}_1^{i}}{\binom{d}{i}}e_i
}\bggn_t + 
\bggn{
(1-\alpha)\sum_{i \in I[d]}
\frac{{W}_2^{i}}{\binom{d}{i}}e_i
}\bggn_t\\
& =  \alpha\bggn{\sum_{i \in [d]}
\frac{{W}_1^{i}}{\binom{d}{i}}e_i
}\bggn_t + 
(1-\alpha)\bggn{
\sum_{i \in [d]}
\frac{{W}_2^{i}}{\binom{d}{i}}e_i
}\bggn_t,
\end{align*}
by convexity of norms. Since $x\longrightarrow x^t$ for $t \ge 2$ is increasing and convex on $[0,+\infty),$ it follows that $\displaystyle({W}^i)_{i \in I}\longrightarrow 
\bggn{
\sum_{i \in [d]}
\frac{{W}^{i}}{\binom{d}{i}}e_i(g) 
}\bggn^t_t
$ is also convex. Thus, $h$ is convex as a linear combination with positive coefficients of convex functions.
\end{proof}

Now, we will define suitable polytopes on which to apply the convexity argument as follows. For any partition size $\Xi\in [d],$ partition $\mathbf{I} = (I_\xi)_{\xi =1}^{\Xi}$ of $[d]$ into $\Xi$ disjoint sets, and non-negative vector
$\mathbf{w} = (w_\xi)_{\xi = 1}^\Xi,$
define the following product of simplices  $\Delta_{\mathbf{I}, \mathbf{w}}.$ It is given by 
$0 \le {W}^i$ for all $i \in [d]$ and $\sum_{i \in I_\xi}
{W}^i\le w_\xi$ for all $\xi \in [\Xi].$ Using convexity of $h$ on $\Delta_{\mathbf{I}, \mathbf{w}},$ we reach the following conclusion.

\begin{corollary}
\label{cor:convexityonsimplex}
For any $\Delta_{\mathbf{I}, \mathbf{w}}$ defined as just above, 
\[
\sup_{({W}^{i})_{i = 1}^d\in \Delta_{\mathbf{I}, \mathbf{w}}}
h\left(({W}^{i})_{i = 1}^d\right) 
\le \sup_{i_\xi \in I_\xi \;\forall \xi \in [\Xi]}
1 + \sum_{t = 2}^k   \binom{k}{t}\frac{\Xi^t}{p^t(1-p)^t}
\bigg(
\sum_{\xi \in [\Xi]}
\bggn{
\frac{w_{\xi}}{\binom{d}{i_\xi}}e_{i_\xi}
}\bggn_t^t  
\bigg).
\]
\end{corollary}
\begin{proof}
Since $h$ is convex by the previous theorem, 
its supremum over the convex polytope $\Delta_{\mathbf{I}, \mathbf{w}}$ is attained at a vertex.  Vertices correspond to vectors 
$({V}^i)_{i = 1}^d$ of the following type. For each $\xi \in [\Xi],$ there exists a unique index $i_\xi \in I_\xi$ such that $V^{i_\xi}\neq 0$ and, furthermore, $V^{i_\xi} = w_{\xi}.$ Thus, 
$$\sup_{({W}^{i})_{i = 1}^d\in \Delta_{\mathbf{I}, \mathbf{w}}}
h\left(({W}^{i})_{i = 1}^d\right) 
\le \sup_{i_\xi \in I_\xi \;\forall \xi \in [\Xi]}
1 + \sum_{t = 2}^k   \binom{k}{t}p^{-t}(1-p)^{-t}
\bggn{
\sum_{\xi \in [\Xi]}
\frac{w_{\xi}}{\binom{d}{i_\xi}}e_{i_\xi}
}\bggn_t^t. 
$$
Using triangle inequality and the simple inequality 
$(\sum_{\xi \in [\Xi]}u_\xi)^t\le \Xi^t\sum_{\xi \in [\Xi]}u^t_\xi$ for non-negative reals $u_\xi,$
we complete the proof
\[
\bggn{
\sum_{\xi \in [\Xi]}
\frac{w_{\xi}}{\binom{d}{i_\xi}}e_{i_\xi}
}\bggn^t_t\le 
\bigg(\sum_{\xi \in [\Xi]}\bggn{
\frac{w_{\xi}}{\binom{d}{i_\xi}}e_{i_\xi}
}\bggn_t
\bigg)^t\le 
\Xi^t
\sum_{\xi \in [\Xi]}\bggn{
\frac{w_{\xi}}{\binom{d}{i_\xi}}e_{i_\xi}
}\bggn^t_t.
\qedhere\]
\end{proof}

Given the statement of \cref{thm:maintheoremindistingishability}, it should be no surprise that we will apply the intervals\linebreak $I_1, I_2, \ldots, I_{2m+3}$ defined right after \cref{thm:maintheoremindistingishability}. In doing so, however, we will need to bound the norms of elementary symmetric polynomials which appear in \cref{cor:convexityonsimplex}. This is our next step.
\subsection{Moments of Elementary Symmetric Polynomials}
\label{sec:elemntarysymmetric}
When dealing with norms of low-degree polynomials, one usually applies \cref{lem:hypercontractivity} for $q = 2$ due to the simplicity of the formula for second moments. If we directly apply this inequality in our setup, however, we obtain
$$\displaystyle\norm{e_i(g)}_t\le 
\norm{(t-1)^{i/2}e_i(g)}_2 = 
(t-1)^{i/2}
\binom{d}{i}^{1/2}.
$$ 
This inequality becomes completely useless for our purposes when we consider large values of $i.$ Specifically, suppose that 
$i = d/2.$ Then, this bound gives $(t-1)^{d/4}2^{\Theta(d)},$ which is  too large if $t = \omega(1),$ since 
$$
\norm{e_i(\bfg)}_t \le 
\norm{e_i(\bfg)}_{\infty}\le 
\binom{d}{i}\le 2^d.
$$
Even though the $L_\infty$ bound is already much better, it is still insufficient for our purposes. Instead, we obtain the following better bounds in \cref{thm:symmetricpolybounds}. To the best of our knowledge, the second of them is novel. 

\begin{theorem}
\label{thm:symmetricpolybounds} Suppose that $d$ and $t$ are integers such that $d>2te$ and $t>1.$ Take any $s\in [0,d]$ and define $\sm(s) = \min (s,d-s).$ Then, we have the following bounds: 
\begin{itemize}
    \item For any $s,$ 
    $$\norm{e_s}_t\le \binom{d}{s}^{1/2}(t-1)^{\frac{\sm(s)}{2}}.$$
    \item If, furthermore, $\sm(s)\ge \frac{d}{2te},$ then
    $$\norm{e_s}_t\le \binom{d}{s}\exp\left( - \frac{d}{4te}\right).$$
    Finally, this second bound is tight up to constants in the exponent.
\end{itemize}
\end{theorem} 
\begin{proof}
First, we show that we can without loss of generality assume $s \le d/2$ and, hence, $\sm(s) = s.$ Indeed, note that for any function $q$ on the hypercube and any character $\omega_T,$ clearly
$\norm{q\omega_T}_k = \expect[|q\omega_T|^k]^{1/k} = 
\expect[|q|^k]^{1/k} = \norm{q}_k.
$ Thus, 
$$
\norm{e_s}_t = 
\Bgn{\omega_{[d]}\sum_{S\; : \; |S| = s}\omega_{S}}\Bgn_t = 
\Bgn{\omega_{[d]}\sum_{S\; : \; |S| = s}\omega_{S}}\Bgn_t = 
\Bgn{\sum_{S\; : \; |S| = s}\omega_{\bar{S}}}\Bgn_t = 
\norm{e_{d-s}}_t.
$$
With this, we are ready to prove the first bound, which simply uses the hypercontractivity inequality \cref{lem:hypercontractivity}. Indeed, when $s<d/2,$
$$
\norm{e_s}_t\le \sqrt{t-1}^s\norm{e_s}_2 = 
(t-1)^{\frac{\sm(s)}{2}}\binom{d}{s}^{1/2}.
$$
We now proceed to the much more interesting second bound. Suppose that $s\ge \frac{d}{2et}$ and chose $\psi$ such that $s - \psi = \frac{d}{2et}$ (for notational simplicity, we assume $\frac{d}{2et}$ is an integer). The problem of directly applying \cref{lem:hypercontractivity} is that the degree $s$ of the polynomial $e_s$ is too high. We reduce the degree as follows.
\begin{align*}
\norm{e_s}_t  & =\Bgn{
    \sum_{S \; : \; |S| = t}\omega_S
    }\Bgn_t\\
& =  
\bggn{
\sum_{T \; : \; |T| = \psi}\sum_{A\; : \; |A| = s-\psi, A\cap T = \emptyset}\omega_{A\cup T}\binom{s}{\psi}^{-1}
}\bggn_t\\
&\le 
\sum_{T \; : \; |T| = \psi}\bggn{
\sum_{A\; : \; |A| = s-\psi, A\cap T = \emptyset}
\omega_{A}
\omega_{T}
\binom{s}{\psi}^{-1}
}\bggn_t\\
& = 
\sum_{T \; : \; |T| = \psi}\bggn{
\sum_{A\; : \; |A| = s-\psi, A\cap T = \emptyset}
\omega_{A}
\binom{s}{\psi}^{-1}
}\bggn_t.
\end{align*}
We used the following facts.
Each $S$ can be represented in $\binom{s}{\psi}$ ways as $A\cup T,$ where $A$ and $T$ are disjoint and $|T| = \psi.$ Second, as $A$ and $T$ are disjoint, 
$\omega_{A\cup T} = \omega_A\omega_T.$ Finally, we used that for any polynomial $q,$ $\norm{\omega_Tq}_t = \norm{q}_t.$ Now, as the polynomial in the last equation is of degree $s-\psi = \frac{d}{2et},$ which is potentially much smaller than $s,$ we can apply \cref{lem:hypercontractivity} followed by Paresval's identity to obtain
\begin{align*}
& \sum_{T \; : \; |T| = \psi}\bggn{
\sum_{A\; : \; |A| = s-\psi, A\cap T = \emptyset}
\omega_{A}
\binom{s}{\psi}^{-1}
}\bggn_t \\
&\le  
\sum_{T \; : \; |T| = \psi}
(t-1)^{\frac{s-\psi}{2}}
\bggn{
\sum_{A\; : \; |A| = s-\psi, A\cap T = \emptyset}
\omega_{A}
\binom{s}{\psi}^{-1}
}\bggn_2\\
& = 
\sum_{T \; : \; |T| = \psi}
(t-1)^{\frac{s-\psi}{2}}
\sqrt{
\sum_{A\; : \; |A| = s-\psi, A\cap T = \emptyset}
\binom{s}{\psi}^{-2}
}\\
& = 
\sum_{T \; : \; |T| = \psi}
(t-1)^{\frac{s-\psi}{2}}
\sqrt{
\binom{d-\psi}{s-\psi}
\binom{s}{\psi}^{-2}
}_2\\
& =
\binom{d}{\psi}
(t-1)^{\frac{s-\psi}{2}}
\sqrt{
\binom{d-\psi}{s-\psi}
\binom{s}{\psi}^{-2}
}_2\\
& =
\binom{d}{s}
\sqrt{
(t-1)^{s-\psi}
\binom{d}{\psi}^2
\binom{d-\psi}{s-\psi}
\binom{s}{\psi}^{-2}
\binom{d}{s}^{-2}
}.
\end{align*}
We now use the simple combinatorial identity that 
$\displaystyle \binom{d}{\psi}
\binom{d-\psi}{s-\psi} = 
\binom{s}{\psi}
\binom{d}{s}.$
One can either prove it by expanding the factorials or note that both sides of the equation count pairs of disjoint subsets $(T,A)$ of $[d]$ such that $|T| =\psi, |A| = s-\psi.$ In other words, we have

\begin{align*}
&\binom{d}{s}
\sqrt{
(t-1)^{s-\psi}
\binom{d}{\psi}^2
\binom{d-\psi}{s-\psi}
\binom{s}{\psi}^{-2}
\binom{d}{s}^{-2}
}\\
& =  
\binom{d}{s}
\sqrt{
(t-1)^{s-\psi}
\binom{d-\psi}{s-\psi}^{-1}
}\le 
\binom{d}{s}
\sqrt{
(t-1)^{s-\psi}
\left(\frac{d-\psi}{s-\psi}\right)^{-(s-\psi)}
}\\ 
& =
\binom{d}{s}
\left(\frac{(t-1)(s-\psi)}{d-\psi}\right)^{\frac{s-\psi}{2}}\le 
\binom{d}{s}
\left(\frac{2t(s-\psi)}{d}\right)^{\frac{s-\psi}{2}} =
\binom{d}{s}\exp\left( -\frac{d}{4et} \right),
    \end{align*}
as desired. Above, we used the bounds $\displaystyle\binom{d-\psi}{s-\psi}\ge \left(\frac{d-\psi}{s-\psi}\right)^{(s-\psi)}$ and 
$d- \psi \ge d-s\ge d/2,$ and the choice of $\psi$ given by $\displaystyle s - \psi = \frac{d}{2et}.$ With this, the proof of the bounds is complete.

All that remains to show is the tightness of the second bound. Note, however, that 
$\prob[e_s(\bfg) = \binom{d}{s}]= \frac{1}{2^d}$
as this event happens when $\bfg_i = 1$ holds for all $i.$
Therefore,
\[
\norm{e_s}_t \ge 
\left( \frac{1}{2^d}\binom{d}{s}^t\right)^{1/t} = 
\binom{d}{s}
2^{d/t} = 
\binom{d}{s}\exp\left(- O\left(\frac{d}{t}\right)\right).
\qedhere\]
\end{proof}

We make the following remarks about generalizing the inequality above. They are of purely theoretical interest as we do not need the respective inequalities in the paper.

\begin{remark}
The exact same bounds and proof also work if we consider weighted versions with coefficients in $[-1,1].$ That is, polynomials of the form 
$$
q(\bfg) = \sum_{S\; : \; |S| = s}\varepsilon_S\omega_S(\bfg),
$$
where $\varepsilon_S\in [-1,1]$ holds for each $S.$ This suggests that as long as $q$ has a small spectral infinity norm (i.e., uniform small bound on the absolute values of its Fourier coefficients), one can improve the classical theorem \cref{lem:hypercontractivity}. We are curious whether this phenomenon extends beyond polynomials which have all of their weight on a single level.\end{remark}

\begin{remark}
    When $\frac{d}{2et}<s < \frac{d}{2},$ the same inequality with the exact same proof also holds when we evaluate $e_s$ at other random vectors such as standard Gaussian $\mathcal{N}(0, I_d)$ and the uniform distribution on the sphere since analogous inequalities to \cref{lem:hypercontractivity} hold (see \cite[Proposition 5.48, Remark 5.50]{AliceBobBanach}).
\end{remark}
\subsection{Putting It All Together}
\label{sec:puttingitalltogether}
Finally, we combine all of what has been discussed so far to prove our main technical result.

\begin{proof}[Proof of \cref{thm:maintheoremindistingishability}]
First, we note that without loss of generality, we can make the following assumptions:
\begin{enumerate}
    \item $B^4_1 = O_m(\frac{dp^2(1-p)^2}{n^3}),$
    \item $B^4_{d-1} = O_m(\frac{dp^2(1-p)^2}{n^3}),$
    \item $B^4_d = O_m(\frac{p^2(1-p)^2}{n^3}),$
    \item $D^2 = O_m\left(\exp(\frac{d}{2en})\times 
    \frac{p^2(1-p)^2}{n^3}
    \right).$
\end{enumerate}
Indeed, if one of those assumptions is violated, the bound on the KL-divergence that we are trying to prove is greater than 2 (for appropriate hidden constants depending only on $m$). The resulting bound on total variation via Pinsker's inequality is more than 1 and is, thus, trivial.

We now proceed to apply \cref{cor:convexityonsimplex} with the intervals $I_1, I_2, \ldots, I_{2m+3}$ and the following weights.
\begin{itemize}
    \item $I_u = \{u\}$ and $w_u = B_u^2$ for $u \in [m],$
    \item $I_{m+u} = \{d-u\}$ and 
    $w_{m+u} = B_{d-u}^2$ for $u \in [m],$
    \item $I_{2m+1} = [m+1, \frac{d}{2en}]\cup 
    [d - \frac{d}{2en}, d-m-1]
    $ and $w_{2m+1} = C_m,$
    \item $I_{2m+2} = [\frac{d}{2en}, d-\frac{d}{2en}]$ and $w_{2m+2} = D,$
    \item $I_{2m+3} = \{d\}$ and $w_{2m+3} = B_d^2.$
\end{itemize}
The defined weights vector $\mathbf{w}$ exactly bounds the weight of $f$ on each respective level. Thus, the supremum of $h$ over $\Delta_{\mathbf{I}, \mathbf{w}}$ is at least as large as the expression in \cref{eq:symmetrizednormbbound} which gives an upper bound on the total variation between the two random graph models of interest. 
Explicitly,  using \cref{eq:firsteqonindist}
and \cref{cor:convexityonsimplex}, we find 

\begin{align*}
    \label{eq:firsteqinmainindist}
     &1 + \sum_{t = 2}^k   \binom{k}{t}p^{-t}(1-p)^{-t}\expect[\gamma^t]\\
    &\le 1 + \sum_{t = 2}^k   \binom{k}{t}p^{-t}(1-p)^{-t}\norm{f}_t^t\\
    &\le   1 + 
    \sup_{i_u\in I_{u}\; \forall u \in [2m+3]}
    \sum_{u \in [2m+3]}
    \sum_{t = 2}^k
    \binom{k}{t}p^{-t}(1-p)^{-t}(2m+3)^t\bggn{\frac{w_u}{\binom{d}{u}} e_{i_u}}\bggn_t^t.
    \end{align*}

We now consider different cases based on $u$ in increasing order of difficulty.

\paragraph{Case 1)} When $u = 2m+3.$ Then, $i_u = d,$ so $\norm{e_{i_u}}_t = \norm{e_d}_t = 1.$ We have 
\begin{align*}
        &\sum_{t = 2}^k
    \binom{k}{t}p^{-t}(1-p)^{-t}(2m+3)^t\norm{w_u e_{i_u}}_t^t\\
    & =
     \sum_{t = 2}^k
    \binom{k}{t}p^{-t}(1-p)^{-t}(2m+3)^tB_d^{2t}\\
    & = (1+\psi)^k - 1 - k\psi,
    \end{align*}
where $\displaystyle \psi = \frac{2m+3}{p(1-p)}B^2_d = o\left(\frac{1}{n}\right) = o\left(\frac{1}{k}\right)$ by the assumption on $B_d$ at the beginning of the proof. It follows that 
$$
(1+\psi)^k - 1 - k\psi\le 
e^{1+\psi k} - 
1 - 
\psi k 
= 
O(\psi^2k^2).
$$
Thus, for $u = 2m+3,$ 
\begin{equation}
\label{eq:highestterm}
    \begin{split}
    &\sum_{t = 2}^k
    \binom{k}{t}p^{-t}(1-p)^{-t}(2m+3)^t\norm{w_u e_{i_u}}_t^t =
    O_m\left(
    k^2\frac{B_d^4}{p^2(1-p)^2}
    \right).
    \end{split}
\end{equation}

\paragraph{Case 2)} When $u = 2m+2.$
We bound 
$ \sum_{t = 2}^k   \binom{k}{t}p^{-t}(1-p)^{-t}
    (2m+3)^t 
    \norm{\frac{w_{2m+2}}{\binom{d}{i_{2m+2}}}e_{i_{2m+2}}}^t_t.
$
Note that for $i_u\in I_{2m+2},$ we have $\sm(i)\ge \frac{d}{2ne}.$ Thus, 
$\norm{e_{i_u}}_t\le \norm{e_{i_u}}_n \le 
\binom{d}{i_u}\exp(-\frac{d}{4en})
$ by \cref{thm:symmetricpolybounds}. It follows that
\begin{align*}
    &\sum_{t = 2}^k   \binom{k}{t}p^{-t}(1-p)^{-t}
    (2m+3)^t
\frac{D^t}{\binom{d}{i_u}^t}\norm{e_{i_u}
}_t^t\\
& \le 
\sum_{t = 2}^k \binom{k}{t}p^{-t}(1-p)^{-t}
(2m+3)^t
\frac{D^t}{\binom{d}{i_u}^t}\binom{d}{i_u}^t
\exp\left( - \frac{dt}{4en}\right)\\
& =
\left(1 + \frac{D(2m+3)}{p(1-p)}\exp\left(-\frac{d}{4en}\right)\right)^k  - 
1 - k \frac{D(2m+3)}{p(1-p)}\exp\left(-\frac{d}{4en}\right).
\end{align*}
Note, however, that $$\psi: =  \frac{D(2m+3)}{p(1-p)}\exp\left(-\frac{d}{4en}\right) = o\left(\frac{1}{n}\right)
$$
by the assumption on $D$ at the beginning of the proof. It follows that 
$$(1 + \psi)^k - 1 - \psi k = O(\psi^2k^2).$$ Thus, for $u = 2m+2,$
\begin{align*}
    &\sum_{t = 2}^k   \binom{k}{t}p^{-t}(1-p)^{-t}
    (2m+3)^t 
    \bggn{\frac{w_{2m+2}}{\binom{d}{i_u}}e_{i_u}}\bggn^t_t=
O_m\left(
\frac{D^2}{p^2(1-p)^2}\exp\Big(-\frac{d}{2en}\Big)
\right)
    \end{align*}
for a large enough hidden constant depending only on $m.$

\paragraph{Case 3)} When $u = 2m+1.$ Again, 
we assume without loss of generality that $i_u \le d/2$ and we want to bound 
$$ \sum_{t = 2}^k
    \binom{k}{t}p^{-t}(1-p)^{-t}(2m+3)^t\bggn{\frac{w_{2m+1}}{\binom{d}{i_{2m+1}}} e_{i_u}}\bggn_t^t.$$ Using \cref{thm:symmetricpolybounds}, we know that $\displaystyle\norm{e_{i_{2m+1}}}_t\le \binom{d}{i_{2m+1}}^{1/2}(t-1)^{\frac{i_{2m+1}}{2}},$ so
\begin{align*}
&\sum_{t = 2}^k   \binom{k}{t}p^{-t}(1-p)^{-t}(2m+3)^t
\frac{C_m^t}{\binom{d}{i_{2m+1}}^t}\norm{e_{i_{2m+1}}
}_t^t\\
& \le 
\sum_{t = 2}^k   \binom{k}{t}p^{-t}(1-p)^{-t}(2m+3)^t
\frac{C_m^t}{\binom{d}{i_{2m+1}}^t}\binom{d}{i_{2m+1}}^{t/2}
(t-1)^{i_{2m+1}t/2}.\\
    \end{align*}
This time we show exponential decay in the summands. More concretely, we prove the following inequality
\[
\binom{k}{t}
\left(
\frac{(2m+3)C_m}{p(1-p)\binom{d}{i_{2m+1}}^{1/2}}\right)^t
(t-1)^{i_{2m+1}t/2}\ge 
\frac{e}{2}
\binom{k}{t+1}
\left(
\frac{(2m+3)C_m}{p(1-p)\binom{d}{i_{2m+1}}^{1/2}}\right)^{t+1}
t^{i_{2m+1}(t+1)/2}\,.
\]
This reduces to showing that 
\begin{align*}
        \binom{d}{i_{2m+1}}^{1/2}\ge 
        \frac{(2m+3)e}{2}\frac{k-t}{t+1}
        \frac{C_m}{p(1-p)}
        t^{i_{2m+1}/2}\left(\frac{t}{t-1}\right)^{ti_{2m+1}/2}.
    \end{align*}
Using that $\left(\frac{t}{t-1}\right)^{t}\le 4,\binom{d}{i_{2m+1}}\ge \left(\frac{d}{i_{2m+1}}\right)^{i_{2m+1}},
$ and $n \ge k-t \ge k,$
we simply need to prove that 
\begin{equation}
    \label{eq:powersofdandn}
    \begin{split}
        \left(\frac{d}{i_{2m+1}}\right)^{i_{2m+1}/2}\ge
        \frac{(2m+3)e}{2} 
        n \frac{C_m}{p(1-p)} t^{i_{2m+1}/2-1}4^{i_{2m+1}/2}.
    \end{split}
\end{equation}
Since $t \le k \le n,$ the right-hand side is maximized when $t = n,$ so we simply need to show that
\begin{align*}
     \left(\frac{d}{i_{2m+1}}\right)^{i_{2m+1}/2}&\ge 
        \frac{(2m+3)e}{2} 
        n^{i_{2m+1}/2}4^{i_{2m+1}/2}\frac{C_m}{p(1-p)}\Longleftrightarrow\\
     d &\ge  i_{2m+1}\times n\times 4 \times \left(\frac{(2m+3)eC_m}{2p(1-p)}\right)^\frac{2}{i_{2m+1}}.     
\end{align*}
Now, the right-hand side has form $KxM^{\frac{1}{x}}:=Z(x) $ for $x = i_{2m+1}.$ The derivative of \linebreak $\log Z(x)$ is 
$\frac{1}{x}(1-\frac{1}{x}\log M),$ which has at most one positive root, at $x = \log M.$ Thus, the maximum of $Z(x)$ on $[m+1, \frac{d}{2en}]$ is at either $m+1$ or $\frac{d}{2en}.$ Substitution at these values, we obtain the following expressions for the right-hand side. First, when $i_{2m+1} = m+1,$
$$
O_m\left(n \times \left(\frac{C_m}{p(1-p)}\right)^{\frac{2}{m+1}}\right) \le d, 
$$
where the inequality holds from the assumptions in the theorem for large enough constant $K_m.$
Second, when $i_{2m+1} = \frac{d}{2en},$
\begin{align*}
    &\frac{2}{e}d \times \left(\frac{(2m+3)eC_m}{2p(1-p)}\right)^\frac{4en}{d}\\
    & = \frac{2}{e}d \times \left(\frac{(2m+3)e}{2}\right)^{\frac{4en}{d}}\times 
    \left(\left(\frac{C}{p(1-p)}\right)^{\frac{2}{m+1}}\right)^{\frac{en(m+1)}{d}}\\
    &\le 
    \frac{2}{e}d \times \left(\frac{(2m+3)e}{2}\right)^{\frac{4en}{d}}\times 
    \left(\frac{1}{K_m} \times \frac{d}{n}\right)^{\frac{en(m+1)}{d}}\\
    &\le 
    \frac{2}{e}d \times \left(\frac{(2m+3)e}{2}\right)^{\frac{4en}{d}}\times 
    K_m^{ - \frac{en(m+1)}{d}}
    e^{e(m+1)}\,,
    \end{align*}
where we used $\left(\frac{d}{n}\right)^{\frac{n}{d}}\le \exp(\left(\left(\frac{d}{n}\right) - 1\right)\frac{n}{d}) \le e.$  Now, for a large enough value of $K_m,$ when $d> K_m n,$ the above expression is at most $d.$
Since the exponential decay holds,
\begin{align*}
&\sum_{t = 2}^k   \binom{k}{t}p^{-t}(1-p)^{-t}
\frac{C_m^t}{\binom{d}{i_{2m+1}}^t}\binom{d}{i_{2m+1}}^{t/2}
(t-1)^{i_{2m+1}t/2}\\
& = 
O\left(
k^2 \frac{C_m^2}{p^2(1-p^2)\binom{d}{i_{2m+1}}}
\right)\\
& = 
O_m\left(
k^2 \frac{C_m^2}{p^2(1-p^2)d^{m+1}}
\right).\end{align*}

\paragraph{Case 4)} When $u \le 2m.$ The cases $u \in [2m]\backslash \{1,m+1\}$ are handled in the exact same way as the previous Case 3), except that we use $w_u = B_u^2$ when $u \in [m]$ and 
$w_u = B^2_{d-u+m}$ when $u \in [m+1, 2m]$ instead of $w_u = C_m.$ The only subtlety occurs when $i_u \in \{1,d-1\}.$ Then, the analogue of
\cref{eq:powersofdandn} becomes 
$$
d^{1/2}\ge 
O_m\left(n \frac{w_u}{p(1-p)}t^{-1/2}\right).
$$
This time, however, the right-hand side is maximized for $t = 2.$ Thus, we need 
$$
d = \Omega_m\left(n^2 \frac{w_u^2}{p^2(1-p)^2}\right),
$$
which is satisfied by the assumption made in the beginning of the proof.

Combining all of these cases, we obtain that for any fixed $k,$
\begin{align*}
   &1 + \sum_{t = 2}^k   \binom{k}{t}p^{-t}(1-p)^{-t}\expect[\gamma^t]\\ 
   & =  
   1 + O_m\Bigg(
   k^2
   \sum_{i = 1}^m
   \frac{B_i^4}{p^2(1-p)^2 d^i} + 
   k^2
   \sum_{i = 0}^m
   \frac{B_{d-i}^4}{p^2(1-p)^2 d^i}
   + 
   k^2 \frac{C_m^2}{p^2(1-p)^2 d^{m+1}} + 
   \frac{D^2}{p^2(1-p)^2}\exp\bigg( - \frac{d}{en}\bigg)
   \Bigg).
   \end{align*}
Thus, the logarithm of this expression is at most 
$$
K_m\Bigg(
   k^2
   \sum_{i = 1}^m
   \frac{B_i^4}{p^2(1-p)^2 d^i} + 
   k^2
   \sum_{i = 0}^m
   \frac{B_{d-i}^4}{p^2(1-p)^2 d^i} + 
   k^2 \frac{C_m^2}{p^2(1-p)^2 d^{m+1}} + k^2
   \frac{D^2}{p^2(1-p)^2}\exp\bigg( - \frac{d}{en}\bigg)
   \Bigg)
$$
for an appropriately chosen $K_m.$ Summing over $k\in [n-1],$ we arrive at the desired conclusion.
\end{proof}

\section{Applications of the Main Theorem}
\label{sec:applications}
As already discussed, part of the motivation behind the current work is to extend the ongoing effort towards understanding random geometric graphs to settings in which the connection is not necessarily monotone or symmetric. We remove the two assumptions separately. 
Initially, we retain symmetry and explore the implications of \cref{thm:maintheoremindistingishability} in existing monotone and novel non-monotone setups. Then, we focus on two specific classes of non-monotone connections which are obtained via simple modifications of symmetric connections. These two classes will be of interest later in \cref{sec:detection} as they lead to the failure of the commonly used signed-triangle statistic, and in \cref{sec:improvingautocorrelation} as they rigorously demonstrate an inefficiency in \cref{claim:RaczLiuIndistinguishability}.
Finally, we go truly beyond symmetry by studying ``typical'' dense indicator connections.
\subsection{Symmetric Connections}
\label{sec:symmetricconnectionsapplications}
Consider \cref{thm:maintheoremindistingishability} for symmetric functions. Then, $B_i^2 = \weight{i}{\sigma}\le \Var[\sigma]\le p(1-p).$ Similarly, the inequalities $C_m \le p(1-p), D\le p(1-p)$ hold. In particular, this means that as long as $d = \Omega_m(n),$ all other conditions are automatically satisfied. Now, if $d = \Omega(n\log n)$ for a large enough hidden constant, the exponential term disappears. If we set $m = 2,$ the term $\frac{n^3}{p^2(1-p)^2}\frac{C_m^2}{d^{m+1}}\le \frac{n^3}{d^3} = o_n(1)$ also disappears. We have the following corollary.

\begin{corollary}
\label{cor:5relevantlevelsforsuperlinearithmic}
There exists an absolute constant $K$ with the following property. Suppose that $d >Kn\log n.$ Let $\sigma:\hypercube\longrightarrow [0,1]$ be a symmetric connection. Then,
\begin{align*}
        & \KL( \hypercubegraph\| \ergraph)\\ 
        &\quad=   O\left(
        \frac{n^3\weight{d}{\sigma}^2}{p^2(1-p)^2} + 
        \frac{n^3(\weight{1}{\sigma}^2 + \weight{d-1}{\sigma}^2)}{p^2(1-p)^2d} +
        \frac{n^3(\weight{2}{\sigma}^2 + \weight{d-2}{\sigma}^2)}{p^2(1-p)^2d^2}\right) +
        o(1).
\end{align*}
\end{corollary}

For symmetric connections in the regime $d = \Omega(n\log n),$ only Fourier weights on levels $1,2,d-1,d-2,d-1,d$ matter. In \cref{sec:nolowdegreeandhighdegreeterms} we conjecture how to generalize this phenomenon.

Now, we proceed to concrete applications of \cref{cor:5relevantlevelsforsuperlinearithmic}. In them, we simply need to bound the weights of $\sigma$ on levels $1,2,d-1,d-2,d.$ All levels besides $d$ are easy since they have a normalizing term polynomial in $d$ in the denominator. However, more care is needed to bound $\weight{d}{\sigma}.$  As we will see in \cref{sec:fluctuations}, this is not just a consequence of our proof techniques, but an important phenomenon related to the ``fluctuations'' of $\sigma.$ In the forthcoming sections, we use the following simple bound, the proof of which is deferred to \cref{appendix:corwithparity}. 

\begin{lemma}
    \label{lem:lastfourier}
    Suppose that $f:\hypercube\longrightarrow \mathbb{R}$ is symmetric.
    Denote by $f_i$ the value of $f$ when exactly $i$ of its arguments are equal to $1$ and the rest are equal to $-1.$ Then,
\[
|\widehat{f}([d])| \le
\frac{1}{2}
        \expect_{j \sim Bin(d-1, \frac{1}{2})} |f_{j+1} - f_j|  = 
        O\left(\frac{\sum_{j = 1}^d|f_{j+1} - f_j|}{\sqrt{d}}\right).
\]
\end{lemma}
We remark that the quantity $\sum_{j = 1}^d|f_{j+1} - f_j|$ is a direct analogue of the \textit{total variation} appearing in real analysis literature \cite[p.117]{SteinSakarchi}, which measures the ``fluctuations'' of a function. As the name ``total variation'' is already taken by a more central object in our paper, we will call this new quantity \textit{fluctuation} instead. We denote it by $\fluctuation(f):= \sum_{j = 1}^d|f_{j+1} - f_j|.$ 
We believe that the interpretation of $|\widehat{\sigma}([d])|$ as a measure of fluctuations may be useful in translating our results to the Gaussian and spherical setups when non-monotone connections are considered. We discuss this in more depth in \cref{sec:fluctuations}.

\subsubsection{Lipschitz Connections}
As the main technical tool for our indistinguishability results is \cref{claim:RaczLiuIndistinguishability}, we first consider Lipschitz connections to which Racz and Liu originally applied the claim \cite{Liu2021APV}. We show that their result \cref{thm:liuraczlipschitzconnections}
can be retrieved using our framework even if one drops the monotonicity assumption. The first step is to bound 
 $|\widehat{f}([d])|.$

\begin{proposition}
\label{prop:lipschitzcorrwithparity}
If $f:\hypercube\longrightarrow \mathbb{R}$ is a symmetric $L$-Lipschitz function with respect to the Hamming distance, meaning that 
$$|f(x_1, x_2, \ldots, x_{i-1}, 1, x_{i+1}, \ldots, x_n) - 
f(x_1, x_2, \ldots, x_{i-1}, -1, x_{i+1}, \ldots, x_n)|\le L
$$ holds for all $i$ and $x_1, x_2, \ldots, x_{i_1}, x_i, \ldots, x_n,$
then
$|\widehat{f}([d])| \le \frac{L}{2}.$
\end{proposition}
\begin{proof} The proof is immediate given \cref{lem:lastfourier}. Note that 
$$
|f_{j + 1} - f_j| = 
|f(\underbrace{1,1,\ldots, 1}_{j+1},\underbrace{-1, \ldots, -1}_{d-j+1}) - 
f(\underbrace{1,1,\ldots, 1}_{j},\underbrace{-1, \ldots, -1}_{d-j})|\le L
$$
holds for all $j.$
\end{proof}

Using this statement, we obtain the following result.

\begin{corollary}
\label{cor:generallipschitzindist}
Suppose that $\sigma$ is a symmetric $\frac{1}{r\sqrt{d}}$-Lipschitz connection. If
$d = \omega_p(\frac{n^3}{r^4})$ and $d = \Omega(n\log n),$
$$
\KL( \hypercubegraph\|\ergraph) = o(
1
).
$$
\end{corollary}
\begin{proof}
We will apply \cref{cor:5relevantlevelsforsuperlinearithmic}.
By the well-known bounded difference inequality, \cite[Claim 2.4]{RvH} we know that since $\sigma$ is $\frac{1}{r\sqrt{d}}$-Lipschitz, $\Var[\sigma] \le \frac{1}{r^2},$ so each $\weight{i}{\sigma}$ is bounded by $\frac{1}{r^2}.$ On the other hand, $\weight{d}{\sigma}$ is bounded by $\frac{1}{r^2d}$ by \cref{prop:lipschitzcorrwithparity}. It follows that the expression in \cref{cor:5relevantlevelsforsuperlinearithmic} is bounded by 
\[
O\left(\frac{n^3}{dr^4} + 
\frac{n^3}{d^2r^4} + 
\frac{n^3}{d^2r^4}
\right) = o(1).
\qedhere\]
\end{proof}

The flexibility of our approach allows for a stronger result when the map $\sigma$ is even.
\begin{corollary}
\label{cor:evenlipschitzindist}
Suppose that $\sigma$ is a symmetric even $\frac{1}{r\sqrt{d}}$-Lipschitz connection in the sense that
$\sigma(\bfx) = \sigma(-\bfx)$ also holds. If $d$ is even, $d = \Omega(n\log n),$ and 
$d = \omega_p(\frac{n^{3/2}}{r^2}),$ then
$$
\KL\Big( \hypercubegraph\|\ergraph\Big) = o(
1
).
$$
\end{corollary}
\begin{proof}
The proof is similar to the one in the previous proposition. However, we note that when $\sigma$ is even, all Fourier coefficients on odd levels vanish. In particular, this holds for levels $1$ and $d-1.$ Thus, the expression \cref{cor:5relevantlevelsforsuperlinearithmic} is bounded by 
\[
O\left( 
\frac{n^3}{d^2r^4} + 
\frac{n^3}{d^2r^4}
\right) = o(1).
\qedhere\]
\end{proof}

\begin{remark}
    \normalfont
    The stronger results for even connections observed in this section is a general phenomenon, which will reappear in the forthcoming results. It is part of our more general conjecture in \cref{sec:nolowdegreeandhighdegreeterms}. We note that the requirement that $d$ is even is simply needed to simplify exposition. Otherwise, one needs to additionally bound $\weight{d-1}{\sigma}$ which is a standard technical task (see \cref{sec:furtherfouriercomp}).
\end{remark}

\subsubsection{Hard Threshold Connections}
\label{sec:hardthresholdindistinguishbility}
We now move to considering the classical setup of hard thresholds. Our tools demonstrate a striking new phenomenon if we make just a subtle twist -- by introducing a ``double threshold'' -- to the original model.

We start with the basic threshold model.
Let $\tau_p$ be such that $\prob_{\bfx\sim \unif(\hypercube)}[\sum_{i}x_i\ge \tau_p] = p.$ Let $\threshold_p$ be the map defined by $\indicator\left[\sum_{i= 1}^n x_i \ge \tau_p\right].$ Again, we want to understand the Fourier coefficient on the last level. Using that there exists a unique $j$ such that $f_{j+1}\neq f_j$ for $f = \threshold_p$ in \cref{lem:lastfourier}, in \cref{appendix:corwithparity} we obtain the following bound.

\begin{proposition}
\label{prop:simplethresholdcorwithparity}
$\widehat{\threshold_p}([d]) = O\Big(\frac{p\sqrt{\log\frac{1}{p}}}{\sqrt{d}}\Big).$
\end{proposition}

\begin{corollary}
\label{cor:classicrggindistinguishability}
 Suppose that $d = \Omega(n\log n),d = \omega(n^3p^2\log^2\frac{1}{p}),$ and $p = d^{O(1)}.$ Then,
$$
\KL\Big(\RAG(n,\hypercube,p,\threshold_{p})\|\ergraph\Big) = o(
1
).
$$
\end{corollary}
\begin{proof}
From \cref{thm:levelkinequalities}, we know that $\weight{1}{\sigma}  = O(p^2\log \frac{1}{p}), 
\weight{d-1}{\sigma}  = O(p^2\log \frac{1}{p}), $
$\weight{2}{\sigma}  = O(p^2\log^2 \frac{1}{p}),$ and $ 
\weight{d-2}{\sigma}  = O(p^2\log^2 \frac{1}{p}). 
$ From \cref{prop:simplethresholdcorwithparity}, 
$\weight{d}{\sigma} = 
O(d^{-1}p^2\log\frac{1}{p}).$ Thus, the expression in \cref{cor:5relevantlevelsforsuperlinearithmic} becomes
$$
O\left(
\frac{n^3 p^2\log^2\frac{1}{p}}{d} + 
\frac{n^3 p^2\log^4\frac{1}{p}}{d^2} + 
\frac{n^3 p^2\log^2\frac{1}{p}}{d^2}
\right) = o(1),
$$
where we used the fact that $p$ is polynomially bounded in $d.$
\end{proof}

Remarkably, this matches the state of the art result for 
$p = \omega (\frac{1}{n})$
of \cite{Liu2022STOC}. As before, if we consider an even version of the problem, we obtain a stronger result.
Let $\delta_p$ be such an integer that
$\prob[|\sum^n_{i = 1}x_i|\ge \delta_p] = p.$ Consider the``double threshold'' connection given by the indicator $\doublethreshold_p(\bfx) = \indicator\left[ |\sum^n_{i = 1}x_i|\ge \delta_p\right].$ In \cref{appendix:corwithparity}, we prove the following inequality.
\begin{proposition}
\label{prop:doublethresholdcorwithparity}
$\widehat{\doublethreshold}_p([d]) = O\Big(\frac{p\sqrt{\log\frac{1}{p}}}{\sqrt{d}}\Big).$
\end{proposition}

\begin{corollary}
\label{cor:doublethresholdconnectionsindist}
Suppose that $d$ is even, $d = \Omega(n\log n),d = \omega(n^{3/2}p\log^2\frac{1}{p}),$ and $p = d^{O(1)}.$  Then,
$$
\KL\Big(\RAG(n,\hypercube,p,\doublethreshold_{p})\|\ergraph\Big) = o(1).
$$
\end{corollary}
\begin{proof}
As in the proof of \cref{cor:evenlipschitzindist}, we simply note that when $d$ is even, the weights on levels 1 and $d-1$ vanish as $\doublethreshold_p$ is an even map. Thus, we are left with 
\[
O\left(
\frac{n^3 p^2\log^4\frac{1}{p}}{d^2} + 
\frac{n^3 p^2\log^2\frac{1}{p}}{d^2}
\right) = o(1).
\qedhere\]
\end{proof}

Note that the complement of $\RAG(n,\hypercube,p,\doublethreshold_{p})$ is the random algebraic graph defined with respect to the connection $\sigma(\bfx): = \indicator[\sum^n_{i = 1}x_i\in (-\delta_p, \delta_p)].$ This random graph is very similar to the one used in \cite{Khot14}, where the authors consider the interval indicator $\indicator\left[\langle \bfx_i,\bfx_j\rangle\in [\tau_1, \tau_2]\right],$ albeit with latent space the unit sphere. The utility of interval indicators in \cite{Khot14} motivates us to ask:
What if $\sigma$ is the indicator of a union of more than one interval? 
\subsubsection{Non-Monotone Connections, Interval Unions, and Fluctuations}
\label{sec:fluctuations}
To illustrate what we mean by fluctuations, consider first the simple connection for $\kappa\in [0,1/2],$ given by
\begin{equation}
    \label{eq:pureparity}
    \pi_\kappa(\bfx):=
\frac{1}{2} + \kappa \omega_{[d]}(\bfx).
\end{equation}
When viewed as a function of the number of ones in $\bfx,$ this function ``fluctuates'' a lot for large values of $\kappa.$ Note that $\pi_\kappa$ only takes values ${1}/{2}\pm \kappa$ and alternates between those when a single coordinate is changed. In the extreme case $\kappa = {1}/{2},$ $\pi_\kappa$ is the indicator of all $\bfx$ with an even number of coordinates equal to $-1.$ Viewed as a function of $\sum_{i = 1}^d x_i,$ this is 
a union of $d/2 + O(1)$ intervals. A simple application of \cref{claim:RaczLiuIndistinguishability} gives the following proposition. 

\begin{corollary}
Whenever ${n\kappa^{4/3}} = o(1),$ 
$\displaystyle
\KL\Big(\RAG(n, \hypercube, 1/2, \pi_\kappa), \ergraphhalf\Big) = o(1).$
\end{corollary}
\begin{proof}
Let $\gamma_\kappa$ be the autocorrelation of $\pi_\kappa - \frac{1}{2}.$ Clearly, 
$\gamma_\kappa = \kappa^2\omega_{[d]}(\bfx),$ so $\expect[\gamma_\kappa^t] = 0$ when $t$ is odd and 
$\expect[\gamma_\kappa^t] = \kappa^{2t}$ when $t$ is even. Since $n\kappa^2 = o(1)$ holds as ${n\kappa^{4/3}} = o(1),$ we obtain
\begin{align*}
     &\sum_{k = 0}^{n-1}
    \log \expect \left[\left(1 + \frac{\gamma_\kappa(\bfx)}{p(1-p)}\right)^k\right]\\
    & \le 
    \sum_{k = 0}^{n-1}
    \log \left( 1 + 
    \sum_{t = 2}^{k}
    \binom{k}{t}\frac{\kappa^{2t}}{4^t}
    \right)\\
    & = \sum_{k = 0}^{n-1}
    \log \left( \left(1 + \frac{\kappa^2}{4}\right)^k -
    k \frac{\kappa^2}{4}\right)\\
    & = \sum_{k = 0}^{n-1} \log \left(1 + O\left(k^2 \frac{\kappa^4}{16}\right)\right)
    =  O\left(
    \sum_{k = 0}^{n-1}k^2\times \kappa^4
    \right) = O(n^3\kappa^4),
\end{align*}
from which the claim follows.
\end{proof}

We will explore tightness of these bounds later on. The main takeaway from this example for now is the following - when the connection ``fluctuates'' more, it becomes more easily distinguishable. Next, we show that this phenomenon is more general. Recall the definition of fluctuations $\fluctuation$ from \cref{lem:lastfourier}. We begin  with two simple examples of computing the quantity $\fluctuation.$

\begin{example}
\label{ex:motoneanalytictv}
\normalfont
If $\sigma:\{\pm 1\}^d\longrightarrow [0,1]$ is monotone, then $\fluctuation(\sigma)\le 1.$ Indeed, this holds as all differences $f_{i+1} - f_i$ have the same sign and, so, $\sum_{i = 0}^{d-1} |f_{i+1} - f_i| = 
|\sum_{i = 0}^{d-1} f_{i+1} - f_i| = |f_d - f_0|\le 1.$
\end{example}

\begin{example}
\label{ex:intervalunions}
\normalfont
Suppose that $I = [a_1, b_1]\cup [a_2, b_2]\cup\cdots\cup [a_s, b_s],$ where these are $s$ disjoint intervals with integer endpoints ($b_i +1<a_{i+1}$). Then, if 
$\sigma(\bfx) = \indicator[\sum_{i=1}^d x_i\in I],$ it follows that 
$\fluctuation(f)\le 2s$ as $|f_{b_{i}+1} - f_{b_i}| = 1, |f_{a_{i}} - f_{a_i-1}| = 1$ and all other marginal differences equal $0.$
\end{example}

We already saw in \cref{lem:lastfourier} that $|\widehat{f}([d])| = O\left(\frac{\fluctuation(\sigma)}{\sqrt{d}}\right).$
Using the same proof techniques developed thus far, we make the following conclusion.

\begin{corollary}
\label{cor:fluctuationindist}
Suppose that $\sigma$ is a symmetric connection. Whenever $d = \Omega(n\log n)$ and \linebreak
$\displaystyle 
d= \omega\left(\max\left(n^3p^2\log^{2}\frac{1}{p}, n^{3/2}\fluctuation(\sigma)^2p^{-1}\right)\right)
$ are satisfied,
$$\displaystyle \KL\Big( \hypercubegraph\|\ergraph\Big) = o(
1
).$$
If $\sigma$ and $d$ are even, the weaker condition 
$\displaystyle 
d= \omega\left(\max\Big(n^{3/2}p\log^2\frac{1}{p}, n^{3/2}\fluctuation(\sigma)^2p^{-1}\Big)\right),
$
is sufficient to imply
$\displaystyle \KL\Big( \hypercubegraph\|\ergraph\Big) = o(
1
).$
\end{corollary}

\begin{example}
\normalfont
A particularly interesting case is when $\sigma$ is the indicator of $s$ disjoint intervals as in \cref{ex:intervalunions}. We obtain indistinguishability rates by replacing $\fluctuation(\sigma)$ with $2s$ above. Focusing on the case when $p = \Omega(1)$ (so that the dependence on the number of intervals is clearer), we
need 
\[
    d = \omega(n^3 + n^{3/2}s^2)
\]
for indistinguishability for an arbitrary union of intervals and 
\[
    d = \omega(n^{3/2}s^{2})
\]
for indistinguishability for a union of intervals symmetric with respect to $0.$

We note that both of these inequalities make sense only when $s = o(\sqrt{d}).$ The collapse of this behaviour at $s \approx \sqrt{d}$
is not a coincidence as discussed in \cref{rmk:morethansqrtdintervals}.
\end{example}
\subsubsection{Conjecture for Connections Without Low Degree and High Degree Terms}
\label{sec:nolowdegreeandhighdegreeterms}
We now present the following conjecture, which generalizes the observation that random algebraic graphs corresponding to connections $\sigma$ satisfying $\sigma(\bfx) = \sigma(-\bfx)$ are indistinguishable from $\ergraph$ for smaller dimensions $d.$
Let $k\ge 0$ be a constant.  
Suppose that the symmetric connection $\sigma:\hypercube\longrightarrow[0,1]$ satisfies $\weight{i}{\sigma}=0$ whenever $i < k$ and $\weight{i}{\sigma} = 0$ whenever $i > d-k.$ Consider the bound on the $\KL$ divergence between $\hypercubegraph$ and $\ergraph$ in \cref{thm:maintheoremindistingishability} for $m = k+1.$
Under the additional assumption $d = \Omega(n\log n),$ it becomes of order 
$$
\frac{n^3B^2_m}{p^2(1-p)^2d^m} + 
\frac{n^3 C_m}{p^2(1-p)^2d^{m+1}}\le 
\frac{n^3\Var[\sigma]}{p^2(1-p)^2d^m}+
\frac{n^3 \Var[\sigma]}{p^2(1-p)^2d^{m+1}} = 
O\left(\frac{n^3}{d^{m}}\right).
$$
In particular, this suggests that whenever $d = \omega(n^{3/m}),$ $\KL\Big(\hypercubegraph\|\ergraph\Big) = o(1).$ Unfortunately, in our methods, there is the strong limitation $d = \Omega(n\log n).$ 

\begin{conjecture}
\label{conj:nolowandhighdegreeterms}
Suppose that $m\in \mathbb{N}$ is a constant and $\sigma$ is a symmetric connection such that 
the weights on levels $\{1,2,\ldots,m-1\}\cup\{d-m+1, \ldots, d\}$ are all equal to 0. If $d = \tilde{\Omega}(n^{3/m}),$ then 
$$
\TV\Big(\hypercubegraph, \ergraph\Big) = o(1).
$$
\end{conjecture}

Unfortunately, our methods only imply this conjecture for $m= \{1,2,3\}$ as can be seen from \cref{thm:maintheoremindistingishability} as we also need $d = \Omega(n\log n).$  In \cref{sec:detectionnolowandhighdegreeterms}, we give further evidence for this conjecture by proving a matching lower bound.
Of course, one should also ask whether there exist non-trivial symmetric functions satisfying the property that their Fourier weights on levels $\{1,2,\ldots,m-1\}\cup\{d-m+1, \ldots, d\}$ are all equal to 0. In \cref{appendix:functionswithoutlowandhighdegreeterms},
 we give a positive answer to this question.

\subsubsection{Low-Degree Polynomials}
\label{sec:lowdegreepolyindist}
We end with a brief discussion of low-degree polynomials, which are a class of particular interest in the field of Boolean Fourier analysis (see, for example, \cite{Eskenazis_21,ODonellBoolean}). The main goal of our discussion is to complement \cref{conj:nolowandhighdegreeterms}. One could hastily conclude from the conjecture that it is only the ``low degree'' terms of $\sigma$ that make a certain random algebraic graph $\hypercubegraph$ distinct from an \ER graph for large values of $d.$ This conclusion, however, turns out to be incorrect - symmetric connections which are low-degree polynomials, i.e. everything is determined by the low degrees, also become indistinguishable from $\ergraph$ for smaller dimensions $d.$ Thus, a more precise intuition is that the interaction between low-degree and high-degree is what makes certain random algebraic graphs $\hypercubegraph$ distinct from \ER graphs for large values of $d.$

\begin{corollary}
\label{thm:lowdegreeindist}
Suppose that $k\in \mathbb{N}$ is a fixed constant and $\sigma:\hypercube\longrightarrow [0,1]$ is a symmetric connection of degree $k.$ If $d =\Omega(n)$ and $d = \omega_k\left(\min(n^3p^2\log^2 \frac{1}{p}, np^{-2/3}) \right),$ then
$$
\KL\Big(\hypercubegraph\| \ergraph\Big) = o(1).
$$
\end{corollary}
\begin{proof}
We apply \cref{cor:5relevantlevelsforsuperlinearithmic} for $m = k$ and obtain 
$$
\KL\Big(\hypercubegraph\| \ergraph\Big)  = O\left( \frac{n^3\weight{1}{\sigma}}{p^2(1-p)^2d} + 
\frac{n^3\weight{2}{\sigma}}{p^2(1-p)^2d^2}
\right).
$$
Now, we use \cref{cor:lowdegreepolyweights,thm:levelkinequalities} to bound the right-hand side in two ways by, respectively:
\[
 O\left(\frac{n^3}{p^2(1-p)^2d^3}\right)\quad\text{and}\quad O\left(\frac{n^3p^2\log^2\frac{1}{p}}{d}\right).
\]
The conclusion follows by taking the stronger bound of the two.
\end{proof}
 
\subsection{Transformations of Symmetric Connections}
\label{sec:transformationsofsymmetric}
One reason why \cref{thm:maintheoremindistingishability} is flexible is that it does not explicitly require $\sigma$ to be symmetric, but only requires a uniform bound on the \textit{absolute values} of Fourier coefficients on each level. This allows us, in particular, to apply it to modifications of symmetric connections. In what follows, we present two ways to modify a symmetric function such that the claims in \cref{cor:generallipschitzindist,cor:evenlipschitzindist,cor:classicrggindistinguishability,cor:doublethresholdconnectionsindist,cor:fluctuationindist} still hold. The two  modifications are both interpretable and will demonstrate further important phenomena in \cref{sec:doom}.

It is far from clear whether previous techniques used to prove indistinguishability for random geometric graphs can handle non-monotone, non-symmetric connections. In particular, approaches based on analysing Wishart matrices (\cite{Bubeck14RGG,Brennan19PhaseTransition}) seem to be particularly inadequate in the non-symmetric case as the inner-product structure inherently requires symmetry. The methods of Racz and Liu in \cite{Liu2021APV} on the other hand, exploit very strongly the representation of connections via CDFs which inherently requires monotonicity. 
While our approach captures these more general settings, we will see in \cref{sec:doom} an example demonstrating that it is unfortunately not always optimal.

\subsubsection{Coefficient Contractions of Real Polynomials}
\label{sec:coefficientcontractions}
Suppose that $f:\hypercube\longrightarrow \mathbb{R}$ is given. We know that it can be represented as a polynomial 
$$
f(x_1, x_2, \ldots, x_d) = \sum_{S\subseteq [d]}\widehat{f}(S)\prod_{i \in S}x_i.
$$
We can extend $f$ to the polynomial $\tilde{f}:\mathbb{R}^n \longrightarrow \mathbb{R},$ also given by 
$$
\tilde{f}(x_1, x_2, \ldots, x_d) = \sum_{S\subseteq [d]}\widehat{f}(S)\prod_{i \in S}x_i.
$$
Now, for any vector $\alpha = (\alpha_1, \alpha_2, \ldots, \alpha_d)\in [-1,1]^d,$ we can ``contract'' the coefficients of $\tilde{f}$ by considering the polynomial $$\tilde{f}_\alpha(x_1, x_2, \ldots, x_d):=\tilde{f}(x_1\alpha_1, x_2\alpha_2,\ldots, x_d\alpha_d) = 
\sum_{S\subseteq [d]}\widehat{f}(S)
\prod_{i \in S}\alpha_i\prod_{i \in S}x_i.
$$
Subsequently, this induces the polynomial 
$f_{\alpha}:\hypercube\longrightarrow \mathbb{R}$ given by 
$${f}_\alpha(x_1, x_2, \ldots, x_d) = 
\sum_{S\subseteq [d]}\widehat{f}(S)
\prod_{i \in S}\alpha_i\prod_{i \in S}x_i.
$$
This is a ``coefficient contraction'' since $|\widehat{f_\alpha}(S)|\le |\widehat{f}(S)|$ clearly holds.
Note, furthermore, that when $\bfg\in \hypercube,$ clearly $f_\bfg(\bfx) = f(\bfg\bfx).$ In particular, this means that if $f$ takes values in $[0,1],$ so does $f_\bfg.$ This turns out to be the case more generally.

\begin{observation}
For any function $f:\hypercube\longrightarrow [0,1]$ and vector $\alpha \in [-1,1]^d,$ $f_\alpha$ takes values in $[0,1].$
\end{observation}
\begin{proof}
Note that the corresponding real polynomials $\tilde{f}, \tilde{f}_\alpha$ are linear in each variable. Furthermore, the sets $A = \prod_{i = 1}^d \left[-|\alpha_i|, |\alpha_i|\right], B =  [-1,1]^d$ are both convex and $A\subseteq B.$ Thus,
$$
\sup_{\bfx \in \hypercube} f_\alpha(\bfx) = 
\sup_{\bfy \in B}
\tilde{f}_\alpha(\bfy) = 
\sup_{\bfz \in A}\tilde{f}(\bfz)\le 
\sup_{\bfz \in B}\tilde{f}(\bfz) = 
\sup_{\bfx\in \hypercube} f(\bfx)\le 1.
$$
In the same way, we conclude that $f_\alpha\ge 0$ holds.
\end{proof}

In particular, this means that for any $\alpha \in [-1,1]^d$ and connection $\sigma$ on $\hypercube$, $\sigma_\alpha$ is a well-defined connection. Therefore, all of  \cref{cor:generallipschitzindist,cor:evenlipschitzindist,cor:classicrggindistinguishability,cor:doublethresholdconnectionsindist,cor:fluctuationindist} hold if we replace the respective connections in them with arbitrary coefficient contractions of those connections. As we will see in \cref{sec:doom}, particularly interesting is the case when $\alpha \in \hypercube$ since this corresponds to simply negating certain variables in $\sigma.$ In particular, suppose that $\alpha = (\underbrace{1,1,\ldots, 1}_\ell, \underbrace{-1,-1,\ldots, -1}_{d-\ell})$ for some $1\le \ell\le d.$ Then, 
the connection $\sigma(\bfx,\bfy)$ becomes a function of 
$$
\sum_{i \le \ell}x_i y_i - \sum_{i >\ell}x_i y_i. 
$$
The corresponding two-form is not PSD, so it is unlikely that approaches based on Gram-Schmidt and analysing Wishart matrices could yield meaningful results in this setting.

\subsubsection{Repulsion-Attraction Twists}
Recall that one motivation for studying non-monotone, non-symmetric connections was that in certain real world settings - such as friendship formation - similarity in different characteristics might have different effect on the probability of forming an edge.

To model this scenario, consider the simplest case in which  $d = 2d_1$ and each latent vector has two components of length $d_1,$ that is $\bfx = [\bfy,\bfz],$ where $\bfy, \bfz\in \{\pm 1\}^{d_1}$ and $[\bfy, \bfz]$ is their concatenation. Similarity in the $\bfy$ part will ``attract'', making edge formation more likely,  and similarity in $\bfz$ will ``repulse'', making edge 
formation less likely. Namely, consider two symmetric connections $a:\{\pm 1\}^{d_1}\longrightarrow [0,1]$ and $r:\{\pm 1\}^{d_1}\longrightarrow [0,1]$ with, say, means $1/2.$ We call the following function their \textit{repulsion-attraction} twist 
$$
\rho_{a,r}([\bfy,\bfz]) = \frac{1 + a(\bfy) - r(\bfz)}{2}.
$$
One can check that $\rho_{a,r}$ is, again, a connection with mean $1/2$ and satisfies $\Var[\rho_{a,r}] = \frac{1}{4}(\Var[a] + \Var[r]).$ Furthermore,
$$
(\rho_{a,r} - {1}/{2})*
(\rho_{a,r} - {1}/{2})= 
(a - 1/2)*(a-1/2) + 
(r-1/2)*(r-1/2).
$$
Thus, if we denote $\gamma_\rho = 
\left(\rho_{a,r} - \frac{1}{2}\right)*
\left(\rho_{a,r} - \frac{1}{2}\right),
\gamma_a = (a - 1/2)*(a-1/2), \gamma_r = (r-1/2)*(r-1/2),$ we have $\gamma_\rho  = \gamma_a + \gamma_r.$ Using $\norm{\gamma_a + \gamma_r}_p\le \norm{\gamma_a}_p + \norm{\gamma_r}_p$ for each $p,$ we can repeat the argument in \cref{thm:maintheoremindistingishability} losing a simple factor of 2 for $\gamma_\rho.$ In particular, this means that when $d = \Omega(n\log n),$ we have that for any two symmetric connections $a,r,$
\begin{align*}
         &\KL\Big( \RAG(n, \hypercube, p, \rho_{a,r})\| \ergraph\Big)\\ 
         & =  O \left( 
        {n^3\weight{d}{a}^2}+
        \frac{n^3(\weight{1}{a}^2 + \weight{d-1}{a}^2)}{d} +
        \frac{n^3(\weight{2}{a}^2 + \weight{d-2}{a}^2)}{d^2}\right)\\
        &+ O \left({n^3\weight{d}{r}^2} + 
        \frac{n^3(\weight{1}{r}^2 + \weight{d-1}{r}^2)}{d} +
        \frac{n^3(\weight{2}{r}^2 + \weight{d-2}{r}^2)}{d^2}\right)+
        o(1).
    \end{align*}
as in \cref{cor:5relevantlevelsforsuperlinearithmic}. One can now combine with \cref{cor:generallipschitzindist,cor:evenlipschitzindist,cor:classicrggindistinguishability,cor:doublethresholdconnectionsindist,cor:fluctuationindist}. As we will see in \cref{sec:doom}, particularly interesting from a theoretical point of view is the case when $a = r.$

One could also consider linear combinations of more than two connections. We do not pursue this here.

\subsection{Typical Indicator Connections}
\label{sec:typicaldense}
Finally, we turn to a setting which goes truly beyond symmetry. Namely, we tackle the question of determining when a ``typical'' random algebraic graph with an indicator connection is indistinguishable from $\ergraph.$ More formally, we consider
$\RAG(n, \hypercube, \sigma_A,p_A)$ for which $\sigma_A(\bfg) = \indicator[\bfg \in A],$ where $A$ is sampled from $\anteuniform(\hypercube,p).$ Recall that $\anteuniform(\hypercube,p)$ is the distribution over subsets of $\hypercube$ in which every element is included independently with probability $p$ as all elements over the hypercube have order at most 2 (see \cref{def:anteuniform}).
\begin{theorem} 
\label{thm:typicaldense}
Suppose that $d = \Omega(n\log n).$ With probability at least $1 - 2^{-d}$ over $A\sim \anteuniform(\hypercube, p),$
$$
\TV\Big(\RAG(n, \hypercube, \sigma_A,p_A)\|\ergraph\Big) = o_n(1).
$$
\end{theorem}

As we will see in \cref{sec:polydetection}, the dependence of $d$ on $n$ is tight up to the logarithmic factor when $p = \Omega(1).$
We use the following well-known bound on the Fourier coefficients of random functions. 

\begin{claim}[\cite{ODonellBoolean}]
\label{claim:fouriercoefftypicalindicators} 
Suppose that $B$ is sampled from $\hypercube$ by including each element independently with probability $p$ (i.e., $B \sim \anteuniform(\Group,p)$). Let $\sigma_B$ be the indicator of $B.$ Then, 
with probability at least $1 - {2^{-d}},$ the inequalities 
$|\widehat{\sigma_B}(\emptyset) - p|\le 2^{1-d/2}\sqrt{d}$ and 
$|\widehat{\sigma_B}(S)|\le 2^{1-d/2}\sqrt{d}$ for all $S\neq \emptyset$ hold simultaneously.
\end{claim}

Now, we simply combine 
\cref{claim:fouriercoefftypicalindicators} and \cref{thm:maintheoremindistingishability}.

\begin{proof}[Proof of \cref{thm:typicaldense}]
We will prove the claim for all sets $A\subseteq\hypercube$ satisfying that 
$|\widehat{\sigma_A}(\emptyset) - p|\le \frac{2\sqrt{d}}{2^{d/2}}$ and 
$|\widehat{\sigma_A}(S)|\le \frac{2\sqrt{d}}{2^{d/2}}$ holds for all $S\neq\emptyset.$ In light of \cref{claim:fouriercoefftypicalindicators}, this is enough.
Let $A$ be such a set and let $p_A = \expect[\sigma_A] = \widehat{\sigma_A}(\emptyset).
$
Consider \cref{thm:maintheoremindistingishability}. It follows that $B_i \le \Delta\binom{d}{i}^{1/2}$ for each $i,$ where $\Delta = \frac{2\sqrt{d}}{2^{d/2}}$ for brevity of notation. Take $m = 1.$ Then, 
\begin{align*}
    C_m & = \sum_{2\le i \le \frac{d}{2en}}\Delta^2\binom{d}{i} + 
    \sum_{d-\frac{d}{2en}\le i \le d-2}\Delta^2\binom{d}{i}\\
    & = O\left(
\Delta^2\binom{d}{\frac{d}{2en}}
    \right)\\
    & = 
    O\left(
\frac{4d}{2^d} (2e^2n)^\frac{d}{2en}
    \right)\\
    & =  
    O\left(
    \frac{2^{O(d\log n/n)}}{2^d}
    \right) = O\left(\frac{1}{2^{d/2}}\right).
\end{align*}
Similarly,
\[
    D = \sum_{\frac{d}{2en}<i < d - \frac{d}{2en}}\Delta^2\binom{d}{i}\le \Delta^22^d  = 4d.
\]
Since $d = \Omega(n\log n),$ the conditions of \cref{thm:maintheoremindistingishability} are satisfied. Thus, 
\begin{align}
\label{eq:KLtopaergraph}
&\KL\Big(\RAG(n, \hypercube, 1/2, \sigma_A)\|\mathsf{G}(n,p_A)\Big)\\
& =  O\left( \frac{n^3}{p_A^2}\times \left(
\frac{B_1^4}{d} + 
\frac{B_{d-1}^4}{d} + 
B_d^4  + 
\frac{O(\frac{1}{2^d})}{d^2} + 
{4d}\exp\left( - \frac{d}{2en}\right)
\right)\right)\\
& =  O\left(
\frac{n^3d^3}{p_A^22^{2d}} + 
\frac{n^3}{p_A^2d^22^d} + 
n^3d\times \exp\left( - \frac{d}{2en}\right)
\right).
    \end{align}
The last expression is of order $o_n(1)$
when $d = \Omega(n\log n)$ for a large enough hidden constant. Finally, note that 
$$
\TV\Big(\ergraph, \mathsf{G}(n,p_A)\Big) = o(1)
$$
whenever $p = \Omega({1}/{n^2}), $ and $|p_A - p|\le 2^{1-d/2}\sqrt{d}.$ Indeed, this follows from the following observation. Suppose that $p_A<p$ without loss of generality. Then, $\mathsf{G}(n,p_A)$ can be sampled by first sampling $\ergraph$ and then retaining each edge with probability $\frac{p_A}{p} \ge 1-2^{1-d/2}\sqrt{d}p^{-1} = 1 - o(n^{-3}).$ Thus, with high probability  $(1 - o(n^{-3}))^{\binom{n}{2}} = 1 - o(n^{-1}),$ all edges are retained and the two distributions are a distance $o(1)$ apart in $\TV.$ We use triangle inequality and combine with \cref{eq:KLtopaergraph}.
\end{proof}

\begin{remark}
\label{rmk:constrcutionwithgeularfunctions}
\normalfont
The above proof provides a constructive way of finding connections $\sigma$ for which indistinguishability occurs when $d = \tilde{\Omega}(n).$ Namely, a sufficient condition is that all of the Fourier coefficients of $\sigma_A$ except for $\widehat{\sigma_A}(\emptyset)$ are bounded by $\frac{2\sqrt{d}}{2^{d/2}}$ and $|\widehat{\sigma_A}(\emptyset) - p|\le  \frac{2\sqrt{d}}{2^{d/2}}$ In fact, one could check that the same argument holds even under the looser condition that all coefficients are bounded by $\frac{d^C}{2^d}$ for any fixed constant $C.$ Such functions with Fourier coefficients bounded by $\epsilon$ on levels above $0$ are known as \textit{$\epsilon$-regular} \cite{ODonellBoolean}. Perhaps the simplest examples for $\frac{2\sqrt{d}}{2^{d/2}}$-regular indicators with mean 1/2 are given by the inner product mod 2 function and the complete quadratic function
\cite[Example 6.4]{ODonellBoolean}. These provide explicit examples of functions for which the phase transition to \ER occurs at $d = \tilde{\Theta}(n).$
\end{remark}

\section{Indistinguishability of Typical Random Algebraic Graphs}
We now strengthen \cref{thm:typicaldense} by (1) extending it to all groups of appropriate size and (2) 
replacing the $\log n$ factor by a potentially smaller $\log \frac{1}{p}$ factor. The proof is probabilistic and, thus, does not provide means for explicitly constructing functions 
for which the phase transition occurs at $d = \tilde{\Omega}(n).$

\label{sec:ragindistinguishability}
\begin{theorem}
\label{thm:generalragindistinguishability}
An integer $n$ is given and a real number $p \in [0,1].$ Let $\Group$ be any finite group of size at least $\exp(Cn \log \frac{1}{p})$ for some universal constant $C.$ Then, with probability at least $1 - |\Group|^{-1/7}$ over $A\sim \anteuniform(\Group, p),$ we have 
$$
\KL\Big(\RRAG(n, \Group, \sigma_A, p_A)\| \ergraph\Big) = o(1), 
$$
where $p_A = \expect_{\bfg\sim \unif(\Group)}[\sigma_A(\bfg)].$ The exact same holds for $\LRAG$ instead of $\RRAG.$
\end{theorem}
\begin{proof} In light of \cref{eq:firsteqonindist}, we simply need to show that for typical sets $A,$ the function $$\gamma_A(\bfg) = \expect_{\bfz\sim \unif(\Group)}[(\sigma_A(\bfg\bfz^{-1})-p)(\sigma_A(\bfz)-p)]$$ has sufficiently small moments. We bound the moments via the second moment method over $A.$ Namely, for $t \ge 1,$ define the following functions:
\[
\phi_t(A)  = \expect_{\bfg\sim \unif(\Group)} [\gamma_A(\bfg)^t]\; \text{ and }\;
\Phi(t)  = \expect_{A\sim \anteuniform(\Group,p)}
[\phi_t^2(A)].
\]
We show that for typical $A,$ $\phi_t(A)$ is small by bounding the second moment $\Phi(t)$ and using Chebyshev's inequality. We begin by expressing $\phi_t(A)$ and $\Phi(t).$ First, 
\begin{align*}
\phi_t(A) & =  \expect_{\bfg\sim \unif(\Group)} [\gamma_A(\bfg)^t] 
= \expect_{\bfg\sim \unif(\Group)} \Big[
\expect_\bfz[(\sigma_A(\bfg\bfz^{-1})-p)(\sigma_A(\bfz)-p)]^t
\Big]
\\
& =  \expect_{\bfg, \{\bfz_i\}_{i =1}^t\sim_{iid}\unif(\Group)}
\Bigg[
\prod_{i = 1}^t 
(\sigma_A(\bfg\bfz_i^{-1})-p)(\sigma_A(\bfz_i)-p)
\Bigg].
\end{align*}
Therefore, 
\begin{align*}
\Phi(t) & = \expect_{A\sim \anteuniform(\Group,p)}
 [\phi_t^2(A)]\\
 & = 
\expect_{A\sim \anteuniform(\Group,p)}
\Bigg[
\Bigg(
\expect_{\bfg, \{\bfz_i\}_{i =1}^t\sim_{iid}\unif(\Group)}
\bigg[
\prod_{i = 1}^t 
(\sigma_A(\bfg\bfz_i^{-1})-p)(\sigma_A(\bfz_i)-p)
\bigg]
\Bigg)^2
\Bigg]
\\
& = 
\expect_{A\sim \anteuniform(\Group,p)}
\Bigg[
\expect_{\bfg^{(1)}, \bfg^{(2)}, \{\bfz^{(1)}_i\}_{i =1}^t,\{\bfz^{(2)}_i\}_{i =1}^t\sim_{iid}\unif(\Group)}
\bigg[
\prod_{i = 1}^t 
(\sigma_A(\bfg^{(1)}(\bfz^{(1)}_i)^{-1})-p)
\\
&\qquad \qquad\qquad\qquad \times  
\prod_{i = 1}^t 
(\sigma_A(\bfz^{(1)}_i)-p)
\prod_{i = 1}^t 
(\sigma_A(\bfg^{(2)}(\bfz^{(2)}_i)^{-1})-p)
\prod_{i = 1}^t 
(\sigma_A(\bfz^{(2)}_i)-p)
\bigg]
\Bigg].
\end{align*}
The main idea now is that we can change the order of expectation between $A$ and \linebreak $\{\bfg^{(1)}, \bfg^{(2)}, \{\bfz^{(1)}_i\}_{i =1}^t,\{\bfz^{(2)}_i\}_{i =1}^t\}.$ Since $\expect_{A\sim \anteuniform(\Group,p)}\sigma_A(\bfh) = p$ for each $\bfh\in \Group$ by the definition of the density $\anteuniform(\Group,p),$ each of the terms of the type $(\sigma_A(\cdot)-p)$
has zero expectation over $A.$ The lower bound we have on the size of $\Group$ will ensure that  for ``typical'' choices of $\{\bfg^{(1)}, \bfg^{(2)}, \{\bfz^{(1)}_i\}_{i =1}^t,\{\bfz^{(2)}_i\}_{i =1}^t\},$ at least one of the terms of the type $(\sigma_A(\cdot)-p)$ is independent of the others, which will lead to a (nearly) zero expectation. Namely, rewrite the above expression as 
\begin{equation}
\label{eq:ragindisteq49}
    \begin{split}
\expect_{\bfg^{(1)}, \bfg^{(2)}, \{\bfz^{(1)}_i\}_{i =1}^t,\{\bfz^{(2)}_i\}_{i =1}^t\sim_{iid}\unif(\Group)}
&\Bigg[
\expect_{A\sim \anteuniform(\Group,p)}\bigg[
\prod_{i = 1}^t 
(\sigma_A(\bfg^{(1)}(\bfz^{(1)}_i)^{-1})-p)
\prod_{i = 1}^t 
(\sigma_A(\bfz^{(1)}_i)-p)
\\
&\qquad\qquad\times  
\prod_{i = 1}^t 
(\sigma_A(\bfg^{(2)}(\bfz^{(2)}_i)^{-1})-p)
\prod_{i = 1}^t 
(\sigma_A(\bfz^{(2)}_i)-p)
\bigg]
\Bigg].
\end{split}
\end{equation}
Now, observe that if there exists some $\xi \in \{\bfg^{(1)}, \bfg^{(2)}, \{\bfz^{(1)}_i\}_{i =1}^t,\{\bfz^{(2)}_i\}_{i =1}^t\}$ such that $\xi \neq \zeta, \xi \neq \zeta^{-1}$ for any other $\zeta \in \{\bfg^{(1)}, \bfg^{(2)}, \{\bfz^{(1)}_i\}_{i =1}^t,\{\bfz^{(2)}_i\}_{i =1}^t\},$ the above expression is indeed equal to 0 since this means that $\sigma_A(\xi)$ is independent of all other terms $\sigma_A(\zeta).$ Call $B = B\bigg(\{\bfg^{(1)}, \bfg^{(2)}, \{\bfz^{(1)}_i\}_{i =1}^t,\{\bfz^{(2)}_i\}_{i =1}^t\}\bigg)$ the event that all $\xi \in \{\bfg^{(1)}, \bfg^{(2)}, \{\bfz^{(1)}_i\}_{i =1}^t,\{\bfz^{(2)}_i\}_{i =1}^t\}$ satisfy this property. As discussed, over $B$ we have 
$$
\expect_{A\sim \anteuniform(\Group,p)}\Bigg[
\prod_{i = 1}^t 
(\sigma_A(\bfg^{(1)}(\bfz^{(1)}_i)^{-1})-p)
\prod_{i = 1}^t 
(\sigma_A(\bfz^{(1)}_i)-p)
\prod_{i = 1}^t 
(\sigma_A(\bfg^{(2)}(\bfz^{(2)}_i)^{-1})-p)
\prod_{i = 1}^t 
(\sigma_A(\bfz^{(2)}_i)-p)
\Bigg] = 0,
$$
and outside of $B,$ we have 
\begin{align*}
    \expect_{A\sim \anteuniform(\Group,p)}\Bigg[
\prod_{i = 1}^t 
(\sigma_A(\bfg^{(1)}(\bfz^{(1)}_i)^{-1})-p)
\prod_{i = 1}^t 
(\sigma_A(\bfz^{(1)}_i)-p)
\prod_{i = 1}^t 
(\sigma_A(\bfg^{(2)}(\bfz^{(2)}_i)^{-1})-p)
\prod_{i = 1}^t 
(\sigma_A(\bfz^{(2)}_i)-p)
\Bigg]\\
\le 
\norm{
\prod_{i = 1}^t 
(\sigma_A(\bfg^{(1)}(\bfz^{(1)}_i)^{-1})-p)
\prod_{i = 1}^t 
(\sigma_A(\bfz^{(1)}_i)-p)
\prod_{i = 1}^t 
(\sigma_A(\bfg^{(2)}(\bfz^{(2)}_i)^{-1})-p)
\prod_{i = 1}^t 
(\sigma_A(\bfz^{(2)}_i)-p)}_{\infty}
\le 1.
    \end{align*}
It follows that $\Phi(t)\le 1-\prob[B].$
Next, we will show that $\prob[B] = 1- O\left({|\Group|^{-1/3}}\right).$

Denote by $M \subseteq \Group$ the set of group elements $\bfh$ for which the equation $\bfx^2 = \bfh$ has at least $|\Group|^{1/3}$ solutions $\bfx$ over $\Group.$ By Markov's inequality, $|M|\le |\Group|^{2/3}.$ We now draw the elements $\{\bfg^{(1)}, \bfg^{(2)}, \{\bfz^{(1)}_i\}_{i =1}^t,\{\bfz^{(2)}_i\}_{i =1}^t\}$ sequentially in that order and track the probability with which one of the following conditions fails (we terminate if a condition fails):
\begin{enumerate}
    \item $\bfg^{(1)} \not \in M\cup\{1\}$ and $\bfg^{(2)} \not \in M\cup \{1\}.$
    \item $\bfz^{(1)}_i\not \in \{\bfz^{(1)}_j, (\bfz^{(1)}_j)^{-1}\; : \; j < i \}\cup
    \{(\bfg^{(1)}(\bfz^{(1)}_j)^{-1}, (\bfg^{(1)}(\bfz^{(1)}_j)^{-1})^{-1}\; : \; j < i \}.
    $
    \item $\bfg^{(1)}(\bfz^{(1)}_i)^{-1}\not \in \{\bfz^{(1)}_j, (\bfz^{(1)}_j)^{-1}\; : \; j \le i \}\cup
    \{(\bfg^{(1)}(\bfz^{(1)}_j)^{-1}, (\bfg^{(1)}(\bfz^{(1)}_j)^{-1})^{-1}\; : \; j < i \}.
    $
    \item $\bfz^{(2)}_i\not \in \{\bfz^{(2)}_j, (\bfz^{(2)}_j)^{-1}\; : \; j < i \}\cup
    \{\bfg^{(2)}(\bfz^{(2)}_j)^{-1}, (\bfg^{(2)}(\bfz^{(2)}_j)^{-1})^{-1}\; : \; j < i \},
    $ and\\
    $
    \bfz^{(2)}_i\not \in
    \{\bfz^{(1)}_j,(\bfz^{(1)}_j)^{-1},\bfg^{(1)}(\bfz^{(1)}_j)^{-1}, (\bfg^{(1)}(\bfz^{(1)}_j)^{-1})^{-1}\; : \; j \le n \}.
    $
    \item $\bfg^{(2)}(\bfz^{(2)}_i)^{-1}\not \in \{\bfz^{(2)}_j, (\bfz^{(2)}_j)^{-1}\; : \; j < i \}\cup
    \{\bfg^{(2)}(\bfz^{(2)}_j)^{-1}, (\bfg^{(2)}(\bfz^{(2)}_j)^{-1})^{-1}\; : \; j < i \},
    $ and\\
    $
    \bfz^{(2)}_i\not \in
    \{\bfz^{(1)}_j,(\bfz^{(1)}_j)^{-1},\bfg^{(1)}(\bfz^{(1)}_j)^{-1}, (\bfg^{(1)}(\bfz^{(1)}_j)^{-1})^{-1}\; : \; j \le n \}.
    $
\end{enumerate}
Note that conditions 2-5 exactly capture that event $B$ occurs, so conditions 1-5 hold together with probability at most $\prob[B].$ Condition 1 is added to help with the proof. We analyze each step separately.

\begin{enumerate}
\item $\prob[\bfg^{(1)}\in M\cup \{1\}] = \frac{|M|\cup \{1\}}{|\Group|} = O(|\Group|^{-1/3})$ and the same for $\bfg^{(2)}.$ Thus, the drawing procedure fails at this stage with probability at most $O(|\Group|^{-1/3}).$
\item Since the elements $\bfz^{(1)}_i\not \in \{\bfz^{(1)}_j, (\bfz^{(1)}_j)^{-1}\; : \; j < i \}\cup
    \{(\bfg^{(1)}(\bfz^{(1)}_j)^{-1}, (\bfg^{(1)}(\bfz^{(1)}_j)^{-1})^{-1}\; : \; j < i \}
    $ are independent of $\bfz^{(1)}_{i}$ and there are at most $4i\le 4t \le 4n$ of them, we fail with probability $O({n}/{|\Group|})$ at this step.
\item The calculation for $\bfg^{(1)}(\bfz^{(1)}_i)^{-1}$ is almost the same, except that we have to separately take care of the cases $\bfg^{(1)}(\bfz^{(1)}_i)^{-1} = (\bfz^{(1)}_i)^{-1}$ and $\bfg^{(1)}(\bfz^{(1)}_i)^{-1} = \bfz^{(1)}_i.$ The first event can only happen if $\bfg^{(1)}  =1,$ in which case the drawing procedure has already terminated at Step 1. The second event is equivalent to $\bfg^{(1)} = (\bfz^{(1)}_i)^2.$ If the drawing procedure has not terminated at Step 1, this occurs with probability at most $|\Group|^{-2/3}$ since $\bfg^{(1)}\not \in M.$ Thus, the total probability of failure is 
$O ({n}/{|\Group|} + {|\Group|^{-2/3}}).$
\item As in Step 2, the calculation for $\bfz^{(2)}_i$ gives $O({n}/{|\Group|}).$
\item As in Step 5, the calculation for $\bfg^{(1)}(\bfz^{(1)}_i)^{-1}$ gives $O ({n}/{|\Group|} + {|\Group|^{-2/3}}).$
\end{enumerate}
Summing over all $j \in [t], \ell \in \{1,2\},$ via union bound, we obtain a total probability of failure at most 
$$
O\bigg(\frac{1}{|\Group|^{1/3}}\bigg) + O\bigg(\frac{t}{|\Group|^{2/3}}\bigg) + 
O\bigg(\frac{tn}{|\Group|}\bigg) = 
O\bigg(\frac{1}{|\Group|^{1/3}}\bigg)
$$
as $t\le n \le \log |\Group|$ by the definition of $t$ and assumption on $|\Group|.$

Going back, we conclude that 
$1- \prob[B]\le {C}{|\Group|^{-1/3}}$ for some absolute constant $C,$ so
$
\Phi(t) \le {C}{|\Group|^{-1/3}}
$ for some absolute constant $C.$

Now, we can apply the moment method over $A$ and conclude that 
$$
\prob_A[|\phi_t(A)|>|\Group|^{-1/12}]\le
\frac{\expect_A [\phi_t(A)^2]}{|\Group|^{-1/6}} = 
\frac{\Phi(t)}{|\Group|^{-1/6}}\le 
C|\Group|^{-1/6}.
$$

Therefore, with high probability $1 - C|\Group|^{-1/6}$ over $A\sim \anteuniform(\Group, p),$ 
$\phi_t(A)\le |\Group|^{-1/12}.$ Taking a union bound over $t \in \{1,2, \ldots, n-1\},$ we reach the conclusion that with high probability $1 - Cn|\Group|^{-1/6}> 1 - |\Group|^{-1/7}$ over $A\sim \anteuniform(\Group, p),$ 
$\phi_t(A)\le |\Group|^{-1/12}$ holds for all $t \in [n-1].$ 

Now, suppose that $A$ is one such set which satisfies that 
$\phi_t(A)\le |\Group|^{-1/12}$ holds for all $t \in [n-1].$ Using \cref{eq:firsteqonindist}, we obtain
\begin{align*}
    \sum_{k = 0}^{n-1}
    \log 
    \sum_{t = 0}^k \binom{k}{t}\expect_{\bfg\sim \unif(\Group)}\left[
    \frac{\gamma(\bfg)^t}{p^t(1-p)^t}\right]
     \le  
     \sum_{k = 0}^{n-1}
    \log \Bigg(1 +
    \sum_{t = 1}^k \binom{k}{t}
    \frac{|\Group|^{-1/12}}{p^t(1-p)^t}\Bigg).
    \end{align*}
Now, suppose that $|\Group|>n^{24} 2^{12n}p^{-12n}(1-p)^{-12n},$ which is of order $ \exp(Cn \log \frac{1}{p})$ for some absolute constant $C$ when $p \le \frac{1}{2}.$ Then, the above expression is bounded by 
\begin{align*}
    &\sum_{k = 0}^{n-1}
    \log \Bigg(1 +
    \frac{1}{n^22^n}\sum_{t = 1}^k \binom{k}{t}
    \frac{p^{n}(1-p)^n}{p^t(1-p)^t}\Bigg)\\
    &\le 
    \sum_{k = 0}^{n-1}
    \log \Bigg(1 +
    \frac{1}{n^22^n}\sum_{t = 1}^k \binom{k}{t}\Bigg)\\
    &\le 
    \sum_{k = 0}^{n-1}
    \log \Bigg(1 +
    \frac{1}{n^22^n}2^k\Bigg)\le 
    \sum_{k = 0}^{n-1}
    \log \bigg(1 +
    \frac{1}{n^2}\bigg)\le 
    \sum_{k = 0}^{n-1}\frac{1}{n^2} = \frac{1}{n} = o(1).\qedhere
    \end{align*}
\end{proof}
\section{Detection via Signed Triangle Count over the Hypercube}
\label{sec:detection}
Here, we explore the power of counting signed triangles. We will first provide counterparts to the arguments on symmetric connection in \cref{sec:symmetricconnectionsapplications}. As we will see in \cref{sec:doom}, there are fundamental difficulties in applying the signed triangle statistic to non-symmetric connections. \cref{sec:polydetection} shows an even deeper limitation that goes beyond counting any small subgraphs.
\subsection{Detection for Symmetric Connections}
We first simplify \cref{cor:generaltriangledetection}.

\begin{corollary}
\label{cor:superlineartrianglestats}
Suppose that $d = \omega(n), p = \Omega({1}/{n}),$ and $\sigma$ is a symmetric connection. If
$$
n^6 \Big(
\sum_{S \; : \; 
|S|\in \{1,2,d-1,d-1,d\}
}
\widehat{\sigma}(S)^3
\Big)^2 = 
\omega\Big(n^3 p^3 + 
n^4\sum_{|S|\in \{1,d-1,d\}}
\widehat{\sigma}(S)^4
\Big),
$$
then
the signed triangle statistic distinguishes between 
$\ergraph$ and 
$\hypercubegraph$
with high probability.
\end{corollary}
\begin{proof}
We will essentially show that the above inequality, together with  $d = \omega(n), p = \Omega({1}/{n}),$ imply the conditions of 
\cref{cor:generaltriangledetection}.
Denote $$A = \displaystyle \sum_{S \; : \; 
|S|\in \{1,2,d-1,d-1,d\}
}
\widehat{\sigma}(S)^3.$$ As $ \sum_{|S|\in \{1,d-1,d\}}
\widehat{\sigma}(S)^4>0,$ the conditions above imply that 
$n^6A^2 = \omega(n^3p^3),$ so $A = \omega(n^{-3/2}p^{3/2}).$ Note however, that 
$$
B:=\sum_{S \; : \; 
|S|\not \in \{1,2,d-1,d-1,d\}}
\widehat{\sigma}(S)^3\le 
\sum_{S \; : \; 
|S|\not \in \{1,2,d-1,d-1,d\}}
|\widehat{\sigma}(S)|^3  = 
\sum_{3\le j\le d-3 }
\binom{d}{j}\bigg(\frac{\weight{j}{\sigma}}{\binom{d}{j}}\bigg)^{3/2},
$$
where for the last equality we used the symmetry of $\sigma.$ However, as $\weight{j}{\sigma}\le \Var[\sigma]\le p-p^2<p,$ the last sum is bounded by 
$$
\sum_{3\le j\le d-3 }
\binom{d}{j}^{-1/2}p^{3/2} = 
O\left(d^{-3/2}p^{3/2}\right) = 
o(n^{-3/2}p^{3/2}) = 
o(A).
$$
In particular, this means that $A + B = \Omega(A), 
$ and $n^3|A+B| = \omega(n^{3/2}p^{3/2}) = \omega(1).$ Thus, 
\begin{align*}
        & n^6 
\Big(
\sum_{S \neq \emptyset
}
\widehat{\sigma}(S)^3
\Big)^2
= 
\Omega
 \Big(n^6
\sum_{S \; : \; 
|S|\in \{1,2,d-1,d-1,d\}
}
\widehat{\sigma}(S)^3
\Big)^2\\
& =  
\omega\Big(n^3 p^3 + 
n^4\sum_{|S|\in \{1,d-1,d\}}
\widehat{\sigma}(S)^4
\Big)
 + 
\omega\Big(
n^3
\sum_{S \neq \emptyset
}
\widehat{\sigma}(S)^3
\Big)\\ 
& =  
\omega\Big(n^3 p^3 + 
n^3
\sum_{S \neq \emptyset
}
\widehat{\sigma}(S)^3 + 
n^4\sum_{|S|\in \{1,d-1,d\}} 
\widehat{\sigma}(S)^4
\Big),
    \end{align*}
where we used the assumptions of the corollary and the fact that $n^6(A+B)^2 = \omega (|n^3(A+B)|)$ which follows from 
$|n^3(A+B)| = \omega(1).$ To finish the proof, we simply need to add the missing terms in the sum of fourth powers. Using the same approach with which we bounded $B,$ we have 
$$
n^4\sum_{S \; : \; 
|S|\not \in \{1,2,d-1,d-1,d\}}
\widehat{\sigma}(S)^4 = 
n^4\sum_{2\le j\le d-2}
\binom{d}{j}
p^2\binom{d}{j}^{-2} = 
O(n^4p^2d^{-2}) = 
o(n^2p^2) = 
o(n^3p^3),
$$
where in the last line we used the fact that $pn = \Omega(1).$ With this, the conclusion follows.
\end{proof}

Note that this proposition gives an exact qualitative counterpart to \cref{cor:5relevantlevelsforsuperlinearithmic} since it also depends only on levels $1,2,d-2,d-1,d.$

We now transition to proving detection results for specific functions, which serve as lower bounds for our indistinguishability results. We don't give any lower-bound results for \cref{cor:generallipschitzindist} and \cref{cor:classicrggindistinguishability} as these already appear in literature (even though for the slightly different model over Gauss space). Namely, in
\cite{Liu2021APV}, Liu and Racz prove that when 
there $\sigma$ is monotone $\frac{1}{r\sqrt{d}}$-Lipschitz, 
the signed triangle statistic distinguishes the underlying geometry whenever $d = o({n^3}/{r^6}).$ In \cite{Liu2022STOC}, the authors show that the signed triangle statistic distinguishes the connection $\threshold_p$ when $d = \tilde{o}(n^3p^3).$

\subsubsection{Hard Threshold Connections}
Here, we prove the following bound corresponding to \cref{cor:doublethresholdconnectionsindist}.
\begin{corollary}
\label{cor:doublethresholddetection}
When $d$ is even and $d = \omega(n), p = \Omega(\frac{1}{n})$ and $d = o(n^{3/2}p^{3/2}),$
$$
\TV(\ergraph, \RAG(n,\hypercube, p, \doublethreshold_p)) = 1- o(1).
$$
\end{corollary}

It should be no surprise that we need to calculate the Fourier coefficients of $\doublethreshold_p$ to prove this statement. From \cref{prop:doublethresholdcorwithparity} we know that 
$\widehat{\doublethreshold_p}([d]) = \tilde{O}(\frac{p}{d^{1/2}}).$
In \cref{sec:doublethresholdsfouriercoefficients}, we prove the following further bounds.

\begin{proposition}
\label{prop:fourierofdoublethresh}
When $p = d^{O(1)},$ the Fourier coefficients of $\doublethreshold_p$ satisfy
\begin{enumerate}
    \item $\widehat{\doublethreshold_p}([2]) = \Omega(\frac{p}{d}).$
    \item $\widehat{\doublethreshold_p}([d-2]) = \tilde{O}(\frac{p}{d^{3/2}}).$
\end{enumerate}
\end{proposition}

These are enough to prove the corollary. 
\begin{proof}[Proof of 
\cref{cor:doublethresholddetection} 
]
We simply apply \cref{cor:superlineartrianglestats}. As all Fourier coefficients on levels, 1, $d-1,$ equal zero, we focus on levels $2,d-2, d.$ 
\begin{align*}
        \sum_{S \; : \; 
|S|\in \{1,2,d-2,d-1,d\}
}
\widehat{\sigma}(S)^3 & = 
\Omega\left(\frac{p^3}{d^{3}}\binom{d}{2}\right) +
\tilde{O}\left( 
\binom{d}{2} \frac{p^3}{d^{9/2}} + 
\frac{p^3}{d^{3/2}}
\right) = 
\Omega\left(\frac{p^3}{d}\right),\\
\sum_{S \; : \; 
|S|\in \{1,d-1,d\}
}
\widehat{\sigma}(S)^4 & = 
\tilde{O}\left( 
\frac{p^4}{d^2}
\right).
    \end{align*}
Therefore, by \cref{cor:superlineartrianglestats}, we can detect geometry when 
$$
n^6 \frac{p^6}{d^2} = \omega(n^3p^3) + 
\tilde{\omega}\left(
\frac{p^4}{d^2}\right).
$$
One can easily check that this is equivalent to $d = o(n^{3/2}p^{3/2}).$
\end{proof}
\subsubsection{Non-Monotone Connections, Interval Unions, and Fluctuations}
\label{sec:nonmonotonedetection}
We now turn to the more general case of non-monotone connections and interval unions. We only consider the case when $p = \Theta(1)$ so that we illustrate better the dependence on the correlation with parity, manifested in the number of intervals and analytic total variation.
We first handle the simple case of $\pi_\kappa.$ 

\begin{corollary}
\label{cor:pureparitycor}
If $n\kappa^2 = \omega(1),$ the signed triangle statistic distinguishes 
$\ergraph$ and \linebreak $\RAG(n, \hypercube, 1/2,\pi_\kappa)$ with high probability.
\end{corollary}
\begin{proof}
By \cref{cor:generaltriangledetection}, the signed triangle statistic distinguishes the two graph models with high probability whenever  
$$
n^6\kappa^6 = \omega(n^3 + n^3\kappa^3 + n^4\kappa^4).
$$
Trivially, this is equivalent to $n\kappa^2 = \omega(1).$
\end{proof}

\begin{remark}
\normalfont
\label{rmk:cubeandgaussiandifference}
    In particular, note that when $\kappa = \Theta(1),$ one can distinguish with high probability $\ergraphhalf$ and $\RAG(n, \hypercube, 1/2, \pi_\kappa)$  for \textit{any} $d.$ This phenomenon is clearly specific to the hypercube. In the Gaussian setting, whenever $d = \Tilde{\omega}(n^3),$ $\TV\Big(\ergraph, \RGG(n, \mathbb{R}^d, \mathcal{N}(0, I_d), p, \sigma)\Big) = o(1)$ for any connection $\sigma(\bfx,\bfy)$ depending only on $\langle \bfx,\bfy\rangle$ simply because of the fact that Wishart and GOE matrices are indistinguishable in this regime due to \cref{thm:wisharttogoe}.
    One way to understand this difference between Gaussian space and the hypercube is that the latter also posses an arithmetic structure, so connections such as $\pi_\kappa$ which heavily depend on the arithmetic might lead to a very different behaviour. In the particular case of $\pi_{1/2},$ the resulting graph $\RAG(n, \hypercube, 1/2, \pi_{1/2})$ is bipartite: vertices $i,j$ in this graph are adjacent if and only if the number of coordinates equaling $-1$ has the same parity in $\bfx_i$ and $\bfx_j.$
\end{remark}

\begin{remark}
\label{rmk:qualitativeconverseofmain}
    \normalfont
    Note that $\pi_{1/2}$ has a single large Fourier coefficient on non-zero levels, but this already makes $\RAG(n, \hypercube, 1/2, \pi_{1/2})$ detectable for \textit{any} value of $n.$ This is the qualitative converse of \cref{thm:maintheoremindistingishability}.
\end{remark}

\begin{remark}
    \label{rmk:morethansqrtdintervals}
    \normalfont
    Finally, take any $s = \omega(\sqrt{d})$ and let $h(\bfx):=\indicator\left[\sum_{i = 1}^d x_i \in [-s,s]\right].$ Then, the connection $\sigma(\bfx):=h(\bfx)\times \pi_{1/2}(\bfx)$ is an indicator of a union of $O(s)$ intervals.  Nevertheless, $\hypercubegraph$ and $\ergraphhalf$ are distinguishable with probability $1 - o(1).$ Indeed, $\prob_{\bfg\sim\unif(\hypercube)}[\sigma(\bfg) = \pi_{1/2}(\bfg)] = 1- o(1),$ which implies that $\TV\Big(\RAG(m,\hypercube, 1/2, \sigma), \RAG(m, \hypercube, 1/2, \pi_{1/2})\Big) = o(1)$ for all small enough $m.$ Thus, one can distinguish $\hypercubegraph$ and $\ergraphhalf$ by focusing on the induced subgraphs on vertices $\{1,2,\ldots, m\}$ for a small enough $m =\omega(1).$
 \end{remark}

We now proceed to the more interesting interval unions.
 Note that in this case, the indistinguishability rates we prove are 
$d = \omega(n^3 + n^{3/2}s^2)$ for a general union of intervals and $d = \omega(n^{3/2}s^2)$ for a union of intervals symmetric around 0.
The counterparts to these rates that we present are based on the following interval unions. For simplicity, assume that $d$ is even and $d = 2d_1.$ We begin with the general case. 

\begin{corollary}
\label{cor:geeneralintervalunionsdetection}
Suppose that $s = \tilde{\Omega}(d^{1/3})$ and $s = o(\sqrt{d}).$ Then, there exists a connection $\zeta$ which is the indicator of a union of $s$ intervals and
$\expect[\zeta] = \Omega(1).$ Furthermore,
$\zeta$ is detected by the signed triangle statistic whenever $d = o(ns^2).$ 
\end{corollary}

We construct this function $\zeta$ as follows.

\begin{proposition}
\label{prop:exofnonsymmetricinerval}
Let $s = o(\sqrt{d})$
and consider the union of $s$ intervals given by 
$$I_s = 
\left\{
d_1, d_1  +2, d_1 + 4,  \ldots,  d_1 + 2(s-2)
\right\}\cup [d_1 + \sqrt{d}, d].$$
Let the connection $\zeta_s$ be defined as 
$\displaystyle 
\zeta_s(x) = \indicator[\sum_{i = 1}^d x_i \in I_s].
$ Then, the following facts hold:
$\expect[\zeta_s] = \Omega(1),$ $\displaystyle |\widehat{\zeta_s}([\ell])| \le {\binom{d}{l}}^{-1/2}$ for $\ell\in \{1,2,d-1,d-2\},$ and $\widehat{\zeta_s}([d]) = \Theta(\frac{s}{\sqrt{d}}).$
\end{proposition}

 We defer the proof to \cref{sec:intervalunions}. We proceed to providing lower bounds for connections which are indicators of interval unions.

\begin{proof}[Proof of \cref{cor:geeneralintervalunionsdetection}]
We use \cref{cor:superlineartrianglestats} for the interval union indicator 
$\zeta_s.$
The left-hand side becomes 
$$
n^6\Bigg(
\tilde{\Theta}\bigg(\frac{s^3}{d^{3/2}}\bigg)+ 
d O\bigg(\binom{d}{1}^{-3/2}\bigg) + 
\binom{d}{2}O\bigg(\binom{d}{2}^{-3/2}\bigg)\Bigg) = 
n^6\tilde{\Theta}\Bigg(\frac{s^3}{d^{3/2}}\Bigg)
$$
as $s = \tilde{\Omega}(d^{1/3}).$ Similarly, in the right-hand side is 
$$
\Theta(n^3) +
n^4\tilde{\Theta}\bigg(\frac{s^4}{d^2}\bigg)+
n^4  dO\bigg(\binom{d}{1}^{-2}\bigg) + 
n^4 s^2 O\bigg(\binom{d}{2}^{-2}\bigg) = 
O\Bigg(n^3 + n^4\frac{s^4}{d^2}\Bigg).
$$
Thus, 
we need the following inequality to hold for detection:
$$
n^6\frac{s^6}{d^3} = \tilde{\omega}\Bigg(n^3 + n^4\frac{s^4}{d^2}\Bigg).
$$
This is equivalent to $d = \tilde{o}(ns^2).$
\end{proof}

Our construction for the symmetric case is similar.

\begin{proposition}
\label{prop:exofsymmetricinerval}
Let $s = o(\sqrt{d})$
and consider the union of $2s-1$ intervals given by 
$$I^{sym}_s = 
[-d, -d_1 - \sqrt{d}]\cup
\left\{ d_1 - 2(s-2), \ldots,
d_1, \ldots,  d_1 + 2(s-2)
\right\}\cup [d_1 + \sqrt{d}, d].$$
Define the connection 
$\zeta^{sym}_s(x) = \indicator\Big[\sum_{i = 1}^d x_i \in I^{sym}_s\Big].
$ Then, the following facts hold about $\zeta^{sym}_s$:
$\expect[\zeta^{sym}_s] = \Omega(1),$ $ |\widehat{\zeta^{sym}_s}([\ell])| \le {1}/{\sqrt{\binom{d}{l}}}$ for $\ell\in \{1,2,d-1,d-2\},$ and $\widehat{\zeta^{sym}_s}([d]) = \Theta({s}/{\sqrt{d}}).$
\end{proposition}

The same proof yields the following.

\begin{corollary}
Suppose that $s = \Omega(d^{1/6})$ and 
$s = o(\sqrt{d}).$
Then, there exists a connection $\zeta$ which is the indicator of a union of $2s-1$ intervals, symmetric around 0, and $\expect[\zeta] = \Omega(1).$ Furthermore $\zeta$
is detected by the signed triangle statistic whenever $d = o(ns^2).$ 
\end{corollary}

\subsubsection{Connections Without Low Degree and High Degree Terms}
\label{sec:detectionnolowandhighdegreeterms}
Finally, we turn to proving a detection rate that matches the conjecture we stated in 
\cref{conj:nolowandhighdegreeterms}.

\begin{example} 
\label{example:detectnolowandhighdegrees} Suppose that $m$ is fixed.
There exists a symmetric connection $\sigma:\{\pm 1\}^d\longrightarrow [0,1]$ such that 
$\weight{k}{\sigma_1} = 0$ for $k \in\{1,2,\ldots, m-1\}\cup\{d-m, \ldots, d\}$ and, furthermore, if $d = \tilde{\omega}_p(n^{3/m}),$ the graphs  $G(n, p)$ and 
$\RAG(n,\hypercube, p,\sigma)$ are distinguishable with high probability via the signed-triangle statistic. 
\end{example}
 
To prove this result, we will use the following construction. Again, we focus on the case $p = \frac{1}{2}$ to illustrate better the dependence on 
$m.$
\begin{proposition}
\label{prop:nolowandhighdegreeterms}
For any constant $m \in \mathbb{N},$ there exists a symmetric connection $\sigma:\{\pm 1\}^d \longrightarrow [0,1]$ such that
\begin{enumerate}
    \item $\expect[\sigma_1] = \frac{1}{2},$
    \item $\weight{m}{\sigma_1} = \Omega (\frac{1}{\log^m d}),$
    \item $\weight{k}{\sigma_1} = 0$ for $k \in \{1,2,\ldots, m-1\}\cup\{d-m,d-m+1, \ldots, d\}.$
\end{enumerate}
\end{proposition}

We construct this connection in \cref{appendix:functionswithoutlowandhighdegreeterms} and now we continue to the proof of the detection statement.

\begin{proof}[Proof of \cref{example:detectnolowandhighdegrees}]
We simply apply \cref{cor:generaltriangledetection}. To do so, we need to compute the sum of third and fourth powers of Fourier coefficients. Let the sign on the $m$-th level of $\sigma$ be $\epsilon.$

\begin{align*}
        & \sum_{S\neq \emptyset}\widehat{\sigma}(S)^3 = 
        \sum_{S : |S| = m}
        \widehat{\sigma_1}(S)^3 + 
        \sum_{S : m<|S| <d-m} \widehat{\sigma_1}(S)^3\\
        & =  
        \epsilon \binom{d}{m}\left(\frac{\weight{m}{\sigma_1}}{\binom{d}{m}}
        \right)^{3/2}  + 
        O\left(
        \sum_{m<j<d-m}
        \binom{d}{j}
        \left(\frac{p^2}{\binom{d}{j}}\right)^{3/2}
        \right)\\
        & = 
        \epsilon \times \Omega(d^{-m/2}\log^{-3m/2}d) +
        O(d^{-(m+1)/2}) = 
        \Omega(d^{-m/2}\log^{-3m/2}d).
    \end{align*}
In the exact same way, we conclude that 
$$
\sum_{S\neq \emptyset}\widehat{\sigma}(S)^3= 
 O(d^{-m/2})\text{ and }
\sum_{S\neq \emptyset}\widehat{\sigma}(S)^4= 
 O(d^{-m}).
$$
Thus, to detect the geometry, we simply need the following inequality to be satisfied.
$$
n^6 d^{-m}\log^{-3m}d  = 
\omega(n^3 + n^3 d^{-m/2} + 
n^4d^{-m}
).
$$
One can easily check that this is satisfied when 
$d = o(n^{3/m}\log^3d)$ or , equivalently, 
$d = \tilde{o}(n^{3/m}).$
\end{proof}

\subsubsection{Low Degree Polynomials}
\label{sec:lowdegreedetection}
Our lower bounds for low-degree connections leave a large gap to \cref{thm:lowdegreeindist}.

\begin{theorem} There exists a degree 1 symmetric connection $\sigma$ with expectation $p$ such that the signed triangle statistic distinguishes $\ergraph$ and $\hypercubegraph$ with probability $1 - o(1)$ when $d = o((np)^{3/4}).$
\end{theorem}
\begin{proof}Consider the connection $\sigma(\bfg) = p + \sum_{i = 1}^d \frac{p}{d}g_i$ and apply \cref{cor:superlineartrianglestats}.
\end{proof}
\subsection{The Doom Of Signed Triangles: Fourier-Paired Connections}
\label{sec:doom}
\label{sec:fourierpaired}
Thus far, we have used the signed triangle statistic as a useful tool in detecting geometry. However, this technique is not foolproof - it might be the case that the signed triangle statistic does not yield any detection rates. We must note that \cite{Brennan21DeFinetti} explores a similar idea as the authors introduce a bipartite mask and, thus, their graphs do not include any (signed) triangles. Here, we illustrate a different, perhaps more intrinsic way, of making the signed triangle statistic fail. Indeed suppose that 
$ \sum_{S\neq \emptyset}\widehat{\sigma}(S)^3 = 0.
$
Then, in expectation, the signed triangle statistic applied to $\ergraph$ and $\hypercubegraph$ gives the same count: zero.

Of course, one should ask: are there any interesting functions $\sigma$ which satisfy  $\widehat{\sigma}(S)^3= 0$? We show that the answer is yes. We analyze two concrete families of connections, both of which satisfy the following broader definition.

\begin{definition}
We call a connection $\sigma:\{\pm 1 \}^d\longrightarrow [0,1]$ Fourier-paired if the following condition holds. There exists a bijection $\beta$ on the nonempty subsets of $[d]$ such that $\widehat{\sigma}(S) = - 
\widehat{\sigma}(\beta(S))
$ for each $S\neq \emptyset.$
\end{definition}

Trivially, any Fourier-paired connection $\sigma$ satisfies $\sum_{S\neq \emptyset} \widehat{\sigma}(S)^k = 0$ for any odd integer $k.$ The converse is also true -- if 
$\sum_{S\neq \emptyset} \widehat{\sigma}(S)^k = 0$ holds for any odd $k,$ the connection is Fourier-paired. So, how can we construct Fourier-paired connections? It turns out that in \cref{sec:transformationsofsymmetric}.
we already have!

First, note that for any connection $\sigma$ (not even necessarily, symmetric), its repulsion-attraction twist $\rho_{\sigma,\sigma}$ is Fourier-paired! Another example can be derived using the approach of negating certain variables. Consider first the majority mapping over $d = 2d_1$ variables defined as follows. $\maj(\bfg) = 1$ if $\displaystyle \sum_{i = 1}^dg_i >0,$ $\maj(\bfg) = 1/2$ if $\displaystyle \sum g_i =0,$ and $\maj(\bfg) = 0$ if $\displaystyle \sum g_i <0.$ Note that $\RAG(n,\hypercube, 1/2,\maj)$ is just the analogue of the usual dense random geometric graph with hard thresholds considered in \cite{Bubeck14RGG}.\footnote{Of course, there is the small difference at latent vectors $\bfx_i,\bfx_j$ for which $\langle\bfx_i,\bfx_j\rangle = 0$ since the edge between them is not formed deterministically. But we will ignore this.} One can easily show using \cref{cor:superlineartrianglestats} that $\ergraph$ and $\RAG(n,\hypercube, 1/2,\maj)$ are distinguishable when $d = o(n^3)$ and this is tight in light of \cref{cor:classicrggindistinguishability}. 

Now, consider the small modification of $\maj,$ which we call the half-negated majority and denote by $\hnmaj.$ Let $\mathbf{h} = (\underbrace{1,1,\ldots,1}_{d_1}\underbrace{-1,-1,\ldots, -1}_{d_1})$ and $\hnmaj(\bfg):=\maj(\mathbf{h}\bfg),$ which is just the coefficient contraction with respect to $\mathbf{h}.$ It turns out that it is relatively easy to show that $\hnmaj$ is Fourier-paired.

\begin{proposition} The mapping $\hnmaj$ is Fourier-paired.
\end{proposition}
\begin{proof} First, we claim that all coefficients on even levels greater than 1 of $\hnmaj$ are equal to 0. Indeed, consider $f(\bfg): = 2\hnmaj(\bfg) - 1.$ Clearly , $f$ is odd in the sense that $f(\bfg) = - f(-\bfg)$ holds for all $\bfg.$ Thus, all coefficients of $f$ on even levels equal $0,$ so this also holds for $\hnmaj$ except for level $0.$

Now, for each $S\subseteq [d]$ of odd size, define $\mathsf{s}(S) =  S\cap [1,d_1]$ and $\mathsf{b}(S) = S\cap [d_1 + 1, 2d].$ We claim that the bijection $\beta$ sending $S$ to a set $T$ such that $\mathsf{s}(S) = \mathsf{b}(T)$ and $\mathsf{s}(T) = \mathsf{b}(S)$ proves that $\hnmaj$ is Fourier paired. Indeed, this follows from the following fact. For any two $\bfg_1, \bfg_2\in \{\pm1\}^{d_1},$ we have\footnote{By $\bfg_1,\bfg_2,$ below we mean their concatenation.} 
$\hnmaj(\bfg_1,\bfg_2)  =\maj(\bfg_1,-\bfg_2) = \maj(-\bfg_2,\bfg_1) = \hnmaj(-\bfg_2,-\bfg_1).$ Combining these two statements, we obtain:
$$
\widehat{\hnmaj}(S)\omega_S(\bfg_1,\bfg_2) = 
\widehat{\hnmaj}(T)\omega_T(-\bfg_2,-\bfg_1) = 
- \widehat{\hnmaj}(T)\omega_T(\bfg_2,\bfg_1),
$$
from which the claim follows as $|T| = |S|$ and $|S|$ is odd.
\end{proof}

Now, for both $\rho_{\maj,\maj}$ and $\hnmaj,$ our \cref{thm:maintheoremindistingishability} implies that the corresponding random algebraic graphs are indistinguishable from $\ergraphhalf$ when $d = \omega(n^3).$ However, we cannot apply our usual statistic for counting triangles to detect those graphs as the number of signed triangles matches the corresponding quantity in the \ER model. Clearly, counting signed 4-cycles, one can hope to obtain some lower bound as their number $\sum_{S\neq }\widehat{\sigma}(S)^4$ (see \cref{obs:fouriercycles}) is positive unless $\sigma$ is constant. Indeed, one can easily verify the following facts:
\begin{align*}
        \expect_{G\sim \ergraphhalf}[\tau_4(G)] & = 0,\\
         \Var_{G\sim \ergraphhalf}[\tau_4(G)] & = \Theta(n^4),\\
        \expect_{H \sim \RAG(n, \hypercube, 1/2, \hnmaj)}[\tau_4(H)] & = \Theta\Big(\frac{n^4}{d}\Big),\\
         \expect_{K \sim \RAG(n, \hypercube, 1/2, \rho_{\maj,\maj})}[\tau_4(K)] & = \Theta\Big(\frac{n^4}{d}\Big).
    \end{align*}
These suggest that when $d = o(n^2),$ the two random algebraic graph models have a total variation of $1 - o(1)$ with $\ergraphhalf.$ Of course, to make this conclusion rigorous, we also need to compute the quantities $\Var_{G\sim \RAG(n, \hypercube,1/2, \cdot)}[\tau_4(G)].$ We choose not to take this approach in the current paper as it is computationally cumbersome. Instead, we demonstrate that this in the case for the analogue of $\hnmaj$ with Gaussian latent vectors using \cref{thm:bipartitemaskwishart} from \cite{Brennan21DeFinetti}.

Suppose that $d = 2d_1,\Omega = \mathbb{R}^{d},$ $\distribution = \mathcal{N}(0,I_d)$ and $\sigma(\bfx, \bfy): = \indicator\left[\sum_{i \le d_1}x_iy_i - \sum_{j>d_1}x_jy_j\ge 0\right].$ We want to understand when  the following inequality holds.
\begin{equation}
\label{eq:halfnegatedgaussians}
\TV\Big(\RGG(n,\mathbb{R}^{d},\mathcal{N}(0,I_d), 1/2,  \sigma), \ergraphhalf\Big) = o(1).
\end{equation}
The edges of $\RGG(n,\mathbb{R}^{d},\mathcal{N}(0,I_d), 1/2,  \sigma)$ are formed as follows. Each of the latent vectors $\bfx_1, \bfx_2,\ldots, \bfx_n$ can be represented as a concatenation of two iid standard Gaussian vectors in $\mathbb{R}^{d_1},$ that is
$\bfx_i = (\bfx^1_i,\bfx^2_i).$ With respect to those, one forms the matrix $W\in \mathbb{R}^{n\times n},$ where $W_{i,j} = \langle \bfx^1_i,\bfx^1_j\rangle - 
\langle \bfx^2_i, \bfx^2_j\rangle.
$ Finally, one forms the adjacency matrix $\bfA$ of $\RGG(n,\mathbb{R}^{d},\mathcal{N}(0,I_d), 1/2,  \sigma)$  by taking $\bfA_{i,j}= sign(W_{i,j}).$ As in \cite{Bubeck14RGG}, it is enough to show that 
$W$ is close in total variation to a symmetric matrix with iid entries to conclude that \cref{eq:halfnegatedgaussians} holds. This is our next result. Recall the definition of $M(n)$ from \cref{sec:previouswork}. 

\begin{theorem}
\label{thm:differenceofwisharts}
Define the law of the difference of two Wishart matrices as follows. $\wishart(n,d,-d)$ is the law of $XX^T - YY^T,$ where $X,Y\in \mathbb{R}^{n\times d}$ are iid matrices with independent standard Gaussian entries. If 
$d = {\omega}(n^2),$ then
$$
\TV\Big(\wishart(n,d,-d), \sqrt{d}M(n)\Big) = o_n(1).
$$
\end{theorem}
\begin{proof} Since $X,Y$ are independent Gaussian matrices, so are $X - Y, X+Y.$ Thus $$V = \frac{1}{\sqrt{2}}\begin{pmatrix} X-Y\\
X+Y
\end{pmatrix}\sim \mathcal{N}(0, I_{d_1})^{\otimes 2n}.
$$ Now, consider the bipartite mask on $VV^T$ from \cref{thm:bipartitemaskwishart}. Specifically, focus on the lower-left triangle. From $VV^T,$ we have 
$\frac{1}{2}(X+Y)(X-Y)^T$ and from 
$\sqrt{d_1}M(2n) + I_{2n},$ we have $\sqrt{d_1} Z,$ where $Z\sim \mathcal{N}(0,1)^{n\times n}$ since it is the lower-left corner of a GOE matrix. From \cref{thm:bipartitemaskwishart}, when $d = {\omega}(n^2),$ we conclude
$$
\TV\Big(\frac{1}{2}(X+Y)(X-Y), \sqrt{d_1}Z\Big) = o(1).
$$
By data-processing inequality (\cref{thm:dataprocessing}) for the function $f(V) = (V+V^T),$ we conclude that 
$$
\TV\Big(f\left(\frac{1}{2}(X+Y)(X-Y)\right), f\left(\sqrt{d_1}Z\right)\Big) = o(1).
$$
Now, $f(\frac{1}{2}(X+Y)(X-Y)) = XX^T - YY^T\sim \wishart(n,d,-d)
$ and 
$f(\sqrt{d_1}Z)  =\sqrt{d_1}(Z + Z^T)\sim \sqrt{d}M(n).$ The conclusion follows.
\end{proof}

Conversely, \cite[Theorem 3.2]{Brennan21DeFinetti} directly shows that counting signed 4-cycles in the adjacency matrix $\bfA$ distinguishes $W$ from a GOE matrix and, thus, $\RGG(n,\mathbb{R}^{d},\mathcal{N}(0,I_d), 1/2,  \sigma)$ from $\ergraphhalf,$ so the phase transition indeed occurs at $d  = \Theta(n^2).$

We use this result in \cref{sec:improvingautocorrelation} to formally demonstrate the inefficiency of $\KL$-convexity used to derive \cref{claim:RaczLiuIndistinguishability} in \cite{Liu2021APV}.

\subsubsection{Beyond Triangles}
Thus far in this section, we showed that the triangle statistic can sometimes be useless. Certain connections -- for example, the Fourier-paired ones -- result in latent space graphs which cannot be distinguished from $\ergraph$ using signed triangle counts. One can prove that something more is true. For any odd $k$ and $\mathcal{H},$ an induced subgraph of the $k$-cycle, the expected number of  copies of $\mathcal{H}$ in $\ergraph$ and $\hypercube$ is the same whenever $\sigma$ is Fourier-paired.

At a first glance,
the key word here turns out to be \textit{odd}. Note that unless $\sigma$ is constant, the expected number of signed $k$-cycles in $\hypercubegraph$ for $k$ even, $\sum_{S\neq \emptyset} \widehat{\sigma}(S)^k,$ is positive and, thus, strictly greater than the expected number for $\ergraph.$ Therefore, for any fixed $\sigma$ and $d,$ if we take $n$ to be large enough, we can distinguish $\hypercubegraph$ and $\ergraph$ by counting signed 4-cycles (or signed $k$-cycles for some other $k$ even). In particular, this demonstrates an extremal property of \ER graphs
which is a priori not all obvious. For any fixed $p,$ the latent space graph $\hypercubegraph$ with the minimal expected number of signed $k$-cycles for any even $k$ is $\ergraph.$
This difference between odd and even cycles in random graphs which are close to \ER is not accidental. It plays a crucial role in the literature on quasi-random graphs \cite{Krivelevich2005}.

Nevertheless, we next show that counting (signed) even cycles is also sometimes useless.

\section{Detection for Typical Indicator Connections}
\label{sec:polydetection}
We move on to establishing lower bounds for \cref{thm:typicaldense} and \cref{thm:generalragindistinguishability}.
\subsection{Signed Subgraph Counts}
Most naturally, we would like to apply signed cycle counts to detect a random algebraic graph. However, a simple calculation over the hypercube shows that this approach will typically fail unless $\log |\Group| = O(\log n).$ 

\begin{proposition} 
\label{prop:typicalfailureofsignedkcycles}
A fixed constant $k \in \mathbb{N}, k\ge 3$ is given.
With probability at least $1 - 2^{-d}$ over\linebreak
$A\sim 
\anteuniform(\Group, p),
$
$$
\expect_{H \sim \RAG(n, \hypercube, \sigma_A, p_A)}[\tau_{k,p}(H)]  = 
O\Big(
\frac{(2\sqrt{d})^k}{2^{(k-2)d/2}}
\Big) = O\left(
2^{-d/3}
\right).
$$
\end{proposition}
\begin{proof} In light of \cref{claim:fouriercoefftypicalindicators}, it is enough to prove the claim for all sets $A\subseteq\hypercube$ such that $|\widehat{\sigma_A}(S)|\le \frac{2\sqrt{d}}{2^{d/2}}$ holds for all $S\neq\emptyset$ and 
$|\widehat{\sigma_A}(\emptyset) - p|\le \frac{2\sqrt{d}}{2^{d/2}}.$
Let $A$ be such a set. Applying \cref{obs:fouriercycles}, we have
$$
\expect_{H \sim \RAG(n, \hypercube, \sigma_A,p_A)}[\tau_k(H)] = 
(p_A - p)^k + 
\sum_{S\neq \emptyset}\widehat{\sigma_A}(S)^k  = 
O\bigg(
2^{d}\times \Big(\frac{2\sqrt{d}}{2^{d/2}}\Big)^k
\bigg),
$$
as desired.
\end{proof}

This makes the signed $k$-cycle statistic useless for $k = O(1)$ when $d = \omega(\log n), p = \omega(\frac{1}{n})$ since one can trivially check that $\Var_{G\sim \ergraph}[\tau_{k}(G)] = \Theta(np^k(1-p)^k).$

\subsection{Low Degree Polynomials}
So, can we do better by considering some other poly-time computable statistic? We show that the answer to this question is no, insofar as one considers statistics which are low-degree polynomial tests. Our argument applies in the same way over any finite group $\mathcal{G}$ of appropriate size. We consider right algebraic graphs, due to their relevance to Cayley graphs, but the same statement with the exact same proof holds for left algebraic graphs.

\begin{theorem}
\label{thm:lowdegreepolyagainstrandomsubsets}
An integer $k = o(n)$ and real number $p \in [\frac{1}{n^3},1 - \frac{1}{n^3}]$ are given.
Suppose that $\Group$ is an arbitrary group of size at least $\exp(Ck\log n)$ for some universal constant $C.$ Let $A\sim \anteuniform(\Group,p).$ Consider the connection $\sigma_{A}(\bfx,\bfy):=\indicator[\bfx\bfy^{-1}\in A]$ with $\expect[\sigma_{A}] = p_A.$ Then, one typically cannot distinguish $\ergraph$ and $\RRAG(n , \Group, p_A, \sigma_{A})$ using low-degree polynomials in the following sense. With high probability\linebreak  $1 - \exp(-\Omega(k\log n))$ over $A,$ 
\[
|
\expect_{G\sim \RRAG(n , \Group, p_A, \sigma_{A})} P(G)-
\expect_{H\sim \ergraph}
P(H)
| = O\left(
\frac{n^{2k}}{|\Group|^{1/3}}\right) = \exp(- \Omega(k \log n))
\]
holds simultaneously for all polynomials $P$ of degree at most $k$ in variables the edges of the input graph which satisfy $\Var_{H\sim \ergraph}[P(H)]\le 1.$
\end{theorem}
As we will see, the proof of this theorem highly resembles the proof of \cref{thm:generalragindistinguishability}.
\begin{proof} Suppose that $P$ is a polynomial of degree at most $k.$ For a graph $G,$ denote by $(G)_{(i,j)\in [n]\otimes [n]}$ the $\{0,1\}$ edge indicators of edges of $G,$ where $[n]\otimes [n]$ is the set of $\binom{n}{2}$ pairs of distinct numbers $i,j$ in $[n].$
We can express $P$ as 
$$
P = \sum_{\mathcal{H}\subseteq [n]\otimes [n] : |\mathcal{H}|\le k} a_{\mathcal{H}}
\prod_{(i,j)\in \mathcal{H}}\frac{G_{(i,j)} - p}{\sqrt{p(1-p)}}.
$$
Indeed, $\psi_\mathcal{H}(G):=\prod_{(i,j)\in \mathcal{H}}\frac{G_{(i,j)} - p}{\sqrt{p(1-p)}}$ 
are just the $p$-based Walsh polynomials with respect to the $\{\pm 1\}$-valued indicators  $(2G-1)_{(i,j)\in [n]\otimes [n]}.$ Note that the 
Walsh polynomial $\psi_\mathcal{H}(G)$ corresponds to (signed) subgraph 
indicator of $\mathcal{H}$ 
in $G.$
Now, by Parseval's identity, 
$$
1 = \Var_{H\sim \ergraph}[P(H)] = \sum_{0 < |\mathcal{H}|\le k}a^2_\mathcal{H}.
$$
In particular, this means that $|a_\mathcal{H}|\le 1$ holds for each non-empty subgraph $\mathcal{H}$ with at most $k$ edges. Using triangle inequality, 
\begin{equation}
\label{eq:supboundwalshrandomsubsetconnection}
    \begin{split}
        &|
\expect_{G\sim \RRAG(n , \Group, p_A, \sigma_{A})} P(G)-
\expect_{H\sim \ergraph}
P(H)
|\\
&\le  \sum_{\mathcal{H}\subseteq [n]\otimes [n] : 0<|\mathcal{H}|\le k} |a_{\mathcal{H}}|\times |
\expect_{G\sim \RRAG(n , \Group, p_A, \sigma_{A})}
\psi_\mathcal{H}(G) - 
\expect_{H\sim \ergraph}\psi_\mathcal{H}(H)
|\\
& =  \sum_{\mathcal{H}\subseteq [n]\otimes [n] : 0<|\mathcal{H}|\le k} |a_{\mathcal{H}}|\times |
\expect_{G\sim \RRAG(n , \Group, p_A, \sigma_{A})}
\psi_\mathcal{H}(G) |\\
&\le  \sum_{\mathcal{H}\subseteq [n]\otimes [n] : 0<|\mathcal{H}|\le k} |
\expect_{G\sim \RRAG(n , \Group, p_A, \sigma_{A})}
\psi_\mathcal{H}(G) |\\
&\le 
n^{2k}\times \sup_{\mathcal{H}\subseteq [n]\otimes [n] : 0<|\mathcal{H}|\le k} |
\expect_{G\sim \RRAG(n , \Group, p_A, \sigma_{A})}
\psi_\mathcal{H}(G)|,
    \end{split}
\end{equation}
where in the last line we used the fact that there are at most $n^{2k}$ choices for $\mathcal{H}$ as $|[n]\otimes [n]| = \binom{n}{2}.$ Thus, we simply need to establish a sufficiently strong bound on $|
\expect_{G\sim \RRAG(n , \Group, p_A, \sigma_{A})}
\psi_\mathcal{H}(G)|$ for all small enough $\mathcal{H}.$
We do so by taking a square of the desired quantity and using second moment method as follows.
\begin{equation}
\label{eq:disjointunionofisomorphicgraphs}
    \begin{split}
    \Psi(A)&:=
 |\expect_{G\sim \RRAG(n , \Group, p_A, \sigma_{A})}
\psi_\mathcal{H}(G)|^2\\
&\; =   \expect_{G_1, G_2\sim_{iid} \RRAG(n , \Group, p_A, \sigma_{A})}\left[
\psi_\mathcal{H}(G_1)
\psi_\mathcal{H}(G_2)
\right]\\
&\;=  \expect_{G\sim_{iid} \RRAG(n , \Group, p_A, \sigma_{A})}\left[
\psi_{\mathcal{K}}(G)\right],
    \end{split}
\end{equation}
where $\mathcal{K}$ is an arbitrary union of two vertex-disjoint copies of $\mathcal{H}.$ Indeed, 
suppose that $\mathcal{H}$ has a vertex set $V(\mathcal{H})\subseteq [n],$ where $|V(\mathcal{H})|\le 2k$ (as $|E(\mathcal{H})|\le k$). Take some set of vertices $V_1\subseteq [n]\backslash V(\mathcal{H})$ such that 
$|V_1| = |V(\mathcal{H})|$ (this is possible since $n>4k = o(n)$) and let $\mathcal{H}_1$ be an isomorphic copy of $\mathcal{H}$ on $V_1.$ Since the vertex sets of $\mathcal{H}$ and $\mathcal{H}_1$ are disjoint, the variables $\psi_\mathcal{H}(G)$ and $\psi_{\mathcal{H}_1}(g)$ are independent. Indeed, 
$\psi_\mathcal{H}(G)$ only depends on latent vectors $(\bfx_i)_{i \in V(\mathcal{H})}$ and
 $\psi_{\mathcal{H}_1}(G)$ on latent vectors $(\bfx_i)_{i \in V(\mathcal{H}_1)}.$ Furthermore, the two variables are trivially iid. Thus,  
 $$
 \expect_{G_1, G_2\sim_{iid} \RRAG(n , \Group, p_A, \sigma_{A})}\left[
\psi_\mathcal{H}(G_1)
\psi_\mathcal{H}(G_2)
\right] = 
\expect_{G\sim_{iid} \RRAG(n , \Group, p_A, \sigma_{A})}\left[
\psi_{\mathcal{H}\cup \mathcal{H}_1}(G)
\right],
 $$
and $\mathcal{K}$ is isomorphic to $\mathcal{H}\cup \mathcal{H}_1.$
 
Note that $\mathcal{K}$ has at most $2k$ edges. Without loss of generality, let $V(\mathcal{K}) = \{1,2,\ldots, s\},$ where $s \le 4k.$ 
We want to show that, with high probability over $A,$  $\displaystyle \Psi(A) = 
\expect_{G\sim_{iid} \RRAG(n , \Group, p_A, \sigma_{A})}\left[
\psi_{\mathcal{K}}(G)\right]
$
 is small. Note that
 $0 \le \Psi(A)$ by construction (see \cref{eq:disjointunionofisomorphicgraphs}), which will allow us to use Markov's inequality. Now, taking the expectation over $A,$ we have 
\begin{equation}
    \label{eq:averageoverA}
    \begin{split}
        \expect_A[\Psi(A)] & =  
    \expect_A 
    \expect_{G\sim_{iid} \RRAG(n , \Group, p_A, \sigma_{A})}\left[
\psi_{\mathcal{K}}(G)\right]\\
& =   
\expect_A
\expect_{\bfx_1, \bfx_2, \ldots, \bfx_s\sim_{iid}\unif(\mathcal{G})}
\prod_{(i,j)\in E(\mathcal{K})}
\frac{\indicator[\bfx_i\bfx_j^{-1} \in A] - p}{\sqrt{p(1-p)}}
\\
& =  
\expect_{\bfx_1, \bfx_2, \ldots, \bfx_s\sim_{iid}\unif(\mathcal{G})}\expect_A
\prod_{(i,j)\in E(\mathcal{K})}
\frac{\indicator[\bfx_i\bfx_j^{-1} \in A] - p}{\sqrt{p(1-p)}}.
    \end{split}
\end{equation}
We proceed in a similar fashion to \cref{thm:generalragindistinguishability}. Let $B\subseteq \mathcal{G}^{s}$ be the subset of $s$-tuples of vectors for which the following holds. For each $(i,j)\neq (i',j'),$ where $(i,j)\in [s]\otimes [s]$ and 
$(i',j')\in [s]\otimes [s],$ $\bfx_{i}\bfx_j^{-1} \neq \bfx_{i'}\bfx_{j'}^{-1}$ and 
$\bfx_{j}\bfx_i^{-1} \neq \bfx_{i'}\bfx_{j'}^{-1}.$ Clearly, each of these equalities, which $B$ has to avoid, happens with probability at most $\frac{1}{|\Group|}$ when
$\bfx_1, \bfx_2, \ldots, \bfx_s\sim_{iid}\unif(\mathcal{G}).$
As there are at most  $2\binom{s}{2}^2 \le O(k^4)$ such equalities to be avoided,
$$
\prob_{\bfx_1, \bfx_2, \ldots, \bfx_s\sim_{iid}\unif(\mathcal{G})}\left[
(\bfx_1, \bfx_2, \ldots, \bfx_s)\in B
\right]\ge 1 - O\left(\frac{k^4}{|\mathcal{G}|}\right).
$$
Note that if $(\bfx_1, \bfx_2, \ldots, \bfx_s)\in B,$ then
$$
\expect_{A\sim \anteuniform(\Group,p)}
\prod_{(i,j)\in E(\mathcal{K})}
\frac{\indicator[\bfx_i\bfx_j^{-1} \in A] - p}{\sqrt{p(1-p)}} = 
\prod_{(i,j)\in E(\mathcal{K})}
\frac{\expect_{A\sim \anteuniform(\Group,p)}\Big[\indicator[\bfx_i\bfx_j^{-1} \in A] - p\Big]}{\sqrt{p(1-p)}} = 0,
$$
where we used that $p \le $
This follows directly from the definition of $\anteuniform$ since the different terms $\bfx_i\bfx_j^{-1}$ are all in different orbits as defined in \cref{def:anteuniform}. On the other hand, whenever $(\bfx_1, \bfx_2, \ldots, \bfx_s)\not\in B,$ 
\begin{align*}
&\expect_A
\prod_{(i,j)\in E(\mathcal{K})}
\frac{\indicator[\bfx_i\bfx_j^{-1} \in A] - p}{\sqrt{p(1-p)}}\\ 
&\le  
\bggn{\prod_{(i,j)\in E(\mathcal{K})}
\frac{\indicator[\bfx_i\bfx_j^{-1} \in A] - p}{\sqrt{p(1-p)}}}\bggn_\infty = 
\bigg(\max\Big(
\sqrt{\frac{1-p}{p}},
\sqrt{\frac{p}{1-p}}\Big)
\bigg)^{|E(\mathcal{K})|} \le 
\sqrt{n^3}^{2k}\le
n^{3k},
    \end{align*}
where we used the fact that $p \in [\frac{1}{n^3},1 - \frac{1}{n^3}].$
Combining these observations,
\begin{align*}
\expect_A[\Psi(A)] & = 
\expect_{\bfx_1, \bfx_2, \ldots, \bfx_s\sim_{iid}\unif(\mathcal{G})}\expect_A
\prod_{(i,j)\in E(\mathcal{K})}
\frac{\indicator[\bfx_i\bfx_j^{-1} \in A] - p}{\sqrt{p(1-p)}}\\
&\le (1 - \prob[B])n^{3k} = 
O\left(
\frac{k^4n^{3k}}{|\Group|}
\right) <
\frac{n^{4k}}{|\Group|}
    \end{align*}
for sufficiently large $n.$
Therefore, by positivity of $\Psi,$  we conclude from Markov's inequality that 
$$
\prob[\Psi(A)\ge |\Group|^{2/3}\times |\Group|^{-1}]=
\frac{n^{4k}}{|\Group|^{2/3}}\le \frac{\exp(4k\log n)}{|\Group|^{2/3}}.
$$
Thus, with high probability $1 - \frac{\exp(4k\log n)}{|\Group|^{2/3}}$ over $A,$ we conclude from 
\cref{eq:disjointunionofisomorphicgraphs} that 
$$
|\expect_{G\sim \RRAG(n , \Group, p_A, \sigma_{A})}
\psi_\mathcal{H}(G)|\le |\Group|^{-1/3},
$$
where we used $p = \Omega(n^{-2}).$
Taking union bound over all $O(n^{2k})$ Walsh polynomials of degree at most $k,$ we conclude that 
$$
|\expect_{G\sim \RRAG(n , \Group, p_A, \sigma_{A})}
\psi_\mathcal{H}(G)|\le |\Group|^{-1/3}
$$
holds for all $\mathcal{H}$ with at most $k$ edges with probability at least $1 - O(n^{2k}\frac{\exp(4k\log n)}{|\Group|^{2/3}}) = 1 - \frac{\exp(6k\log n)}{|\Group|^{2/3}}.
$ Combining with \cref{eq:supboundwalshrandomsubsetconnection}, we conclude that with probability at least $1 - \frac{\exp(6k\log n)}{|\Group|^{2/3}},$ the inequality 
$$
|
\expect_{G\sim \RRAG(n , \Group, p_A, \sigma_{A})} P(G)-
\expect_{H\sim \ergraph}
P(H)
|\le 
\frac{n^{2k}}{|\Group|^{1/3}}
$$
holds for all polynomials $P$ satisfying the assumptions of the theorem.
\end{proof}

In particular, this means that if $|\Group| = \exp(\omega(\log n)),$ one cannot distinguish $\ergraph$ and \linebreak
$\RRAG(n, \Group, p_A, \sigma_{A})$ for typical sets $A\sim \anteuniform(\Group,p)$ with constant-degree polynomials and if $|\Group| = \exp(\omega(\log^2n)),$ one cannot distinguish the two graph models with polynomials of degree $O(\log n).$ Polynomials of degree $O(\log n)$ capture a very wide range of algorithms such as spectral methods and others \cite{Schramm_2022}.

It turns out that the dependence on $|\Group|$ is essentially tight and this is illustrated by a variety of simple algorithms which identify the latent space structure for groups of size at most $\exp(C\log n) = n^C$ for appropriate constants $C.$ In fact, some of these algorithms do not even use the algebraic structure, but apply over arbitrary latent-space graphs with indicator connections! One such algorithm (which is not a direct converse to the low-degree polynomials) is the following.

\begin{theorem}
\label{thm:birthdaypaardox}
Suppose that $\Omega$ is an arbitrary finite latent space of size $o(n^2).$ Let $\sigma$ be any indicator connection with $\expect_{\bfx, \bfy \sim \unif(\Omega)}[\sigma(\bfx,\bfy)]= p = \omega(\frac{\log n}{n}).$
One can distinguish the graphs $\ergraph$ and 
$\mathsf{PLSG}(n , \Omega,\unif, p, \sigma)$ with high probability in polynomial time using $O(n^2)$ polynomials of degree 2. Namely, in 
$\mathsf{PLSG}(n , \Omega,\unif, p, \sigma),$ with high probability there exist two vertices $u\neq v$ which satisfy the following property. For any $w\not \in \{v,u\},$ $(v,w)$ is an edge if and only if $(u,w)$ is an edge.
\end{theorem}
\begin{proof}
Call two vertices $u\neq v$ \textit{neighbourhood-identical} if for any $w\not \in \{v,u\},$ $(v,w)$ is an edge if and only if $(u,w)$ is an edge. A simple union bound calculation shows that in $\ergraph,$ the probability that two neighbourhood-identical vertices exist is of order $O\left(\binom{n}{2}(1 -2p +p^2)^n\right) = o(1).$ On the other hand, by the Birthday Paradox, whenever $|\Group| = o(n^2),$ with high probability there exist two integers $u,v \in \{1,2,\ldots, n\}$ such that $\bfx_u = \bfx_v.$ Clearly, whenever $\bfx_u = \bfx_v,$ the vertices $u$ and $v$ are neighbourhood-identical.

Finally, we need to translate the above argument to degree 2 polynomials. One simple way to do this is to analyze the supremum over all pairs $(i,j)\in [n]\otimes [n]$ of
$$
P_{i,j}(G):=
\sum_{k \not \in \{i,j\}} (1 - G_{i,k} - G_{j,k})^2.
$$
Note that $P_{i,j}$ counts the number of vertices $u$ such that $(i,u)$ is an edge if and only if $(j,u)$ is an edge.
\end{proof}

For more direct converses of the low-degree framework one can, for example, count signed four-cycles. We have the following proposition, the proof of which we deffer to \cref{appendix:signedfourcyclesrag} due to its relative simplicity.

\begin{theorem}
\label{signedfourcyclesforrag}
Suppose that $\Omega$ is an arbitrary finite latent space of size $o(n^{1/8}).$ Let $\sigma$ be any indicator connection with $\expect_{\bfx, \bfy \sim \unif(\Omega)}[\sigma(\bfx,\bfy)]= p >0.$
One can distinguish the graphs $\ergraph$ and \linebreak
$\mathsf{PLSG}(n , \Omega,\unif, p, \sigma)$ with high probability in polynomial time using the signed 4-cycle statistic.
\end{theorem}

The constant $1/8$ can easily be improved to $1/3$ (and, perhaps, beyond) if we consider random algebraic graphs. We do not pursue this.

\subsection{Statistical Detection}
So far, in this section we have provided strong evidence that $|\Group| = \exp(\log^{O(1)}n)$ is the computational limit for detecting $\RRAG(h, \Group, p_A,\sigma_A)$ for a typical $A\sim \anteuniform(\Group,p).$ This leaves a large gap from the $|\Group| = \Omega(n\log \frac{1}{p})$ in \cref{thm:generalragindistinguishability} above which we proved that the two graph models are statistically indistinguishable. So, a question remains - can we prove a better statistical detection rate, ideally on the order of $ \exp(\Tilde{\Omega}(n))$? We obtain the following lower-bound, which matches our indistinguishability rate when $p = \Omega(1).$ It holds in the more general setup of probabilistic latent space graphs.

\begin{theorem}
\label{thm:entropyargument}
There exists an absolute constant $C$ with the following property.
Suppose that $\min(p, 1-p) = \omega(\frac{1}{n^2})$ and $\Omega$ is a latent space of size at most $\exp(CnH(p)),$ where\linebreak  $H(p) = -p\log_2 p - (1-p)\log_2(1-p)$ is the binary entropy. Let $\distribution$ be an arbitrary distribution over $\Omega$ and let $\sigma$ be an arbitrary $\{0,1\}$-valued connection. Then, 
$$
\TV\Big(\ergraph, \PLSG\Big) = 
1 - o(1).
$$
\end{theorem}
\begin{proof} 
We analyze the case of $p \le \frac{1}{2},$ in which we simply have $|\Omega| \le \exp(C'p\log \frac{1}{p}n).$
We apply a simple entropy argument. Consider the $n$ latent vectors, $\bfx_1, \bfx_2, \ldots, \bfx_n.$ The $n$-tuple $(\bfx_1, \bfx_2, \ldots, \bfx_n)$ can take at most $|\Group|^n \le \exp(C'n^2p\log \frac{1}{p})$ values. Since edges are deterministic functions of latent vector (because $\sigma$ is $\{0,1\}$-valued),
$\PLSG$ is a distribution with support of size at most $\exp(C'n^2p\log \frac{1}{p}).$  Denote this support by $S.$

Now, consider $H \sim \ergraph.$ First, by standard Chernoff bounds, 
$$\prob\bigg[|E(H)|\le \frac{1}{2}\expect[|E(H)|]\bigg] \le 
\exp(-\Omega(\expect[|E(H)|])) = 
\exp(- \Omega(n^2p)) = o(1).
$$
On the other hand, for any fixed graph $K$ on at least $\frac{1}{2}\expect|E(H)| = \frac{1}{2}\binom{n}{2}p$ vertices, we have 
$$
\prob[H = K] = 
p^{|E(K)|}
(1 - p)^{\binom{n}{2} - |E(K)|}\le 
p^{|E(K)|}\le 
\exp\bigg(\log p \frac{1}{2}\binom{n}{2}p\bigg)\le 
\exp\bigg(-\frac{1}{5}n^2p\log \frac{1}{p}\bigg).
$$
Combining these two inequalities, we conclude that 
\begin{align*}
\prob_{H\sim \ergraph}[H \in S^c] & = 
1 - \prob_{H\sim \ergraph}[H \in S]\\
&\ge 1 - \prob\bigg[|E(H)|\le \frac{1}{2}\binom{n}{2}p\bigg] - 
\prob\bigg[|E(H)|> \frac{1}{2}\binom{n}{2}p, H \in S\bigg]\\
&\ge 1 - o(1) -|S|\times \sup_{K \; : \; |E(K)|> \frac{1}{2}\binom{n}{2}p} \prob[H = K]\\
& \ge 1 - o(1) - \exp\bigg(-\frac{1}{5}n^2p\log \frac{1}{p}\bigg)\times |S|\\
 & = 1 - o(1)
 - \exp\bigg(-\frac{1}{5}n^2p\log \frac{1}{p}\bigg)\times \exp\bigg(C'n^2p\log \frac{1}{p}\bigg) = 1 - o(1)
    \end{align*}
for small enough $C'.$
The definition of total variation immediately implies the desired conclusion.
\end{proof} 

\section{Discussion and Future Directions}
\label{sec:discussion}
We introduced random algebraic graphs, which turn out to be a very expressive family of random graph models. In particular, it captures the well-studied random geometric graphs on the unit sphere with connections depending on inner products of latent vectors, the stochastic block model, and blow-ups and random subgraphs of Cayley graphs. We studied in depth the model when the latent space is given by the group $\hypercube.$ The hypercube is of primary interest as it simultaneously captures a geometry of ``uniformly arranged'' directions like the sphere, but also has an abelian-group structure. Using this duality of the hypercube, we exploited tools from Boolean Fourier analysis to study a wide range of non-monotone and/or non-symmetric connections, many of which seem unapproachable with previous techniques for studying random geometric graphs. Some of these connections revealed new phenomena. Perhaps most interesting among them is a suggested nearly-exponential statistical-computational gap in detecting geometry over the hypercube for almost all indicator connections with a prescribed density. We extended this to arbitrary groups of appropriate size and interpreted the respective result in the language of random Cayley graphs.

As promised in the beginning, we briefly discuss which of our results on the hypercube can be attributed to the algebraic structure of $\hypercube$ and which to the geometric.
\begin{enumerate}
    \item \textbf{Algebraic.} Two of our results stand out as algebraic. First, our results on indicators of $s$ intervals show that for $s$ sufficiently large, even though $\sigma$ only depends on the inner product of latent vectors, the dimension might need to be an arbitrary large function of $n$ for the phase transition to \ER to occur,  \cref{cor:fluctuationindist,cor:geeneralintervalunionsdetection}. This, as discussed in \cref{rmk:cubeandgaussiandifference}, is a phenomenon specific to the hypercube as it captures the arithmetic property of parity. As follows \cref{thm:wisharttogoe}, random geometric graphs with Gaussian latent vectors become indistinguishable from \ER when $\sigma$ only depends on the inner product for $d = \omega(n^3),$ regardless of the specific form of $\sigma.$ Second, our result on typical indicator connections \cref{thm:typicaldense} not only relies on the fact that the hypercube is a group, but it also utilizes the fact that it is a \textit{finite} latent space. It is far from obvious how to define ``typical dense indicator connections'' when the latent space is, say, the sphere. In \cref{thm:generalragindistinguishability}, we extend \cref{thm:typicaldense} to more general groups of 
    appropriate size, which is further evidence for the dependence of this result on the underlying group structure.
    \item \textbf{Geometric.} We believe that most of the other results are a feature of the geometry of $(\hypercube, \unif, \sigma)$ and, thus, likely extend to other random geometric graphs such as ones defined over the sphere. First, we extended the result of Liu and Racz on Lipschitz connections \cref{cor:generallipschitzindist} to non-monotone connections, which we believe can be also done over the sphere or Gaussian space (potentially, via arguments close to the ones of \cite{Liu2021APV}). Next, we showed that if $\sigma$ is even --- that is, depends on $|\langle \bfx, \bfy\rangle|$ --- phase transitions occur at $d = \tilde{\Theta}(n^{3/2})$ instead of $d = \tilde{\Theta}(n^3)$ (\cref{cor:doublethresholdconnectionsindist,cor:evenlipschitzindist,cor:fluctuationindist}). Again, this behaviour seems to be purely geometric and we expect it to extend to Gauss space and the unit sphere. 
\end{enumerate}

Since this work is the first to explicitly study random algebraic graphs and to apply tools from Boolean Fourier analysis to random geometric graphs in high dimension, it opens the door to many new exciting lines of research. 
\subsection{On The Indistinguishability Argument of Liu and Racz}
\label{sec:improvingautocorrelation}
Our indistinguishability results are based on \cref{claim:RaczLiuIndistinguishability} due to Liu and Racz. Unfortunately, as the authors themselves note in \cite{Liu2021APV}, the statement of this claim is suboptimal. We join Liu and Racz in posing the task of improving \cref{claim:RaczLiuIndistinguishability}. 

The current paper demonstrates that such an improvement can have further interesting consequences beyond closing the gap in \cref{thm:liuraczlipschitzconnections}. Of special interest is its relevance to the Boolean analogue of \cref{conj:sphericalhardthresholds} demonstrated in  \cref{cor:classicrggindistinguishability}. In light of this connection, one could hope to close \cref{conj:sphericalhardthresholds} with an improved \cref{claim:RaczLiuIndistinguishability}. Further applications related to the current paper are \cref{conj:nolowandhighdegreeterms} and the gaps in the 
setting of union of intervals (\cref{cor:geeneralintervalunionsdetection,cor:fluctuationindist}), low degree polynomials (\cref{sec:lowdegreepolyindist,sec:lowdegreedetection}) coefficient contractions and repulsion-attraction twists (\cref{sec:transformationsofsymmetric,sec:doom}), and the statistical phase transition for typical random algebraic graphs \cref{thm:generalragindistinguishability,thm:entropyargument}.

Here, we formally prove two limitations of \cref{claim:RaczLiuIndistinguishability}.
\paragraph{1. Sublinear dimensions for indicator connections.} First, we show that if $\sigma$ is an indicator with expectation $p<1/2,$  one cannot hope to prove indistinguishability for $d=o(n)$ over the hypercube using \cref{claim:RaczLiuIndistinguishability}. Suppose that $n>8d$ and $\sigma$ is $\{0,1\}$-valued, in which case 
$$
\sup_\bfg \gamma(\bfg) = 
\sup_\bfg \sum_{S\neq \emptyset}\widehat{\sigma}(S)^2\omega_S (\bfg) = 
\sum_{S\neq \emptyset}\widehat{\sigma}(S)^2 = 
\Var[\sigma] = p(1-p).
$$
Consider the term $t = 2d$ for $k = 4d$ in \cref{eq:firsteqonindist}. Using $\|\gamma\|_t^t\ge \frac{1}{2^d}\norm{\gamma}^t_{\infty} = \frac{1}{2^d}\Var[\sigma]^t = 
\frac{p^t(1-p)^t}{2^d}
,$ together with the positivity in \cref{prop:symmetrizing}, we conclude

\begin{align*}
& \sum_{t = 0}^{4d} \binom{4d}{t}\frac{\expect[\gamma^t]}{p^t(1-p)^t}\\
& \ge \binom{4d}{2d}
\frac{\expect[\gamma^{2d}]}{p^{2d}(1-p)^{2d}}\\
& \ge \binom{4d}{2d}\times \frac{p^{2d}(1-p)^{2d}}{2^d}\frac{1}{p^{2d}(1-p)^{2d}}\\
& \ge
\binom{4d}{2d}\frac{1}{2^d}\ge 2^{d}.
    \end{align*}
This shows that the bound used in \cref{claim:RaczLiuIndistinguishability} gives a KL-divergence at least $d$ and, thus, no non-trivial bound on total variation. Note that there is a significant gap to the entropy bound of $d = \tilde{\Omega}(np)$ when $p = o(1)$ \cref{thm:entropyargument}.

\paragraph{2. Inefficiency in KL-Convexity.} Using our observation on differences of Wishart matrices in \cref{sec:transformationsofsymmetric}, we give a formal proof that the KL-convexity used by Racz and Liu in deriving \cref{claim:RaczLiuIndistinguishability} is indeed at least sometimes suboptimal. Consider $(\Omega, \distribution) = (\mathbb{R}^d,\mathcal{N}(0,I_d)),$ where without loss of generality $d$ is even, $d = 2d_1.$ Consider the following two indicator connections:
\begin{align*}
        &\sigma_{+}(\bfx,\bfy) = \indicator[\sum_{i = 1}^d x_iy_i \ge 0],\\
        &\sigma_{+,-}(\bfx,\bfy) = \indicator[\sum_{i = 1}^{d_1} x_iy_i - \sum_{i = d_1 + 1}^{2d_1} x_iy_i \ge 0].\\
    \end{align*}
Clearly, both have expectation $1/2.$ From \cite{Bubeck14RGG}, we know that $\RGG(n,\mathbb{R}^d,\mathcal{N}(0,I_d), \sigma_{+}, 1/2)$ converges to $\ergraphhalf$ in $\TV$ if and only if $d = \omega(n^3).$ However, \cref{thm:differenceofwisharts} shows that \linebreak $\RGG(n,\mathbb{R}^d,\mathcal{N}(0,I_d), \sigma_{+,-}, 1/2)$ already converges to $\ergraphhalf$ in $\TV$ when $d = \omega(n^2).$ 

Despite this difference, once KL-convexity is applied as in \cite{Liu2021APV}, the bounds on the two models become identical. Namely, the proof in \cite{Liu2021APV}
goes as follows. Let $\bfA$ be the adjacency matrix of 
$\RGG(n,\mathbb{R}^d,\mathcal{N}(0,I_d), \sigma, 1/2)$ and let $\bfB$ be the adjacency matrix of $\ergraphhalf.$ Let $\bfA_k, \bfB_k$ be the principle minors of order $k,$ and let $\bfa_k, \bfb_k$ be the respective last rows. Finally, let $\bfX = (\bfx_1, \bfx_2, \ldots, \bfx_n)^T\in \mathbb{R}^{n\times d}$ be the matrix of latent vectors and $\bfX_k$ its $k\times d$ submatrix on the first $k$ rows. Clearly, $\bfX_k$ determines $\bfA_k.$

Then, Racz and Liu write:
\begin{align*}
        \KL\Big(\RGG(n,\mathbb{R}^d,\mathcal{N}(0,I_d), \sigma, 1/2)\|\ergraphhalf\Big) 
        & = \KL(\bfA \| \bfB)\\
        & = \sum_{k = 0}^{n-1}\expect_{\bfA_{k}}
        \KL\Big(\bfa_{k+1}|\bfA_k\|\bfb_{k+1}|\bfB_k = \bfA_k\Big)\\
        & = \sum_{k = 0}^{n-1}\expect_{\bfA_{k}}
        \KL\Big(\bfa_{k+1}|\bfA_k\|\bfb_{k+1}\Big)\\
        &\le \sum_{k = 0}^{n-1}\expect_{\bfX_k,\bfA_{k}}
        \KL\Big(\bfa_{k+1}|\bfA_k,\bfX_k\|\bfb_{k+1}\Big)\\
        & = \sum_{k = 0}^{n-1}\expect_{\bfX_k}
        \KL\Big(\bfa_{k+1}|\bfX_k\|\bfb_{k+1}\Big).
    \end{align*}
We claim that at this point, there is noting distinguishing $\sigma_+$ and $\sigma_{+,-}.$ Namely, we claim that $\bfa^{(+)}_{k+1}|\bfX_k$ and $\bfa^{(+,-)}_{k+1}|\bfX_k$ have identical distributions for any fixed $k$ (where the subscripts show the dependence on $\sigma$). This captured by the following simple fact. Take any $\bfv\in \{0,1\}^k.$ Then,
\begin{align*}
        &\prob[\bfa^{(+)}_{k+1} = \bfv|\bfX_k]\\ 
    & =  \prob\bigg[\indicator\bigg[\sum_{i = 1}^d(\bfx_{k+1})_i(\bfx_j)_i \bigg] = \bfv_j \; \forall j \in [k]\bigg]\\
    & =  \prob\bigg[\indicator\bigg[\sum_{i = 1}^{d_1}(\bfx_{k+1})_i(\bfx_j)_i - 
    \sum_{i = d_1}^{2d}(-(\bfx_{k+1})_i)(\bfx_j)_i
    \bigg] = \bfv_j \; \forall j \in [k]\bigg]\\
    &\prob[\bfa^{(+,-)}_{k+1} = \bfv|\bfX_k],
    \end{align*}
which follows from the fact that 
$$(\bfx_1, \bfx_2,\ldots, \bfx_{d_1}, \bfx_{d_1 + 1}, \bfx_{d_1 + 2}, \ldots, 
\bfx_{d}
)\longrightarrow
(\bfx_1, \bfx_2,\ldots, \bfx_{d_1}, -\bfx_{d_1 + 1}, -\bfx_{d_1 + 2}, \ldots, 
-\bfx_{d})
$$
is a measure preserving map on $\bigg(\mathbb{R}^d,\mathcal{N}(0,I_d)\bigg).$

The fact that the \cref{claim:RaczLiuIndistinguishability} cannot distinguish between $\sigma_+$ and $\sigma_{+,-}$ also has a natural interpretation in the hypercube setting. Note that $\sigma_+$ is just the $\maj$ connection and $\sigma_{+,-}$ is the $\hnmaj$ connection (see \cref{sec:doom}). However, those two functions only differ in the signs of their Fourier coefficients. However, any information about signs is lost when one analyses the respective autocorrelation $\gamma$ as Fourier coefficients are squared (see \cref{sec:booleanfourier}).

\subsection{Groups Beyond the Hypercube}
As already discussed, the random algebraic model is very expressive and, for that reason, it would be interesting to consider further instances of it, especially over groups other than $\hypercube.$  Nevertheless, these other groups remain challenging and interesting, especially in the case of concrete rather than typical connections. 

\paragraph{1. Abelian Groups.}We believe that products of small (of, say, constant size) cyclic groups can be handled in nearly the same way in which we handled the hypercube. In particular, for such groups a very strong hypercontratcivity result analogous to \cref{lem:hypercontractivity} holds (see \cite[Ch. 10]{ODonellBoolean}). Nevertheless, analysing more general groups - even in the finite case - seems challenging. If one chooses to study those via Fourier analysis, this could lead to a fruitful investigation of complex Wishart matrices since functions of finite groups are spanned by their representations over the complex plane. Similarly, one would encounter complex representations when studying connections over finite-dimensional tori. High-dimensional random geometric graphs over tori are studied, for example, in \cite{Friedrich23}.

\paragraph{2. Non-Abelian Groups.}The situation with non-abelian groups seems more challenging even in the finite case due to the fact that 1-dimensional characters are insufficient to span all functions. Several alternative directions seem viable. As a first step, one could study connections depending only on conjugacy classes, which are spanned by 1-dimensional characters. Alternatively, one could try to develop hyperconcentration tools in the non-abelian case as in \cite{Filmus20}. This seems like a difficult but extremely worthwhile pursuit even independently of random geometric graphs. Finally, in the finite case, one could undertake a purely combinatorial approach based on the relationship between random algebraic graphs and Cayley graphs outlined in \cref{sec:grouptheorymodel}. 

In the infinite case, it could be especially fruitful to study the orthogonal group, which is relevant to spherical random geometric graphs (see \cref{sec:grouptheorymodel}). 

\subsection{The Statistical-Computational Gap}
In \cref{sec:polydetection}, we gave evidence for a new statistical-computational gap. We propose several further directions to studying it.

\begin{problem} Find explicit group $\Group$ and set $A\subseteq \Group$ for which one cannot distinguish $\ergraph$ and\linebreak $\RRAG(n , \Group, p_A, \sigma_{A})$ using low-degree polynomials when $|\Group| = \exp(\log^{O(1)}n).$ Does there exist such a pair $\Group, A$ with a short (polynomial in $n$) description? 
\end{problem}

Note that our arguments only show that typical sets $A$ satisfy the desired property, but provide no means of constructing such sets $A.$ Similarly, when $\Group$ is not a hypercube, our proofs for statistical indistinguishability are not constructive, which leads to the following question.

\begin{problem} Given is a number $n$ and a group $\Group$ of size $\exp(\Omega(n)).$ Find an explicit set $A\subseteq \Group$ for which 
$$
\TV\Big(
\RRAG(n , \Group, p_A, \sigma_{A})
,\ergraph
\Big) = o(1).
$$
 Does there exist such a set $A$ with a short (polynomial in $n$) description? 
\end{problem}

Based on our construction over the hypercube in \cref{rmk:constrcutionwithgeularfunctions},
we believe that the answer to the existence question should be affirmative for general groups.

Another natural direction is proving hardness results against other computational models than low-degree polynomials. We note that our current result leaves many possible poly-time algorithms unaddressed such as ones based on arithmetic. Low-degree polynomials over $\mathbb{F}_2$ are not necessarily low-degree polynomials over $\mathbb{R}$ (see \cite{ODonellBoolean}). Refuting the statistical-computational gap we proposed for general groups via arithmetic-based polynomial-time algorithms does not seem completely unfeasible due to the underlying arithmetic in random \textit{algebraic} graphs. 

\subsection{Boolean Fourier Analysis and Probability}
We end with a few directions in Boolean Fourier analysis and probability related to the current paper that might be of independent interest.

\paragraph{1. Wishart Matrices.} 
\cref{thm:differenceofwisharts} suggests that differences of Wishart matrices might behave very differently from Wishart matrices. It could be interesting to study more general linear combinations of Wishart matrices. Specifically, it could be interesting to study the following problem asking for the linear combination closest to GOE.

\begin{problem} Given are $n, d_1, d_2, \ldots, d_k \in \mathbb{N}.$ Let $X_1, X_2, \ldots, X_k$ be $k$ independent Gaussian matrices where, $X_i \sim \mathcal{N}(0, I_{d_i})^{\otimes n}\in \mathbb{R}^{n\times d_i}.$ Find 
$$
\mathbf{a}\in \arg \min_{\mathbf{a} \; : \; \norm{\mathbf{a}}_2 = 1} \TV\Big(\sum_{i = 1}^k a_i X_iX_i^T , \sqrt{\sum_{i = 1}^k d_i}M(n) + \left(\sum_{i = 1}^k d_i \alpha_i\right)I_n\Big).
$$
\end{problem}

Tools related to the eigenvalues of Wishart matrices similar to the ones developed in \cite{Kumar_2020} for the complex case may be useful.

\paragraph{2. Hypercontractivity.}In \cref{thm:symmetricpolybounds}, we developed novel, to the best of our knowledge, bounds on the moments of elementary symmetric polynomials of Rademacher variables in certain regimes. As discussed, these also hold with standard Gaussian and uniform spherical arguments. We noted that these bounds improve the classical hypercontractivity result of \cref{lem:hypercontractivity} because of the small spectral $\infty$-norm of the respective polynomials. We wonder whether one can generalize this phenomenon beyond (signed) elementary symmetric polynomials.

\paragraph{3. Bohnenblust-Hille and Level-$k$ Inequalities.}In \cref{sec:lowdegreepolyindist}, we used a combination of level-$k$ inequalities \cref{thm:levelkinequalities} and the corollary of the Bohnenblust-Hille inequality given in \cref{cor:lowdegreepolyweights}. We wonder if there is a stronger result that interpolates between the two. Specifically, in the case of symmetric polynomials, we pose the following question.

\begin{problem} Does there exist a universal constant $c>0$ such that for any symmetric polynomial \linebreak $f:\hypercube\longrightarrow [0,1]$ of constant degree with $\expect[f] = p,$ the inequality
$
\displaystyle\Var[f] = O\left(\frac{p^c}{d}\right)
$ holds?
\end{problem}

Such an inequality would, in particular, improve the bounds in our \cref{thm:lowdegreeindist}. 

\paragraph{4. Interpretable Boolean Functions Without Low-Degree and High-Degree Terms.}In \cref{prop:nolowandhighdegreeterms} we constructed a function $\sigma:\hypercube\longrightarrow [0,1]$ with weight 0 on levels $1,2,\ldots, m-1,d-m+1, d-m+2, \ldots, d$ and variance of order $\tilde{\Omega}(1)$ when $m$ is a constant. We wonder whether there are interpretable functions satisfying these properties when $m >3.$ 

\section*{Acknowledgements}
We would like to thank Miklos Racz for a helpful discussion.

\printbibliography

@misc{Liu2021APV,
  title={A probabilistic view of latent space graphs and phase transitions},
  author={Suqi Liu and Mikl{\'o}s Z. R{\'a}cz},
  journal={ArXiv},
  year={2021},
  archivePrefix = {arXiv},
  eprint = {2110.15886},
}

@book{ODonellBoolean,
  url = {https://arxiv.org/abs/2105.10386},
  author = {O'Donnell, Ryan},
  title = {Analysis of Boolean Functions},
  publisher = {Cambridge University Press;},
  year = {2014},
}

@misc{RvH,
url = {https://web.math.princeton.edu/~rvan/APC550.pdf},
author = {Handel, Ramon van},
title = {Lecture Notes on Probability in High Dimension}
}

@article{Kumar_2020, 
	url = {https://doi.org/10.1088\%2F1751-8121\%2Fabc3fe},  
	year = 2020,
	month = {11},
	publisher = {{IOP} Publishing},
	volume = {53},
	number = {50},
	author = {Santosh Kumar and S Sai Charan},
	title = {Spectral statistics for the difference of two Wishart matrices},
	journal = {Journal of Physics A: Mathematical and Theoretical}
}

@inproceedings{Eskenazis_21, 
author = {Eskenazis, Alexandros and Ivanisvili, Paata}, title = {Learning Low-Degree Functions from a Logarithmic Number of Random Queries}, 
year = {2022}, 
isbn = {9781450392648}, 
publisher = {Association for Computing Machinery}, 
url = {https://doi.org/10.1145/3519935.3519981}, 
booktitle = {Proceedings of the 54th Annual ACM SIGACT Symposium on Theory of Computing}, pages = {203–207}, numpages = {5}, series = {STOC 2022} }

@inproceedings{Liu2022STOC, author = {Liu, Siqi and Mohanty, Sidhanth and Schramm, Tselil and Yang, Elizabeth}, title = {Testing Thresholds for High-Dimensional Sparse Random Geometric Graphs}, year = {2022}, isbn = {9781450392648}, publisher = {Association for Computing Machinery}, address = {New York, NY, USA}, url = {https://doi.org/10.1145/3519935.3519989}}

@article{RaczRichet2019SmoothTransition,
author = {Racz, Miklos and Richey, Jacob},
year = {2019},
month = {06},
pages = {},
title = {A smooth transition from Wishart to GOE},
volume = {32},
url={https://doi.org/10.1007/s10959-018-0808-2},
journal = {Journal of Theoretical Probability}
}

@article{Bubeck14RGG,
author = {Bubeck, Sébastien and Ding, Jian and Eldan, Ronen and Rácz, Miklós},
year = {2014},
month = {11},
pages = {},
title = {Testing for high-dimensional geometry in random graphs},
volume = {49},
journal = {Random Structures \& Algorithms}
}

@misc{Liu22Expander,
author = {Liu, Siqi and Mohanty, Sidhanth and Schramm, Tselil and Yang, Elizabeth},
year = {2022},
archivePrefix = {arXiv},
eprint = {2210.00158},
title = {Local and global expansion in random geometric graphs}
}

@misc{Brennan21DeFinetti,
author = {Brennan, Matthew and Bresler, Guy and Huang, Brice},
year = {2021},
month = {03},
archivePrefix = {arXiv},
eprint = {2103.14011},
pages = {},
title = {De Finetti-Style Results for Wishart Matrices: Combinatorial Structure and Phase Transitions}
}

@article{Bubeck15ntropicCLT,
author = {Bubeck, Sébastien and Ganguly, Shirshendu},
year = {2015},
month = {09},
pages = {},
title = {Entropic CLT and Phase Transition in High-dimensional Wishart Matrices},
journal = {International Mathematics Research Notices}
}

@inproceedings{Brennan19Reductions,
  title={Optimal average-case reductions to sparse pca: From weak assumptions to strong hardness},
  author={Brennan, Matthew and Bresler, Guy},
  booktitle={32nd Annual Conference on Learning Theory},
  pages={469--470},
  year={2019},
  organization={PMLR}
}

@misc{Brennan22AnisotropicRGG,
author = {Brennan, Matthew and Bresler, Guy and Huang, Brice},
year = {2022},
month = {},
pages = {},
journal = {},
title = {Threshold for Detecting High Dimensional Geometry in Anisotropic Random Geometric Graphs},
note={To appear in Random Structures and Algorithms},
url = {10.48550/arXiv.2206.14896}
}

@article{Pieters22CommunityDetection,
author = {Eldan, Ronen and Mikulincer, Dan and Pieters, Hester},
year = {2022},
month = {05},
pages = {1-22},
title = {Community detection and percolation of information in a geometric setting},
volume = {31},
url = {https://doi.org/10.1017/S0963548322000098},
journal = {Combinatorics, Probability and Computing}
}

@Inbook{Mikulincer20,
author="Eldan, Ronen
and Mikulincer, Dan",
editor="Klartag, Bo'az
and Milman, Emanuel",
title="Information and Dimensionality of Anisotropic Random Geometric Graphs",
bookTitle="Geometric Aspects of Functional Analysis: Israel Seminar (GAFA) 2017-2019 Volume I",
year="2020",
publisher="Springer International Publishing",
address="Cham",
pages="273--324",
isbn="978-3-030-36020-7",
url="https://doi.org/10.1007/978-3-030-36020-7_13"
}

@misc{Liu21PhaseTransition,
author = {Liu, Suqi and Racz, Miklos},
year = {2021},
title = {Phase transition in noisy high-dimensional random geometric graphs},
  archivePrefix = {arXiv},
  eprint = {2103.15249},
}

@book{AliceBobBanach,
author = {Szarek, Stanislaw and Aubrun, Guillaume},
year = {2017},
month = {09},
pages = {},
title = {Alice and Bob Meet Banach: The Interface of Asymptotic Geometric Analysis and Quantum Information Theory},
url={https://www.ams.org/books/surv/223/surv223-endmatter.pdf},
isbn = {978-1470434687},
publisher = {American Mathematical Society},
address = {Providence, RI}
}

@article{Schramm_2022,
	url = {https://doi.org/10.1214%2F22-aos2179}, 
	year = 2022,
	month = {6},  
	publisher = {Institute of Mathematical Statistics}, 
	volume = {50}, 
	number = {3}, 
	author = {Tselil Schramm and Alexander S. Wein}, 
	title = {Computational barriers to estimation from low-degree polynomials}, 
	journal = {The Annals of Statistics}
}

@article{Brennan19PhaseTransition,
author = {Brennan, Matthew and Bresler, Guy and Nagaraj, Dheeraj},
year = {2020},
month = {12},
pages = {1215-1289},
title = {Phase transitions for detecting latent geometry in random graphs},
volume = {178},
journal = {Probability Theory and Related Fields}
}

@article{Jiang2013ApproximationOR,
  title={Approximation of Rectangular Beta-Laguerre Ensembles and Large Deviations},
  author={Tiefeng Jiang and Danning Li},
  journal={Journal of Theoretical Probability},
  year={2013},
  url = {https://ideas.repec.org/a/spr/jotpro/v28y2015i3d10.1007_s10959-013-0519-7.html},
  volume={28},
  pages={804-847}
}

@article{Conlon17HypergraphExpanders,
author = {Conlon, David},
year = {2017},
month = {09},
pages = {},
title = {Hypergraph expanders from Cayley graphs},
volume = {233},
journal = {Israel Journal of Mathematics}
}

@article{Alon94RandomCayleyGraphs,
author = {Alon, Noga and Roichman, Yuval},
title = {Random Cayley graphs and expanders},
journal = {Random Structures \& Algorithms},
volume = {5},
number = {2},
pages = {271-284},
url = {https://onlinelibrary.wiley.com/doi/abs/10.1002/rsa.3240050203},
year = {1994}
}

@article{Bubeck17NetworkSurvey,
author = {Mikl{\'o}s Z. R{\'a}cz and S{\'e}bastien Bubeck},
title = {{Basic models and questions in statistical network analysis}},
volume = {11},
journal = {Statistics Surveys},
publisher = {Amer. Statist. Assoc., the Bernoulli Soc., the Inst. Math. Statist., and the Statist. Soc. Canada},
pages = {1 -- 47},
year = {2017},
URL = {https://doi.org/10.1214/17-SS117}
}

@book{SteinSakarchi,
 ISBN = {9780691113869},
 URL = {http://www.jstor.org/stable/j.ctvd58v18},
 author = {Elias M. Stein and Rami Shakarchi},
 publisher = {Princeton University Press},
 title = {Real Analysis: Measure Theory, Integration, and Hilbert Spaces},
 year = {2005}
}

@Inbook{Krivelevich2005,
author="Krivelevich, Michael
and Sudakov, Benny",
title="Pseudo-random Graphs",
bookTitle="More Sets, Graphs and Numbers: A Salute to Vera S{\'o}s and Andr{\'a}s Hajnal",
year="2006",
publisher="Springer Berlin Heidelberg",
address="Berlin, Heidelberg",
pages="199--262",
isbn="978-3-540-32439-3",
url="https://doi.org/10.1007/978-3-540-32439-3_10"
}

@book{Polianskiy22+,
author = {Polianskiy, Yuri and Wu, Yihong},
year = {Forthcoming},
title = {Information Theory: From Coding to Learning},
publisher = {Cambridge University Press},
url={https://people.lids.mit.edu/yp/homepage/data/itbook-export.pdf}
}

@misc{Filmus20,  
  author = {Filmus, Yuval and Kindler, Guy and Lifshitz, Noam and Minzer, Dor},
  title = {Hypercontractivity on the symmetric group},
  archivePrefix = {arXiv},
  eprint = {2009.05503},
  year = {2020},
}

@article{Devroye11,
author = {Luc Devroye and Andr{\'a}s Gy{\"o}rgy and G{\'a}bor Lugosi and Frederic Udina},
title = {{High-Dimensional Random Geometric Graphs and their Clique Number}},
volume = {16},
journal = {Electronic Journal of Probability},
publisher = {Institute of Mathematical Statistics and Bernoulli Society},
pages = {2481 -- 2508},
year = {2011},
URL = {https://doi.org/10.1214/EJP.v16-967}
}

@article{Chetelat2017TheMA,
  title={The middle-scale asymptotics of Wishart matrices},
  author={Didier Ch'etelat and Martin T. Wells},
  journal={The Annals of Statistics},
  year={2017},
  url = {https://www.jstor.org/stable/26784041}
}

@inproceedings{Khot14, author = {Khot, Subhash and Tulsiani, Madhur and Worah, Pratik}, title = {A Characterization of Strong Approximation Resistance}, year = {2014}, isbn = {9781450327107}, publisher = {Association for Computing Machinery}, address = {New York, NY, USA}, url = {https://doi.org/10.1145/2591796.2591817}, 
booktitle = {Proceedings of the Forty-Sixth Annual ACM Symposium on Theory of Computing}, pages = {634–643}, numpages = {10}, location = {New York, New York}, series = {STOC '14} }

@article{Smith19,
author = {Anna L. Smith and Dena M. Asta and Catherine A. Calder},
title = {{The Geometry of Continuous Latent Space Models for Network Data}},
volume = {34},
journal = {Statistical Science},
number = {3},
publisher = {Institute of Mathematical Statistics},
pages = {428 -- 453},
year = {2019},
URL = {https://doi.org/10.1214/19-STS702}
}

@inbook{haenggi_2012, place={Cambridge}, title={Random geometric graphs and continuum percolation},
booktitle={Stochastic Geometry for Wireless Networks}, publisher={Cambridge University Press}, author={Haenggi, Martin}, year={2012}, pages={221–245},
url ={https://doi.org/10.1017/CBO9781139043816.012}
}

@article{ESTRADA201620,
title = {Consensus dynamics on random rectangular graphs},
author = {Estrada, Ernesto and Sheerin, Matthew},
journal = {Physica D: Nonlinear Phenomena},
volume = {323-324},
pages = {20-26},
year = {2016},
note = {Nonlinear Dynamics on Interconnected Networks},
issn = {0167-2789},
url = {https://www.sciencedirect.com/science/article/pii/S0167278915002171},
}

@article{higham08,
    author = {Higham, Desmond J. and Rašajski, Marija and Pržulj, Nataša},
    title = "{Fitting a geometric graph to a protein–protein interaction network}",
    journal = {Bioinformatics},
    volume = {24},
    number = {8},
    pages = {1093-1099},
    year = {2008},
    month = {03},
    issn = {1367-4803},
    url = {https://doi.org/10.1093/bioinformatics/btn079},
}

@misc{duchemin22,
  author = {Duchemin, Quentin and de Castro, Yohann},
  archivePrefix = {arXiv},
  title = {Random Geometric Graph: Some recent developments and perspectives},
  eprint = {2203.15351},
  year = {2022},
}

@article{Erba20,
  title = {Random geometric graphs in high dimension},
  author = {Erba, Vittorio and Ariosto, Sebastiano and Gherardi, Marco and Rotondo, Pietro},
  journal = {Phys. Rev. E},
  volume = {102},
  issue = {1},
  pages = {012306},
  numpages = {7},
  year = {2020},
  month = {Jul},
  publisher = {American Physical Society},
  url = {https://link.aps.org/doi/10.1103/PhysRevE.102.012306}
}

@article{aldous86,
author = {David Aldous and Persi Diaconis},
title = {Shuffling Cards and Stopping Times},
journal = {The American Mathematical Monthly},
volume = {93},
number = {5},
pages = {333-348},
year  = {1986},
publisher = {Taylor & Francis},
URL = {https://doi.org/10.1080/00029890.1986.11971821},
}

@article{dou86,
author = {Carl Dou and Martin Hildebrand},
title = {{Enumeration and random random walks on finite groups}},
volume = {24},
journal = {The Annals of Probability},
number = {2},
publisher = {Institute of Mathematical Statistics},
pages = {987 -- 1000},
year = {1996},
URL = {https://doi.org/10.1214/aop/1039639374}
}

@misc{Hermon2021CutoffFA,
  title={Cutoff for Almost All Random Walks on Abelian Groups},
  archivePrefix = {arXiv},
  eprint = {2102.02809},
  author={Jonathan Hermon and Sam Olesker-Taylor},
  year={2021}
}

@article{ALON20131232,
title = {The chromatic number of random Cayley graphs},
journal = {European Journal of Combinatorics},
volume = {34},
number = {8},
pages = {1232-1243},
year = {2013},
note = {Special Issue in memory of Yahya Ould Hamidoune},
issn = {0195-6698},
url = {https://www.sciencedirect.com/science/article/pii/S019566981300084X},
author = {Noga Alon},
}

@misc{Friedrich23,
  author = {Friedrich, Tobias and Göbel, Andreas and Katzmann, Maximilian and Schiller, Leon},
  title = {Cliques in High-Dimensional Geometric Inhomogeneous Random Graphs},
  archivePrefix = {arXiv},
  eprint = {2302.04113},
  year = {2023},
}

@book{penrose03,
    author = {Penrose, Mathew},
    title = "{Random Geometric Graphs}",
    publisher = {Oxford University Press},
    year = {2003},
    month = {05},
    isbn = {9780198506263},
    url = {https://doi.org/10.1093/acprof:oso/9780198506263.001.0001},
}

@article{Bonami1970,
author = {Bonami, Aline},
journal = {Annales de l'institut Fourier},
language = {fre},
number = {2},
pages = {335-402},
publisher = {Association des Annales de l'Institut Fourier},
title = {Étude des coefficients de Fourier des fonctions de $L^p(G)$},
url = {http://eudml.org/doc/74019},
volume = {20},
year = {1970},
}

@article{Kahn1988TheIO,
  title={The influence of variables on Boolean functions},
  author={Jeff D. Kahn and Gil Kalai and Nathan Linial},
  journal={29th Annual Symposium on Foundations of Computer Science},
  year={1988},
  pages={68-80},
  url={https://ieeexplore.ieee.org/document/21923}
}

\appendix
\section{Variance of the Signed Triangle Statistic}
\label{sec:appendixvariancecomputation}

\begin{proof}[Proof of \cref{prop:variancecomputation}]
Denote $\tau_{(i,j,k)}(G) =(G_{i,j}-p)(G_{j,k} - p)(G_{k,i}-p).$
For\linebreak $G \sim \hypercubegraph,$ as in \cite{Liu2021APV}, we have
\begin{align*}
        \Var[\tau_3(G)] & =
\binom{n}{3}\Var[\tau_{(1,2,3)}(G)]^2 + 
O\left(\binom{n}{4}\right)
\Cov[\tau_{(1,2,3)}(G),\tau_{(1,2,4)}(G)] + \\
& O\left(\binom{n}{5}\right)
\Cov[\tau_{(1,2,3)}(G),\tau_{(1,4,5)}(G)] + 
O\left(\binom{n}{6}\right)
\Cov[\tau_{(1,2,3)}(G),\tau_{(4,5,6)}(G)].
    \end{align*}
We bound  each of the four terms separately.

\paragraph{Case 1)} Covariance when the two triangles have overlap of size 3. Using that each edge is an indicator,
\begin{align*}
&\expect[\tau_{(1,2,3)}(G)^2] = 
\expect[(G_{1,2} - p)^2(G_{2,3} - p)^2(G_{3,1} - p)^2]\\ 
& = \expect
[(1-2p)(G_{1,2}-p) + (p-p^2)]
[(1-2p)(G_{2,3}-p) + (p-p^2)]
[(1-2p)(G_{3,1}-p) + (p-p^2)]\\
& =(1-2p)^3\tau_{(1,2,3)}(G) + (p-p^2)^3.
    \end{align*}
We used the simple observations that $\expect [G_{i,j} - p] = 0$ and 
$$\expect[(G_{i,j}-p)(G_{j,k} - p)] = 
\expect[(\sigma(\bf\bfx_i^{-1}\bfx_j) - p)(\sigma(\bfx_j^{-1}\bfx_k) - p)] = 
\expect_{\bfg,\bfh\sim_{iid}\unif(\hypercube)}[(\sigma(\bfg) - p)(\sigma(\mathbf{h}) - p)] = 0.
$$
It follows that
\begin{align*}
    &\binom{n}{3}\Var[\tau_{(1,2,3)}(G)^2]  = 
\binom{n}{3}
\left(
(p-p^2)^3  + (1-2p)^3\expect[\tau_{(1,2,3)}(G)] -
\expect[
\tau_{(1,2,3)}(G)^2]
\right)\\
& =  
\binom{n}{3}
\left(
(p-p^2)^3  + (1-2p)^3\sum_{S\neq \emptyset}\widehat{\sigma}(S)^3 - 
\left(\sum_{S\neq \emptyset}\widehat{\sigma}(S)^3\right)^2
\right)  = 
O\left(
n^3p^3
 + n^3\sum_{S\neq \emptyset}\widehat{\sigma}(S)^3\right).
 \end{align*}

\paragraph{Case 2)} When the two triangles have an overlap of size $2.$
\begin{align*}
&\expect[\tau_{(1,2,3)}(G) \tau_{(1,2,4)}(G)]  = \\
& \expect[(G_{1,2}-p)^2 (G_{1,3} - p)(G_{2,3} - p)(G_{1,4} - p)(G_{2,4}-p)] =\\
& \expect\left[\expect\left[
(G_{1,2}(1-2p) + p^2) (G_{1,3} - p)(G_{2,3} - p)(G_{1,4} - p)(G_{2,4}-p) |
\bfx_1, \bfx_2, \bfx_2, \bfx_4, \bfx_5
\right]\right] = \\
& \expect\left[
(\sigma(\bfx_1^{-1} \bfx_2)(1-2p) + p^2) (\sigma(\bfx_1^{-1} \bfx_3)- p)(\sigma(\bfx_2^{-1} \bfx_3) - p)(\sigma(\bfx_1^{-1}\bfx_4) - p)(\sigma(\bfx_2^{-1} \bfx_4)-p)\right] = \\
& \expect\left[
(\sigma(\bfx_1^{-1}\bfx_2)(1-2p) + p^2) \expect_{\bfx_3}\left[(\sigma(\bfx_1^{-1}\bfx_3)- p)(\sigma(\bfx_2^{-1} \bfx_3) - p)\right]
\expect_{\bfx_4}\left[
(\sigma(\bfx_1^{-1} \bfx_4) - p)(\sigma(\bfx_2^{-1} \bfx_4)-p)\right]\right].
\end{align*}
Now, observe that 
\begin{align*}
        &\expect_{\bfx_3}\left[(\sigma(\bfx_1, \bfx_3)- p)(\sigma(\bfx_2^{-1} \bfx_3) - p)\right]
\expect\left[
(\sigma(\bfx_1^{-1} \bfx_4) - p)(\sigma(\bfx_2^{-1} \bfx_4)-p)\right]\\
& =  \expect_{\bfx}\left[(\sigma(\bfx_1^{-1} \bfx)- p)(\sigma(\bfx_2^{-1} \bfx) - p)\right]^2\ge 0.
    \end{align*}
Therefore, as $(\sigma(\bfx_1^{-1} \bfx_2)(1-2p) + p^2)\in [0,1]$ (as $\sigma$ takes values in $0,1$), the above expression is bounded by 
\begin{align*}
&\expect\left[
 \expect_{\bfx_3}\left[(\sigma(\bfx_1^{-1} \bfx_3)- p)(\sigma(\bfx_2^{-1} \bfx_3) - p)\right]
\expect_{\bfx_4}\left[
(\sigma(\bfx_1^{-1} \bfx_4) - p)(\sigma(\bfx_2^{-1} \bfx_4)-p)\right]\right]\\
& =
\expect\left[
(\sigma(\bfx_1^{-1} \bfx_3)- p)(\sigma(\bfx_2^{-1} \bfx_3) - p)
(\sigma(\bfx_1^{-1} \bfx_4) - p)(\sigma(\bfx_2^{-1} \bfx_4)-p)\right]\\
& = 
\sum_{S\neq\emptyset}\widehat{\sigma}(S)^4,
    \end{align*}
where the last equality follows from \cref{obs:fouriercycles}. Thus, the total contribution to the variance from this case is $\displaystyle O\left(n^4\sum_{S\neq\emptyset}\widehat{\sigma}(S)^4\right).$

\paragraph{Case 3)} When the two triangles have an overlap of size $1.$
\begin{align*}
    & \expect[\tau_{(1,2,3)}(G)\tau_{(1,4,5)}(G)]\\
    & = 
    \expect[
(\sigma(\bfx_1^{-1}\bfx_2) - p)
(\sigma(\bfx_2^{-1}\bfx_3) - p)
(\sigma(\bfx_3^{-1}\bfx_1) - p)
(\sigma(\bfx_1^{-1}\bfx_4) - p)
(\sigma(\bfx_4^{-1}\bfx_5) - p)
(\sigma(\bfx_5^{-1}\bfx_1) - p)
)]\\
& = \expect[
(\sigma(\mathbf{g}) - p)
(\sigma(\mathbf{h}) - p)
(\sigma(\mathbf{g^{-1}h^{-1}}) - p)
(\sigma(\mathbf{k}) - p)
(\sigma(\mathbf{\ell}) - p)
(\sigma(\mathbf{k^{-1}\ell^{-1}})- p)
)]\\
& = 
\expect[
(\sigma(\mathbf{g}) - p)
(\sigma(\mathbf{h}) - p)
(\sigma(\mathbf{g^{-1}h^{-1}}) - p)]
\expect[
(\sigma(\mathbf{k}) - p)
(\sigma(\mathbf{\ell}) - p)
(\sigma(\mathbf{k^{-1}\ell^{-1}})- p)]\\
& =  
\expect[\tau_{(1,2,3)}(G)]\expect[\tau_{(1,4,5)}(G)],
    \end{align*}
so the covariance equals 0 in that case.

\paragraph{Case 4)} When the two triangles have an overlap of size $0.$
$$
\expect[\tau_{(1,2,3)}(G)\tau_{(4,5,6)}(G)] =
\expect[\tau_{(1,2,3)}(G)]\expect[\tau_{(4,5,6)}(G)],
$$
so the covariance is also zero in this case. The conclusion follows by adding the four covariances.
\end{proof}

\section{Fourier Weights on Level \hmath$d$}
\label{appendix:corwithparity}
\begin{proof}[Proof of \cref{lem:lastfourier}]
Suppose that $f:\{\pm 1\}^d\longrightarrow [0,1]$ is a symmetric function. Define by $f_i$ the value of $f$ when exactly $i$ of its coordinates are equal to $1.$ Let $a_i = f_{i+1}-f_i$ and $a_0 = f_0,$ so 
$f_i = \sum_{j = 0}^i a_j.$
Consider $\widehat{f}([d]),$ which is the correlation of $f$ with the parity function $\omega_{[d]}(\bfx).$ By definition,
\begin{equation}
\label{eq:fhatdexpression}
    \begin{split}
        & \widehat{f}([d]) = 
        \expect[f(\bfx)\omega_{[d]}(\bfx)]\\
        & = 
        \frac{1}{2^d}\sum_{i = 0}^d
        \binom{d}{i} (-1)^{d-i}
        f(\underbrace{1,1,\ldots,1}_i\underbrace{-1,-1,\ldots, -1}_{d-i})\\
        & =
        \frac{1}{2^d}\sum_{i = 0}^d
        \binom{d}{i} (-1)^{d-i}
        \sum_{j = 0}^i a_j\\
        & = 
        \frac{1}{2^d}
        \sum_{j = 0}^d 
        a_j 
        \sum_{i = j}^{d}
        \binom{d}{i}(-1)^{d-i}\\
        & = 
        \frac{1}{2^d}
        \sum_{j = 0}^d 
        a_j (-1)^{d-j}\binom{d-1}{j-1},
    \end{split}
\end{equation}

where in the last line we used the simple combinatorial fact that 
$\displaystyle \sum_{i = j}^{d}
        \binom{d}{i}(-1)^{d-i} = (-1)^{d-j}\binom{d-1}{j-1},$ which can be easily proved by induction and Paskal's identity $\binom{d-1}{j-1}- \binom{d}{j} = -\binom{d-1}{j}.$ We modify this expression even further, showing that 
\begin{equation}
\label{eq:lastfourierbounds}
\begin{split}
    |\widehat{f}([d])| & = 
    \left|
            \frac{1}{2^d}\sum_{j = 0}^d 
        a_j (-1)^{d-j}\binom{d-1}{j-1}\right| 
        =
        \frac{1}{2}\left|
        \expect_{j \sim Bin(d-1, \frac{1}{2})} (-1)^{d-j}(f_{j+1} - f_j)
        \right| \\
        & \le 
        \frac{1}{2}
        \expect_{j \sim Bin(d-1, \frac{1}{2})} |f_{j+1} - f_j|  = 
        O\left(\frac{\sum_{i = 1}^d|f_{j+1} - f_j|}{\sqrt{d}}\right)= 
        O\left(\frac{\fluctuation(f)}{\sqrt{d}}\right).
        \qedhere
\end{split}        
\end{equation}
\end{proof}

We now continue to bounding $\widehat{\threshold_p}([d])$ and $\widehat{\doublethreshold_p}([d]).$
In what follows, we omit all floor and ceiling signs for simplicity of exposition. Clearly, they don't affect the asymptotics.
\begin{proof}[Proof of \cref{prop:simplethresholdcorwithparity}] Note that for $f = \threshold_p,$ we only have one non-negative value $a,$ which corresponds to the case when $\sum_{i = 1}^dx_i = \tau_p-1$ or, equivalently, 
there are exactly $t = \frac{\tau_p-1+d}{2}$ ones in the vector $x.$ Thus,
$a_{t} = 1$ and $a_s = 0$ when $s\neq t.$ In other words,
$$
|
\widehat{\threshold_p}([d])
| = 
\frac{1}{2^d}\binom{d-1}{t- 1}.
$$
Thus, we simply need to bound the binomial coefficient $\binom{d}{t}.$ To do so, we first bound $\tau_p.$ As each $x_i$ is supported on $[-1,1],$ it is $1$-subgaussian. Therefore, $\sum_{i = 1}^d x_i$ is $d$-subgaussian\footnote{For these simple applications of subgaussian concentration, see \cite[Ch.3]{RvH}.} and, thus, 
$$
p = 
\prob\left[\sum x_i \ge \tau_p\right]\le \exp\left(-\frac{\tau_p^2}{2d}\right),
$$
so $\tau_p\le \sqrt{2d\log\frac{1}{p}}.$ By taking a large enough constant in $O\left(\frac{p\sqrt{\log \frac{1}{p}}}{\sqrt{d}}\right),$ we can also assume that\linebreak $\tau_p \ge \sqrt{d}/10.$ Indeed, this is true because
$$
\prob\left[\sum_{i = 1}^d x_i \ge \sqrt{d}/10\right] = 
\frac{1}{2} - 
\frac{1}{2^d}\sum_{s =  d/2 }^{d/2 + \sqrt{d}/10}\binom{d}{s}\ge 
\frac{1}{2}- 
\frac{1}{2^d}\times\frac{\sqrt{d}}{10}
\binom{d}{d/2}\ge 
\frac{1}{2}- 
\frac{1}{2^d}\times\frac{2^d}{\sqrt{d}}
\frac{\sqrt{d}}{10} = \Omega(1).
$$

Now, let $\psi = \frac{d}{\tau_p}< 10\sqrt{d}.$
We claim that 
$
\binom{d}{t + \psi}\ge C\binom{d}{t}
$
for some absolute constant $C,$ independent of $d$ and $p.$ Indeed, note that 
\begin{align*}
    &\frac{\binom{d}{t + \psi}}{\binom{d}{t}} = 
\prod_{i = 0}^{\psi-1} 
\frac{d-t -i}{t+i+1}\ge 
\left(\frac{d-t-\psi}{t + \psi}
\right)^\psi\\
& \ge  
\left(\frac{\frac{d}{2}-\frac{\tau_p-1}{2} -\psi}{\frac{d}{2} + \frac{\tau_p-1}{2} + \psi}
\right)^\psi = 
\left(
1 - 
\frac{\tau_p -1 +2\psi}{\frac{d}{2} + \frac{\tau_p-1}{2} + \psi}
\right)^\psi \ge 
\left(
1 - 
\frac{2\tau_p +4\psi}{d}
\right)^\psi\\
& \ge  \exp\left( - \frac{2\tau_p\psi + 4\psi^2}{d}\right)\ge 
\exp\left(
- \frac{2d + 100d}{d}
\right) = \exp(-102),
    \end{align*}
as desired.

Now, by definition of $t,$ the function $\threshold_p$ is equal to one if and only if there are at least $\frac{\tau_p+d}{2}$ ones among $x_1, x_2, \ldots, x_d.$ It follows that
$$
p = \frac{1}{2^d}\sum_{j = t}^d \binom{d}{j}\ge 
\frac{1}{2^d}\sum_{i = 0}^{\psi-1}\binom{d}{t+i}\ge 
C\psi \frac{1}{2^d}\binom{d}{t}.
$$
Thus,
\[
\frac{1}{2^d}\binom{d-1}{t-1}\le 
\frac{1}{2^d}\binom{d}{t} = 
O\left(
\frac{p}{\psi}
\right) = 
O\left(
\frac{p\tau_p}{d}
\right) = 
O\left(
\frac{p\sqrt{\log \frac{1}{p}}}{\sqrt{d}}
\right).\qedhere
\]
\end{proof}

The proof of \cref{prop:doublethresholdcorwithparity} is a nearly trivial consequence.

\begin{proof}[Proof of \cref{prop:doublethresholdcorwithparity}] Note that 
$\delta_p = \tau_{\frac{p}{2}}.$ For the function $f = \doublethreshold_p,$ we easily see that if 
$t = \frac{\tau_{p/2}-1+d}{2},$
we have $a_0 = 1, a_{d-t} = -1,a_t = 1,$ and all the other differences are equal to $0.$ Thus, we simply need to bound 
$
\frac{1}{2^d}\left(\binom{d}{t} + 
\binom{d}{d-t}
\right)
$
as above. Using the result of the previous proof, this expression is just
$
O\left(
\frac{\frac{p}{2}\sqrt{\log \frac{2}{p}}}{\sqrt{d}}
\right) = 
O\left(
\frac{p\sqrt{\log \frac{1}{p}}}{\sqrt{d}}
\right).
$
\end{proof}

We we end with a remark which will be useful later on and shows that our estimate for $\binom{d}{t}$ is optimal up to the small logarithmic factor.

\begin{corollary}
\label{cor:lowerboundonbinomialspread}
When $p = d^{O(1)},$
$$
\frac{1}{2^d}\binom{d}{\frac{d+\tau_p - 1}{2}} = \Omega \left(\frac{p}{\sqrt{d}}\right).
$$
\end{corollary}
\begin{proof}
Again, let $t = \frac{d+\tau_p-1}{2}.$ Note that $t \in [d/2, 2d/3]$ by the assumption on $p.$ We will first show that whenever $z\ge\frac{d}{2},$ it is the case that 
$$\binom{d}{z + 2\sqrt{d}}\le C\binom{d}{z}$$
for $C = e^{-1/2}<1.$
Indeed, this is nearly trivial as 
$$
\frac{\binom{d}{z + 2\sqrt{d}}}{\binom{d}{z}} = 
\prod_{i = 0}^{2\sqrt{d}-1}
\frac{d-z-i}{z+i+1} \le 
\prod_{i = \sqrt{d}}^{2\sqrt{d}-1}
\frac{d-z-i}{d/2} \le 
\left(
\frac{d-\sqrt{d}}{d/2}
\right)^{\sqrt{d}} = 
\left(1 - \frac{1}{2\sqrt{d}}\right)^{\sqrt{d}}\le 
e^{-1/2}.
$$
This, however, is sufficient as 
$$
p = 
\frac{1}{2^d}\sum_{j = t}^d \binom{d}{j}\le 
\frac{1}{2^d}\sum_{s = 0}^{\infty}
2\sqrt{d}
\binom{d}{t + s2\sqrt{d}}\le 
\frac{1}{2^d}2\sqrt{d}\binom{d}{t}
\sum_{s = 0}^\infty C^s = 
O\left(\sqrt{d}\frac{1}{2^d}\binom{d}{t}\right),
$$
from which the conclusion follows.
\end{proof}

\section{Boolean Functions Without Low-Degree and High-Degree Terms}
\label{appendix:functionswithoutlowandhighdegreeterms}
\begin{proposition}
\label{prop:connectionswithoutlowandhighterms}
For every integer $d$ and integer $m< d/2,$  there exists a symmetric function\linebreak $f:\{\pm1 \}^d\longrightarrow [-1,1]$ such that:
\begin{enumerate}
    \item $\weight{1}{f} = \weight{2}{f} = \cdots = 
    \weight{m-1}{f} = 0,$
    \item 
    $\weight{d-m}{f} = \weight{d-m+1}{f} = \cdots = 
    \weight{d}{f} = 0,$
    \item 
    $\weight{m}{f}  = \Omega\left(\frac{1}{m^{13/2}\log^{m} d}\right).$
\end{enumerate}
\end{proposition}

We only prove the statement when $d$ and $m$ are both odd, but explain how to remove these parity assumptions.
Our construction is based on modifying the majority function $\maj = \threshold_{1/2},$ about which we first recall the following fact.

\begin{lemma}[\cite{ODonellBoolean}]
Suppose that $d$ is odd and $k <d.$
If $k$ is even, $\weight{k}{\maj} = 0.$ If $k$ is odd and 
$k<d/2,$
$\weight{k}{\maj} = \Theta(\frac{1}{k^{3/2}}).$
\end{lemma}

Now, define the threshold operator $T(\cdot)$ over real functions on the hypercube as follows. $T(f)[\bfx] =\min(f(\bfx),1)$ when $f(\bfx)\ge 0 $ and 
$T( f)[x] = \max(f(\bfx),-1)$ when $f(\bfx)<0.$ Define also the operator $R_m$ on real functions over $\{\pm 1\}^d$ removing layers $\{1,2,\ldots, m-1\}\cup\{d-m,d-m+1, \ldots, d\}.$ Formally, for any $h,$
$$
R_m(h) = h - 
\sum_{i = 1}^{m-1} h^{=i} - 
\sum_{j = d-m}^d h^{=j} = 
\sum_{S\; : \; m\le |S|\le d-m }
\widehat{h}(S)\omega_S.
$$

With this definition in mind, we are ready to prove the statement of \cref{prop:connectionswithoutlowandhighterms} in the case when $d$ and $m$ are both odd. Our goal will be to roughly construct a function satisfying the desired properties by combining the operators $T$ and $R_m$ over $\maj.$

\begin{proof}[Proof of \cref{prop:connectionswithoutlowandhighterms} when $m$ and $d$ are both odd.] Let $\displaystyle A = \left(e(4m+1)(2m+2)\log^{m/2}d\right)^{-1}$ (we will demistify this number in the proof). 
We will show that the function $\frac{1}{2}R_m\circ T\circ (A\times R_m\circ \maj)$ satisfies the desired properties. We do this in steps and begin with analyzing the mapping 
\begin{align*}
    g(\bfx) :& =A\times R_m(\maj)[x] \\
    & = A\left(\maj(\bfx) - \sum_{0<i < m}
\maj^{=i}(\bfx) - \sum_{d-m\le j \le d}
\maj^{=j}(\bfx)\right)  \\
& = A\left(\maj(\bfx) - \sum_{0<i < m}
\widehat{\maj}([i])p_i(\bfx) - \sum_{d-m\le j \le d}
\widehat{\maj}([j])p_j(\bfx)\right).
\end{align*}
Note that $\weight{m}{g} = A^2 \weight{m}{\maj} = \Omega(\frac{1}{m^{11/2}\log^{m}d})$ and 
$$\widehat{g}([m]) = \binom{d}{m}^{-1/2}\weight{m}{g}^{1/2} = 
\Omega\left(\binom{d}{m}^{-1/2}\frac{1}{m^{11/4}\log^{m/2}d}\right).
$$
We will show that with probability at least $\frac{1}{d^{3m}},$ it is the case that 
$g(\bfx)\in [-1,1].$ Indeed, note that 
$$
|g(\bfx)|\le 
A|\maj(\bfx)| + 
A
\sum_{1\le i <m}
|\widehat{\maj}([i])|\times|e_i(\bfx)| + 
A
\sum_{d-m\le j \le d}
|\widehat{\maj}([j])|\times |e_j(\bfx)|.
$$
The first term is clearly bounded by $A.$ 
Note also that for each $i,$
$\binom{d}{i}|\widehat{\maj}([i])|^2 = \weight{i}{\maj}\le 1,$ so \linebreak
$|\widehat{\maj}([{i}])|\le \binom{d}{i}^{-1/2}.$ Finally, we bound each $|e_i|$ in high probability using the moment method and \cref{lem:hypercontractivity}.
Suppose first that $i < m.$
For $t = (4m+1)\log d$ and $a = \frac{1}{A(2m+2)} = e(4m+1)\log^{m/2}d,$
$$
\prob\left[
|e_i(\bfx)|> a\sqrt{\binom{d}{i}}
\right]\le 
\frac{\norm{e_i(\bfx)}_t^t}{a^t\sqrt{\binom{d}{i}}^t}\le 
\frac{\sqrt{t-1}^{it}\norm{e_i(\bfx)}_2^t}{a^t\sqrt{\binom{d}{i}}^t}.
$$
Now, observe that $\norm{e_i}_2^2 = \expect[\sum_{S,T\;  |S|= |T| = i}\omega_S\omega_T] = 
\binom{d}{i}.
$ Thus, the above expression becomes 
$$
\left(
\frac{(t-1)^{i/2}}{a}
\right)^t\le 
\left(
\frac{(t-1)^{m/2}}{a}
\right)^t\le
\left(
\frac{(4m+1)^{m/2}\log^{m/2}d}{a}
\right)^{(4m+1)\log d} \le 
e^{- (4m+1) \log d}\le 
\frac{1}{d^{4m+1}}.
$$
The same conclusion follows for $i \ge d-m$ by observing that $|e_i(\bfx)| = |\omega_{[d]}e_i(\bfx)| = |e_{d-i}(\bfx)|$ holds for all $\bfx.$
By union bound, with probability at least $1-\frac{2m}{d^{4m+1}}\ge1- \frac{1}{d^{4m}},$ the event $$K = \left\{\bfx\; : \; |e_i(\bfx)|\le a\sqrt{\binom{d}{i}}\; \forall i\in \{1,2,\ldots, m-1\}\cup\{d-m,d-m+1, \ldots, d\}\right\}$$ occurs. Note that on $K,$ however, 
\begin{align*}
        |g(\bfx)|& \le A\left(
        1 + 
        \sum_{1\le i <m}\binom{d}{i}^{-1/2}a\binom{d}{i}^{1/2} + 
        \sum_{d-m\le j \le d}\binom{d}{j}^{-1/2}a\binom{d}{j}^{1/2}\right) \\
       & \le A(1 + (2m+1)a)\le Aa(2m+2)= 1. 
    \end{align*}
Thus, on $K,$ the function $g$ is bounded by $1.$ Outside of $K,$ we can just bound $g$ by its $L_{\infty}$ norm
\begin{align*}
|g(\bfx)|& \le A\left(\norm{\maj}_{\infty} + 
        \sum_{1\le i <m} \binom{d}{i}^{-1/2}
        \norm{e_i}_{\infty} + 
        \sum_{d-m\le j \le d}
\binom{d}{j}^{-1/2}\norm{e_j}_{\infty}\right)\\
& \le A\left(1+ 
        \sum_{1\le i <m} \binom{d}{j}^{-1/2}
        \binom{d}{i} + 
        \sum_{d-m\le j \le d}
\binom{d}{j}^{-1/2}\binom{d}{j}\right)\le 
d^{m}.
    \end{align*}

Now, define $h(\bfx): = g(\bfx) - T\circ g(\bfx).$ Note that 
$h \equiv 0$ on $K.$ Therefore, any Fourier coefficient of $h$ is bounded as follows. 
\begin{align*}
|\widehat{h}(S)| & = 
|\expect[h(\bfx)\omega_S(\bfx)]|=
|\expect[h(\bfx)\omega_S(\bfx)1_{K^c}(\bfx)]|\\
& \le 
\expect[|h(\bfx)\omega_S(\bfx)|1_{K^c}(\bfx)] \le 
\norm{h}_{\infty}\expect[1_{K^c}] \\
& \le (\norm{g}_{\infty} + 
\norm{T\circ g}_{\infty}
)\prob[K^c] \\
& \le
(d^m + 1)\frac{1}{d^{4m}}<
\frac{1}{d^{2m}}.
    \end{align*}

Therefore, for any set $S,$ we have that 
$|\widehat{[T\circ g - g]}{(S)}|\le \frac{1}{d^{2m}}.$ In particular, this means that 
\begin{align*}
        &\weight{m}{T\circ g}\ge
\binom{d}{m}
\left(\widehat{g}([i]) - \frac{1}{d^{2m}}\right)^2\\
& = \Omega\left(
\binom{d}{m}
\left(\binom{d}{m}^{-1/2}\frac{1}{m^{11/4}\log^{m/2}d} - 
\frac{1}{d^{2m}}
\right)^2\right) = 
\Omega\left(\frac{1}{m^{11/2}\log^{m/2}}d\right).
    \end{align*}
Furthermore, 
for $i \in \{1,2,\ldots, m-1\}\cup\{d-m,d-m+1, \ldots, d\},$
$|\widehat{T\circ g}([i])| = \|\widehat{[T\circ g - g]}([i])\|\le \frac{1}{d^{2m}}.$
Finally, we will show that the symmetric function $ \frac{1}{2}R_m\circ T\circ g$ satisfies the desired properties. First, note that $\weight{m}{\frac{1}{2}R_m\circ T\circ g} = 
\frac{1}{2}\weight{m}{T\circ g} = 
\Omega(\frac{1}{m^{11/2}\log^{m/2} d}).
$
Since $R_m$ is the last operator applied, also 
$\weight{i}{\frac{1}{2}R_m\circ T\circ g} = 0$ for  $i \in \{1,2,\ldots, m-1\}\cup\{d-m,d-m+1, \ldots, d\}.$ Finally, we need to show that $\frac{1}{2}R_m\circ T\circ g$ is bounded in $[-1,1].$ This, however, is simple.
\begin{align*}
        \norm{\frac{1}{2}R_m\circ T\circ g}_{\infty}
    & =\frac{1}{2}\norm{T\circ g - 
    \sum_{i = 1}^{m-1} (T\circ g)^{=i} - 
    \sum_{j = d-m}^{d} (T\circ g)^{=j}
    }_{\infty} \\
    & \le \frac{1}{2}\norm{T\circ g}_{\infty} + 
    \frac{1}{2}\sum_{i = 1}^{m-1} \norm{(T\circ g)^{=i}}_{\infty} + 
    \frac{1}{2}\sum_{j = d-m}^{d} \norm{(T\circ g)^{=j}}_{\infty}\\
    & \le \frac{1}{2} + 
    \frac{1}{2}\sum_{i = 1}^{m-1} |\widehat{T\circ g}([i])|\times \norm{e_i}_{\infty}+
    \frac{1}{2}\sum_{j = d-m}^{d}  |\widehat{T\circ g}([j])|\times \norm{e_j}_{\infty}\\
    & \le \frac{1}{2} + 
    \frac{1}{2}\sum_{i = 1}^{m-1} \frac{1}{d^{2m}}\binom{d}{i}
    +\frac{1}{2}\sum_{j = d-m}^{d} \frac{1}{d^{2m}}\binom{d}{j}\le 1,
    \end{align*}
with which the proof is complete.
\end{proof}

\begin{remark}[Removing parity assumptions]
\normalfont
Removing the assumption on $d$ being odd is easy as we can just consider the majority function over an even number of elements, introduced in
\cref{sec:doom}. Removing the parity assumption on $m$ is also not hard. We need to repeat the above argument with some function instead of $\maj,$ which also has $\Omega(\frac{1}{m^{3/2}\log d})$ weight on even levels. Using the same reasoning as above, we can prove that such a function is 
$$\mathbf{H}(\bfx):=
T\left(\maj(\bfx)\frac{\sum_{i = 1}x_i}{\log d \sqrt{d}}\right) = 
T\left(
\left|
\frac{\sum_{i=1}^d x_i}{\sqrt{d}\log d }
\right|
\right)
$$
\end{remark}

Now, suppose that $f$ satisfies the conditions in \cref{prop:connectionswithoutlowandhighterms} and let 
$\alpha = \expect[f].$ Then, $h = \frac{f-\alpha}{2}$ also satisfies the conditions in \cref{prop:connectionswithoutlowandhighterms} and has mean zero. Thus, for any $p \le 1/2,$ the mapping  $p(1 + h)$ has mean $p$ and takes values in $[0,1].$ With this, we are ready to state our main corollary in this appendix.

\begin{corollary}
\label{cor:truncatedconnections}
For every integer $d,$ integer $m< d/2,$ and real number $p \in [0,1]$  there exists a symmetric connection $\sigma:\{\pm1 \}^d\longrightarrow [0,1]$ such that:
\begin{enumerate}
    \item $\weight{1}{\sigma} = \weight{2}{\sigma} = \cdots = 
    \weight{m-1}{\sigma} = 0,$
    \item 
    $\weight{d-m}{\sigma} = \weight{d-m+1}{\sigma} = \cdots = 
    \weight{d}{\sigma} = 0,$
    \item 
    $\weight{m}{\sigma}  = \Omega\left(\frac{p^2}{m^{13/2}\log^{m} d}\right).$
    \item $\expect[\sigma] = p.$
\end{enumerate}
\end{corollary}

\section{Further Computations of Fourier Coefficients}
\label{sec:furtherfouriercomp}

\addtocontents{toc}{\protect\setcounter{tocdepth}{1}}

\subsection{The Case of Symmetric Thresholds}
\label{sec:doublethresholdsfouriercoefficients}
\begin{proof}[Proof of \cref{prop:fourierofdoublethresh}]
Suppose that $S$ has size 2 or $d-2,$ in particular $|S|$ is even. Then,
\begin{align*}
        & \widehat{\doublethreshold_p}(S) = 
\expect[\doublethreshold_p(\bfx)\omega_S(\bfx)]
 = 
 \expect\left[\indicator[\sum x_i \ge \tau_{p/2}]\omega_S(\bfx)\right] + 
 \expect\left[\indicator[\sum x_i \le -\tau_{p/2}]\omega_S(\bfx)\right]\\
 & = \expect\left[\indicator[\sum x_i \ge \tau_{p/2}]\omega_S(\bfx)\right] + 
 \expect\left[\indicator[\sum (-x_i) \ge \tau_{p/2}]\omega_S(-\bfx)\right]\\
 & = 2\expect\left[\indicator[\sum x_i \ge \tau_{p/2}]\omega_S(\bfx)\right] = 
 2\widehat{\threshold_{p/2}}(S),
    \end{align*}
where we used the fact $\bfx$ and $-\bfx$ have the same distribution and $|S|$ is even. Thus, we simply need to find the Fourier coefficients of $\threshold_{p/2}$ for our purposes. We do this in a similar manner to 
\cite[Chapter 5]{ODonellBoolean}. Namely, suppose that
$|S| = 2k$ and, without loss of generality due to symmetry, $S =\{d\}\cup S_1,$ where $S_1$ has size $2k-1$ it follows that 
\begin{align*}
        &\expect[\threshold_{p/2}(\bfx)\omega_S(\bfx)]\\
    & = 
\frac{1}{2}\expect\left[\left(\threshold_{p/2}(x_1, x_2, \ldots,x_{d-1}, 1) - 
\threshold_{p/2}(x_1, x_2, \ldots,x_{d-1}, -1)\right) \omega_{S_1}\right] =
\frac{1}{2}\expect [J(\bfy)\omega_{S_1}(\bfy)] = 
\frac{1}{2}
\widehat{J}(S_1),
    \end{align*}
where $\bfy = (x_1, x_2, \ldots, x_{d-1})$ and $$J(\bfy) := 
\left(\threshold_{p/2}(y_1, y_2, \ldots,y_{d-1}, 1) - 
\threshold_{p/2}(y_1, y_2, \ldots,y_{d-1}, -1)\right)  = 
\indicator\left[\sum_{i = 1}^{d-1}y_i = \tau_{p/2} -1\right].
$$
We now find the coefficients of $J$ via the noise-stability operator. Namely, pick the all-ones vector $\mathbf{b} = \mathbf{1}$ and let $\bfz$ be $\rho-$correlated with $\mathbf{b}.$ On the one hand, 
$$
\prob[J(\bfz) = 1] = 
\binom{d-1}{\frac{d-1 + \tau_{p/2}-1}{2}}
\left(\frac{1}{2} + \frac{\rho}{2}\right)^{\frac{d-1 + \tau_{p/2}-1}{2}}
\left(\frac{1}{2}  - \frac{\rho}{2}\right)^{\frac{d-1 - \tau_{p/2}+1}{2}}
$$
as we need exactly $\frac{d-1 + \tau_{p/2}-1}{2}$ of the coordinates of $\bfy$ to equal $1$ and each $y_i$ is iid $Bernoulli(\frac{1}{2} + \frac{\rho}{2})$ as $\bfy$ and $\mathbf{b} = \mathbf{1}$ are $\rho-$correlated. On the other hand, by the definition of the noise operator 
$$
\prob[J(\bfz) = 1] = 
\expect[J(\bfz)] = 
T_\rho 
J(\mathbf{b}) = 
\sum_{S}
\rho^{|S|} \widehat{J}(S)\omega_S(\mathbf{b}) = 
\sum_{s = 0}^{d-1}
\rho^s \binom{d}{s}\widehat{J}(S_1).
$$
Now, we simply compare coefficients of the two polynomial expressions as 
$$
\sum_{s = 0}^{d-1}
\rho^s \binom{d-1}{s}\widehat{J}([s]) = 
\binom{d-1}{\frac{d-1 + \tau_{p/2}-1}{2}}
\left(\frac{1}{2} + \frac{\rho}{2}\right)^{\frac{d-1 + \tau_{p/2}-1}{2}}
\left(\frac{1}{2}  - \frac{\rho}{2}\right)^{\frac{d-1 - \tau_{p/2}+1}{2}}
$$
holds for all $\rho \in [-1,1].$ For notational purposes, we set $t = 
\frac{d-1 + \tau_{p/2}-1}{2}
.$ 

\paragraph{Case 1)} When $s = 1.$ Then, we have 
\begin{align*}
        \binom{d-1}{1}\widehat{J}([1]) 
        & =
        \frac{1}{2^d}\binom{d-1}{t}\left(
        t - (d-t-1)\right) \\
         & =
        \frac{1}{2^d}\binom{d-1}{t}\left(
        2t-d+1\right)\\
        & = 
\Omega\left(
\frac{\sqrt{d}}{2^d}\binom{d}{t}\right) = 
\Omega(p),
    \end{align*}
using the bounds $t = \frac{d}{2} + \Omega(\sqrt{d})$ and $\frac{1}{2^d}\binom{d}{t} = \Omega(\frac{p}{\sqrt{d}})$ in \cref{appendix:corwithparity} on $t$ and $\binom{d}{t}.$ Thus, 
$\widehat{J}([1]) = \Omega(\frac{p}{d})$ and 
so $\widehat{\doublethreshold_p}([2]) = \Omega(\frac{p}{d}).$ Using the bounds $t = \frac{d}{2} + \tilde{O}(\sqrt{d})$ and 
$\frac{1}{2^d}\binom{d}{t} = \tilde{O}(\frac{p}{\sqrt{d}}),$ we similarly conclude that 
$\widehat{\doublethreshold_p}([2]) = \tilde{O}(\frac{p}{d}).$

\paragraph{Case 2)} When $s = d-3,$ we similarly have 
\begin{align*}
        &\binom{d-1}{d-3}\widehat{J}([d-2]) 
        \\
        & = 
        \frac{1}{2^d}\binom{d-1}{t}\left(
        \binom{t}{t-2}(-1)^{d-t}  +
        \binom{t}{t-1}
        \binom{d-1-t}{d-2-t} (-1)^{d-t-1} + 
        \binom{d-1-t}{d-3-t} (-1)^{d-t-2} 
        \right) \\ & = 
        \frac{1}{2^d}\binom{d-1}{t}
        O\left(
        \binom{t}{t-2}- 
        \binom{t}{t-1}
        \binom{d-1-t}{d-2-t} + 
        \binom{d-1-t}{d-3-t}
        \right)\\
        & = 
        \frac{1}{2^d}\binom{d-1}{t}
        O\left(
        \frac{t(t-1) + (d-t-1)(d-t-2) - 2t(d-1-t)}{2}
        \right)\\
        & = 
        \frac{1}{2^d}\binom{d-1}{t}
        O\left(
        \frac{t(t-1) + (d-t-1)(d-t-2) - 2t(d-1-t)}{2}
        \right)\\
        & = 
         \frac{1}{2^d}\binom{d-1}{t}
        O\left(
        \frac{(\tau_{p/2}-1)^2 - (d-1)}{2}
        \right)\\
        & = 
        \tilde{O}
        \left( 
        \frac{p}{\sqrt{d}}d
        \right) = 
        \tilde{O}
        \left( 
        p\sqrt{d}
        \right),
    \end{align*}
where we used the fact that $t = \frac{d-1 + \tau_{p/2}-1}{2}$ to expand the expression and the bounds $\tau_{p/2} = \tilde{O}\left(\sqrt{d}\right)$ and $\frac{1}{2^d}\binom{d-1}{t} = \tilde{O}(\frac{p}{\sqrt{d}}).$ Hence, we know that 
$\widehat{\doublethreshold_p}([d-2]) = 
\frac{1}{2}\widehat{J}([d-2]) = \tilde{O}
        \left( 
        pd^{-3/2}
        \right),
$
as desired.
\end{proof}

\subsection{Interval Unions}
\label{sec:intervalunions}
\begin{proof}[Proof of \cref{prop:exofnonsymmetricinerval}] We have to prove three things about $\zeta_s.$ First, note that 
$$
\expect[\zeta_s]\ge 
\prob\left[
\sum x_i \ge d_1 + \sqrt{d}
\right] = \Omega(1).
$$
Second, for each $\ell,$ we have that 
$$
1\ge \expect[\zeta_s^2]\ge \weight{\ell}{\zeta_s} = 
\binom{d}{\ell}\widehat{\zeta_s}([\ell])^2,
$$
from which the bound on coefficients on levels $1,2,d-1,d-2$ follow. 
\cref{lem:lastfourier} implies that 
\begin{align*}
      \left|\widehat{\zeta_s}\right| 
      & = 
\left|      \frac{1}{2^d}\sum_{j = 0}^{s-2}
(-1)^{d - d_1 - 2j}\binom{d-1}{d_1 + 2j-1} + 
(-1)^{d - d_1 - \sqrt{d}}\binom{d-1}{d_1 + \sqrt{d}-1}\right|\\
& =
\Omega\left(\frac{s-1}{\sqrt{d}} + (-1)^{d_1 - \sqrt{d}}\frac{1}{\sqrt{d}}\right) = 
\Omega\left(\frac{s}{\sqrt{d}}\right),
    \end{align*}
where we used the fact that $\frac{1}{2^d}\binom{d-1}{d_1 + \psi} = \Theta (\frac{1}{\sqrt{d}})$ whenever $\psi = O(\sqrt{d}).$
\end{proof}

\section{Distributions of Cayley Graph Generators}
\label{appendix:beyondanteuniform}
Here, we show that both our statistical and computational indistinguishability results on typical Cayley graphs \cref{thm:generalragindistinguishability,thm:lowdegreepolyagainstrandomsubsets} are robust with respect to the underlying distribution of generators in the following sense. They still hold (with potentially different constants) as long as the distribution satisfies the following ``small-scale conditional independence'' property even if it is not $\anteuniform.$

\begin{definition}(Small-scale Conditional Independence)
\label{def:smalllscaleindependence}
Given are a group $|\Group|,$ a probability $p\in (0,1),$ and an integer $N  = o(|\Group|).$ We say that a distribution $\mathsf{D}(\Group,p)$ satisfies the ``$(c_1,c_2)$-small-scale conditional independence property for $(\Group, N,p),$'' where $c_1,c_2\in (0,1)$ are constants if the following holds. There exists a set $\mathcal{L}\subseteq \Group$ of size at most $|\Group|^{c_1}$ such that all $\bfx\in \Group \backslash \mathcal{L}$ satisfy the condition: For any set $M\subseteq \Group\backslash\{\bfx, \bfx^{-1}\}$ of cardinality $0 \le |M|\le N$ and any $I\subseteq M,$
$$
\bigg|\prob_{A\sim \mathsf{D}(\Group,p)}[
\bfx\in A| A\cap M  = I
]
 - p\bigg|\le \frac{1}{|\Group|^{c_2}}.
$$\end{definition}

Importantly, one can easily see that this definition captures $\postuniform$ if the underlying group has at most $|\Group|^c$ elements of order two for some $c<1$ or, respectively, at most most $|\Group|^c$ elements of order at least three for some $c<1.$ Our theorems are rephrased as follows.

\begin{claim} (See \cref{thm:generalragindistinguishability})
\label{claim:generalragindistinguishabilitybeyondante}
An integer $n$ is given and a real number $p \in [0,1].$ Let $\Group$ be any finite group of size at least $\exp(Cn \log \frac{1}{p})$ for some universal constant $C.$ Let $\mathsf{D}(\Group,p)$ be a distribution over subsets of $\Group$ satisfying the ``$(c_1,c_2)$-small-scale conditional independence property for $(\Group, n^3,p),$'' where $c_1, c_2$ are constant independent of $|\Group|, n, p.$
Then, with high probability at least $1 - |\Group|^{-D}$ over $A\sim \mathsf{D}(\Group,p),$ we have 
$$
\KL\Big(\RAG(n, \Group, \sigma_A, p_A)\| \ergraph\Big) = o(1), 
$$
where $p_A = \expect_{\bfg\sim \unif(\Group)}[\sigma_A(\bfg)]$ and $D>0$ depends only on $c_1, c_2.$
\end{claim}

\begin{claim}(See \cref{thm:lowdegreepolyagainstrandomsubsets})
\label{claim:lowdegreepolyagainstrandomsubsetsbeyondante}
An integer $k = o(n)$ and real number $p \in [\frac{1}{n^3},1 - \frac{1}{n^3}]$ are given.
Suppose that $\Group$ is an arbitrary group of size at least $\exp(Ck\log n)$ for some universal constant $C.$
Let $\mathsf{D}(\Group,p)$ be a distribution over subsets of $\Group$ satisfying the `$(c_1,c_2)$-small-scale conditional independence property for $(\Group, n^3,p),$'' where $c_1, c_2$ are constant independent of $|\Group|, n, p.$
Let $A\sim \mathsf{D}(\Group,p).$ Consider the connection $\sigma_{A}(\bfx,\bfy):=\indicator[\bfx\bfy^{-1}\in A]$ with $\expect[\sigma_{A}] = p_A.$ Then, one typically cannot distinguish $\ergraph$ and $\RRAG(n , \Group, p_A, \sigma_{A})$ using low-degree polynomials in the following sense. With high probability $1 - \exp(-Dk\log n)$ over $A,$ 
\begin{align*}
|
\expect_{G\sim \RRAG(n , \Group, p_A, \sigma_{A})} P(G)-
\expect_{H\sim \ergraph}
P(H)
| = \exp(- Dk \log n)
\end{align*}
holds simultaneously for all polynomials $P$ of degree at most $k$ in variables the edges of the input graph which satisfy $\Var_{H\sim \ergraph}[P(H)]\le 1.$
Here, $D>0$ is a constant depending only on $c_1, c_2.$
\end{claim}

For both statements, the proofs are simple modifications of \cref{thm:generalragindistinguishability}, respectively \cref{thm:lowdegreepolyagainstrandomsubsets}. We only outline how one needs to change the first proof as the modifications are identical. 

\begin{proof}[Proof sketch of \cref{claim:generalragindistinguishabilitybeyondante}]
proceed in the same way as in \cref{thm:generalragindistinguishability}, except that we remove from $B$ the event that one of $\bfg^{(1)}(\bfz_i^{(1)})^{-1}, \bfz_i^{(1)}, \bfg^{(2)}(\bfz_i^{(2)})^{-1}, \bfz_i^{(2)}$ belongs to the set $\mathcal{L}$ from \cref{def:smalllscaleindependence}. By union bound, the resulting event $B_1$ satisfies 
$$
\prob[B_1]\ge \prob[B] - 4t\frac{|\mathcal{L}|}{|\Group|} = 1 - O(|\Group|^{-D_1})
$$
for some fixed constant $D_1.$ Thus, again, over $B_1^c,$ the expectation of \cref{eq:ragindisteq49} is at most $O(|\Group|^{-D_1}).$ On the other hand, over $B_1^c,$ we have that 

\begin{align*}
&\Bigg|\expect_{A\sim \mathsf{D}(\Group,p)}\Bigg[
\prod_{i = 1}^t 
(\sigma_A(\bfg^{(1)}(\bfz^{(1)}_i)^{-1})-p)
\prod_{i = 1}^t 
(\sigma_A(\bfz^{(1)}_i)-p)
\prod_{i = 1}^t 
(\sigma_A(\bfg^{(2)}(\bfz^{(2)}_i)^{-1})-p)
\prod_{i = 1}^t 
(\sigma_A(\bfz^{(2)}_i)-p)
\Bigg]\Bigg|\\
&\qquad \qquad \le {|\Group|^{-c_2}},
    \end{align*}
where the last inequality follows directly from \cref{def:smalllscaleindependence} applied to $\bfx = \bfg^{(1)}(\bfz^{(1)}_1)^{-1}$ and the fact that 
$$
\Bigg\|
\prod_{i = 2}^t 
(\sigma_A(\bfg^{(1)}(\bfz^{(1)}_i)^{-1})-p)
\prod_{i = 1}^t 
(\sigma_A(\bfz^{(1)}_i)-p)
\prod_{i = 1}^t 
(\sigma_A(\bfg^{(2)}(\bfz^{(2)}_i)^{-1})-p)
\prod_{i = 1}^t 
(\sigma_A(\bfz^{(2)}_i)-p)
\Bigg\|<1
$$
holds
with probability $1$ over $A.$ Using these bounds in \cref{eq:ragindisteq49}, we finish the proof analogously.
\end{proof}

\section{Signed Four-Cycles on Probabilistic Latent Space Graphs}
\label{appendix:signedfourcyclesrag}

\begin{proof}[Proof of \cref{signedfourcyclesforrag}]
We need to bound $\expect[\tau_4(H)], \Var[{\tau_4(H)}]$ for $H \sim \ergraph$ and\linebreak
$H \sim \mathsf{PLSG}(n , \Omega,\unif, p, \sigma).$ We begin with $\expect_{H \sim \mathsf{PLSG}(n , \Omega,\unif, p, \sigma)}[\tau_4(H)].$ Since $\sigma$ is an indicator with positive expectation $p,$ there exists some $\bfy_1, \bfy_2\in \Omega$ such that $\sigma(\bfy_1, \bfy_2) = 1.$
In particular, this means that $p \ge {|\Omega|^{-2}}.$ Consider the signed cycle $(1,2,3,4).$
\begin{align*}
    &\expect_{H \sim \mathsf{PLSG}(n , \Omega,\unif, p, \sigma)}[(H_{1,2} - p)(H_{2,3} - p)
    (H_{3,4} - p)
    (H_{4,1} - p)]\\
        & =\expect[(\sigma(\bfx_1,\bfx_2) - p)
        (\sigma(\bfx_2,\bfx_3) - p)
        (\sigma(\bfx_3, \bfx_4) - p)
        (\sigma(\bfx_4,\bfx_1) - p)]\\
    & = \expect_{\bfx_1,\bfx_3}\bigg[\expect_{\bfx_2}[(\sigma(\bfx_1, \bfx_2) - p)
        (\sigma(\bfx_3, \bfx_2) - p)| \bfx_1, \bfx_3]
        \expect_{\bfx_4}[
        (\sigma(\bfx_1, \bfx_4) - p)
        (\sigma(\bfx_3,\bfx_4) - p)| \bfx_1, \bfx_3]\bigg]\\
    & = \expect_{\bfx_1, \bfx_3}\bigg[\expect_{\bfx_2}[(\sigma(\bfx_1, \bfx_2) - p)
        (\sigma(\bfx_3 ,\bfx_2) - p)| \bfx_1, \bfx_3]^2\bigg]\\
    & \ge \bigg[\expect_{\bfx_2}[(\sigma(\bfy_1, \bfx_2) - p)
        (\sigma(\bfy_1, \bfx_2) - p)]\bigg]^2\times 
        \prob[\bfx_1= \bfx_3 = \bfy_1]
        \\
    & \ge \frac{1}{|\Omega|^2}
    \bigg[
\expect_{\bfx_2}[(\sigma(\bfx_1,\bfx_2)-p)^2]^2
    \bigg]\\
& \ge \frac{1}{|\Omega|^2}
[(1-p)^2\prob[\bfx_2 = \bfy_2]]^2\\
    & = \frac{1}{|\Omega|^4}(1-p)^4.
    \end{align*}
In particular, this means that 
\begin{align*}
        \expect_{H \sim \mathsf{PLSG}(n , \Omega,\unif, p, \sigma)}[\tau_4(H)] = \Omega(n^4) |\Group|^{-3}(1-p)^4 = \Omega(n^4|\Omega|^{-4}).
    \end{align*}
Now, we will prove that
$$
\Var_{H \sim \mathsf{PLSG}(n , \Omega,\unif, p, \sigma)}[\tau_4(H)] = O(n^7).
$$
For brevity, denote 
$$\tau_{(i_1, i_2, i_3, i_4)}(H) = 
(H_{i_1, i_2} - p)(H_{i_2, i_3} - p)
(H_{i_3, i_4} - p)(H_{i_4, i_1} - p).
$$
Under thies notation, clearly
$
\tau_4(H) = \sum_{i_1, i_2, i_3, i_4}
\tau_{(i_1, i_2, i_3, i_4)}(H).
$
Now,
\begin{align*}
& \Var_{H \sim \mathsf{PLSG}(n , \Omega,\unif, p, \sigma)}[\tau_4(H)]\\
& =
\Var_{H \sim \mathsf{PLSG}(n , \Omega,\unif, p, \sigma)}\bigg[\sum_{(i_1, i_2, i_3, i_4)}
\tau_{(i_1, i_2, i_3, i_4)}(H)
\bigg]\\
& = \sum_{(i_1, i_2, i_3, i_4)}\Var_{H \sim \mathsf{PLSG}(n , \Omega,\unif, p, \sigma)}[\tau_{(i_1, i_2, i_3, i_4)}(H)]\\
& + 
\sum_{(i_1, i_2, i_3, i_4)\neq (j_1, j_2, j_3, j_4)}\Cov_{H \sim \mathsf{PLSG}(n , \Omega,\unif, p, \sigma)}[\tau_{(i_1, i_2, i_3, i_4)}(H), 
\tau_{(j_1, j_2, j_3, j_4)}(H)
]. 
    \end{align*}
The key observation is that if $(i_1, i_2, i_3, i_4)$ and $(j_1, j_2, j_3, j_4)$ are disjoint, the covariance is zero since the two functions depend on disjoint set of latent vectors. Otherwise, the covariance is at most 1 since $\tau$ only takes values in $[0,1].$ However, there are $O(n^7)$ ways to take two four-tuples which are not disjoint among the elements of $[n],$ so the claim follows.
Finally, note that one similarly has 
$$
\expect_{H\sim \ergraph}[\tau_4(H)] = 0, 
\Var_{H \sim \ergraph}[\tau_4(H)] = O(n^7).
$$
Thus, a sufficient condition for the signed 4-cycle statistic to differentiate between the two graph distributions with high probability is that 
$
n^4|\Omega|^{-4} = \omega(\sqrt{n^7}),
$
which is guaranteed by $|\Omega| = o(n^{1/8}).$
\end{proof}

\end{document}